\documentclass[reqno,11pt]{amsart}
\usepackage[T1]{fontenc}
\usepackage[en-GB]{datetime2}
\usepackage{graphicx}
\usepackage{amssymb}
\usepackage{caption}
\usepackage{enumitem}
\setlist[itemize]{leftmargin=2.5em}
\usepackage{comment}
\usepackage{thmtools}
\usepackage{setspace}
\usepackage{tikz}
\usetikzlibrary{matrix,intersections,calc}
\tikzset{ 
	table/.style={
		matrix of nodes,
		nodes={rectangle,text width=1.75em,align=center},
		text depth=1.25ex,
		text height=2.5ex,
		nodes in empty cells
	}
}
\usepackage[sort]{natbib}

\usepackage[pdftex
,pagebackref
]{hyperref}
\hypersetup{
	colorlinks=true,
	allcolors=blue,
}
\usepackage[capitalise]{cleveref}
\usepackage{geometry}
\newtheorem{theorem}{Theorem}[section]
\newtheorem{lemma}[theorem]{Lemma}
\newtheorem{claim}{Claim}[theorem]
\Crefname{claim}{Claim}{Claims}

\Crefname{subsection}{Subsection}{Subsections}
\newtheorem{conjecture}[theorem]{Conjecture}
\Crefname{conjecture}{Conjecture}{Conjectures}
\newtheorem{problem}[theorem]{Problem}
\Crefname{problem}{Problem}{Problems}
\newtheorem{proposition}[theorem]{Proposition}
\allowdisplaybreaks

\usepackage{mathtools}

\newcommand{\blackqed}{%
	\leavevmode
	\hbox to .77777em{\hfil
		\vrule width .6em height .6em depth 0em
		\hfil}%
}

\expandafter\let\expandafter\oldproof\csname\string\proof\endcsname
\let\oldendproof\endproof
\renewenvironment{proof}[1][\proofname]{%
	\oldproof[\normalfont\bfseries #1]%
}{\oldendproof}
\newenvironment{subproof}[1][\normalfont\it Subproof]{%
	\begin{proof}[#1]%
	}{%
	\end{proof}%
}

\DeclarePairedDelimiter\abs{\lvert}{\rvert}%
\DeclarePairedDelimiter\ceil{\lceil}{\rceil}%
\DeclarePairedDelimiter\floor{\lfloor}{\rfloor}%
\newcommand{\erh}{Erd\H{o}s--Hajnal}
\newcommand{\gas}{Gy\'{a}rf\'{a}s--Sumner}

\makeatletter
\newcommand{\leqnomode}{\tagsleft@true}
\newcommand{\reqnomode}{\tagsleft@false}
\makeatother

\AtBeginDocument{
	\addtocontents{toc}{\Small}
}

\usepackage{stix2}

\newcommand\mac[1]{\text{\usefont{OMS}{cmsy}{m}{n}#1}} 
\newcommand{\mab}{\mathbb}

\newcommand{\chis}{\chi^*}

\newcommand{\eps}{\varepsilon}
\newcommand{\nin}{\notin}
\renewcommand{\subset}{\subseteq}

\DeclareSymbolFont{mdcharterletters}{OML}{mdbch}{m}{it}
\SetSymbolFont{mdcharterletters}{bold}{OML}{mdbch}{b}{it}
\DeclareMathSymbol{\ell}{\mathord}{mdcharterletters}{"60}  

\setbox0=\hbox{$x$}
\makeatletter
\DeclareSymbolFont{timesletters}{OML}{ztmcm}{m}{it}

\newcommand*\greekrange[3]{%
	\@tempcnta=#2\relax
	\@for\g:=#3\do{%
		\expandafter\DeclareMathSymbol\csname\g\endcsname
		{\mathalpha}{#1}{\the\@tempcnta}%
		\advance\@tempcnta\@ne}}
\makeatother

\greekrange{timesletters}{11}{alpha,beta,gamma,delta,epsilon,zeta,eta,theta,iota,kappa,lambda,mu,nu,xi,pi,rho,sigma,tau,upsilon,phi,chi,psi,omega}
\greekrange{timesletters}{34}{varepsilon,vartheta,varpi,varrho,varsigma,varphi}

\newcommand{\compactsubs}[1]{%
	\fontdimen16\textfont2=#1\fontdimen16\textfont2
	\fontdimen17\textfont2=#1\fontdimen17\textfont2
	\fontdimen16\scriptfont2=#1\fontdimen16\scriptfont2
	\fontdimen17\scriptfont2=#1\fontdimen17\scriptfont2
	\fontdimen16\scriptscriptfont2=#1\fontdimen16\scriptscriptfont2
	\fontdimen17\scriptscriptfont2=#1\fontdimen17\scriptscriptfont2
}

\compactsubs{.8}
\begin{document}
	\title{
		Polynomial $\chi$-boundedness for excluding $P_5$ 
	}
	\author{Tung H. Nguyen}
	
	\address{Mathematical Institute and Christ Church, University of Oxford, Oxford, UK}
	\email{\href{mailto:nguyent@maths.ox.ac.uk}
		{nguyent@maths.ox.ac.uk}
	}
	\thanks{A preliminary version of this paper was accepted to the 67th Annual IEEE Symposium on Foundations of Computer Science (FOCS 2026).}
	
	\begin{abstract}
		Resolving a 1985 open problem of Gy\'arf\'as, we prove that chromatic number is polynomially bounded by clique number for graphs with no induced five-vertex path $P_5$. Our approach introduces a chromatic density framework involving chromatic quasirandomness and chromatic density increment, which allows us to deduce the desired statement from the Erd\H os--Hajnal result for $P_5$.
	\end{abstract}
	
	\maketitle
	\tableofcontents
	\clearpage
	
	\section{Introduction}
	\subsection{History and the main result}
	\label{subsec:main}
	All graphs in this paper are finite and simple. 
	For a graph~$G$, its {\em chromatic number} $\chi(G)$ is the least integer $k\ge0$ for which one can colour its vertices with $k$ colours such that every two adjacent vertices get different colours,
	and its {\em clique number} $\omega(G)$ is the maximum number of pairwise adjacent vertices in $G$. Clearly $\chi(G)\ge\omega(G)$ always. The converse is far from the truth, however, because there are various classical constructions of triangle-free graphs with arbitrarily large chromatic number~\cite{des47,des54,MR69494,MR35428,MR102081}.
	A fundamental challenge in graph theory is to understand the local properties of such graphs, and more generally, of graphs with chromatic number much larger than clique number.
	
	To formalise this idea, we say that a graph class $\mac G$ is {\em $\chi$-bounded} if there exists a function $f\colon\mab N\to\mab R_{\ge0}$ satisfying $\chi(G)\le f(\omega(G))$ for every $G\in\mac G$; and such an $f$ is called a {\em $\chi$-binding} function for~$\mac G$.
	Introduced by Gy\'arf\'as~\cite{MR951359} in a seminal paper from~1985 to generalise the celebrated notion of perfect graphs (which are $\chi$-bounded with the identity $\chi$-binding~function), the concept of $\chi$-boundedness encodes algorithmic, extremal, and structural aspects of graph theory; see~\cite{scott2022,MR4174126} for surveys. The area of $\chi$-boundedness is connected to a number of disciplines such as geometry~\cite{MR3171778}, topology~\cite{MR4039603,MR514625}, number theory~\cite{MR963118,MR939574}, probability~\cite{MR4841090}, combinatorial optimisation~\cite{MR4232071}, coding theory~\cite{MR2815830}, approximation algorithms~\cite{MR4708895}, and quantum information theory~\cite{prx,prl14}.

	In $\chi$-boundedness, local properties are usually studied via the notion of induced subgraph.
	Here, a graph $H$ is an {\em induced subgraph} of $G$ if it can be obtained from $G$ by removing vertices; and we say that $G$ is {\em $H$-free} if it has no induced subgraph isomorphic to $H$.
	In what follows, a {\em forest} is a graph with no induced cycle.
	The following conjecture made independently by Gy\'arf\'as~\cite{MR382051} and Sumner~\cite{MR634555} is a remarkable instance of local-global phenomena in graph theory, and is arguably the foremost open problem of $\chi$-boundedness:
	
	\begin{conjecture}
		[\gas]
		\label{conj:gs}
		For every forest $T$, the class of $T$-free graphs is $\chi$-bounded.
	\end{conjecture}
	
	This conjecture remains largely open despite intensive effort for five decades. If true, it would be a necessary and sufficient characterisation of forests; indeed, the conjecture would be false if the forbidden graph $T$ contains a cycle, because Erd\H os~\cite{MR102081} constructed high-chromatic graphs with no short cycle.
	
	A central focus in $\chi$-boundedness concerns the growth rate of the optimal $\chi$-binding function of $\chi$-bounded classes.
	Here, for such a class $\mac G$, its {\em optimal} $\chi$-binding function is the function $f_{\mac G}\colon\mab N\to\mab N$ defined by $f_{\mac G}(n):=\max(\chi(G):G\in\mac G,\,\omega(G)=n)$ for all $n\in\mab N$.
	We say that $\mac G$ is {\em polynomially $\chi$-bounded} if $f_{\mac G}$ has polynomial growth, or equivalently if $\mac G$ admits a polynomial $\chi$-binding function.
	Bria\'nski, Davies, and Walczak~\cite{MR4707561} constructed, for each $f\colon\mab N\to\mab N$ with $f(n)\ge\binom{3n+1}{3}$ for all $n\in\mab N$, a $\chi$-bounded class $\mac G$ with $f_{\mac G}=f$ (see~\cite{MR4707564,MR4484828} for related results), thereby disproving in a very strong form a conjecture of Esperet~\cite{esperet} that $\chi$-boundedness implies its polynomial strengthening.
	Hence, there exist $\chi$-bounded classes that are~far from being polynomially $\chi$-bounded, which show that the latter property indicates exceptionally good control over chromatic number.
	Thus, perhaps many known $\chi$-bounded classes are not polynomially $\chi$-bounded; and it may not be surprising that the proofs of most $\chi$-boundedness results so far do not yield their polynomial refinements in turn.
	These include the majority of confirmed special cases of the \gas{} conjecture~\ref{conj:gs}, which have only been verified with super-exponential $\chi$-binding functions~\cite{MR1437291,MR1258244,MR2085214,MR4014345,MR4009302}.
	Nonetheless, the constructions from~\cite{MR4707561} contain arbitrarily large induced trees; and so the following tremendous strengthening of \cref{conj:gs} could still be true:
	
	\begin{conjecture}
		[Polynomial \gas]
		\label{conj:pgs}
		For every forest $T$, the class of $T$-free graphs is polynomially $\chi$-bounded.
	\end{conjecture}
	
	(The case when $T$ is a path was first asked by Esperet, Lemoine, Maffray, and Morel~\cite{MR3010736}.)

	For every integer $t\ge1$, let $P_t$ be the $t$-vertex path.
	Our focus in this paper is \cref{conj:pgs} for $T=P_5$, the smallest open case of this conjecture until now (see~\cite[Chapter 13]{2025thes} and~\cite{MR4472775} for previously confirmed cases).
	This special case dates back to Gy\'arf\'as's 1985 paper introducing $\chi$-boundedness~\cite[Problem 2.7]{MR951359} where he obtained the first $\chi$-binding function $x\mapsto 4^{x-1}$ for $P_5$-free graphs and asked to determine the order of magnitude\footnote{Andr\'as Gy\'arf\'as (private communication) confirmed with us that in~\cite[Problem 2.7]{MR951359}, `order of magnitude' means whether the optimal $\chi$-binding function for $P_5$-free graphs is polynomial/quasi-polynomial/exponential\dots.} of their optimal $\chi$-binding function.
	Since then, the polynomial $\chi$-boundedness problem for excluding $P_5$ has attracted a great deal of unsuccessful attention, and has been highlighted by various groups of authors (see~\cite{MR4951164,MR4648583,MR4840998,MR3898374,MR4672178,scott2022,MR4708898,MR4332745,geisser,MR3010736,MR4174126,MR5025835,MR4033105,MR1361383,MR4174128} and the references therein) as the outstanding test case for \cref{conj:pgs} and a major open problem in $\chi$-boundedness.
	It is also highly relevant from an algorithmic perspective (at least with respect to \cref{conj:pgs} for excluding an arbitrary induced path), because the class of $P_5$-free graphs represents the critical~tractability boundary for colouring graphs with a forbidden induced path.
	To explain, we consider the {\sc $k$-colourability} problem for $k\in\mab N$, which decides whether an input graph $G$ has chromatic number at most $k$. If $k$ is part of the input, then this can be done in polynomial time when $G$ is $P_4$-free, but becomes {\sf NP}-complete when $G$ is $P_5$-free~\cite{MR1905637}; and at the other extreme, for every fixed $k$, the running time is polynomial if $G$ is $P_5$-free~\cite{MR2581077}, but returns to {\sf NP}-complete already with $k=5$ colours if $G$ is $P_6$-free~\cite{MR3398861}. 
	
	Over the past 40 years, there have been two main lines of research towards proving that $P_5$-free graphs are polynomially $\chi$-bounded. The first line attempts to provide successively stronger $\chi$-binding functions for these graphs:
	\begin{itemize}
		\item $x\mapsto 4^{x-1}$, by Gy\'arf\'as~\cite{MR951359} (more generally, $x\mapsto (t-1)^{x-1}$ for $P_t$-free graphs for each $t\ge2$);
		
		\item $x\mapsto 3^{x-1}$, by Gravier, Ho\`ang, and Maffray~\cite{MR2009549} (more generally, $x\mapsto (t-2)^{x-1}$ for $P_t$-free graphs for each $t\ge2$);
		
		\item $x\mapsto \max(3,5\cdot3^{x-3})$, by Esperet, Lemoine, Maffray, and Morel~\cite{MR3010736};
		
		\item $x\mapsto 2^x$, claimed without proof by Kierstead, Penrice, and Trotter~\cite{MR1361383};
		
		\item $x\mapsto x^{\log x}$, by Scott, Seymour, and Spirkl~\cite{MR4648583} (here $\log$ denotes the binary logarithm); and
		
		\item $x\mapsto x^{O(\log x/\log\log x)}$, by the author~\cite[Chapter 14]{2025thes}. 
	\end{itemize}
	
	The second line of research strives to confirm polynomial $\chi$-boundedness for various $\{P_5,H\}$-free (meaning both $P_5$-free and $H$-free) subclasses for some non-complete $P_5$-free graph $H$ (see the two recent PhD theses~\cite{MR4966898,geisser} and the surveys~\cite{MR3898374,MR4840998} for results in this direction). In the case of five-vertex $H$, perhaps the most well-known verified instances are:
	\begin{itemize}
		\item when $H$ is the complement $\overline{P_5}$ of $P_5$, first proved with $\chi$-binding function $x\mapsto \binom{x+1}{2}$ by Fouquet, Giakoumakis, Maire, and Thuillier~\cite{MR1360104} who adapted a structural result of Fouquet~\cite{MR1246159} that every $\{P_5,\overline{P_5}\}$-free graph is either the five-cycle $C_5$, or perfect, or obtained by vertex-substitution from two smaller $\{P_5,\overline{P_5}\}$-graphs (see~\cite{MR3601318} for a refinement of this result); and
		
		\item when $H$ is the {\em bull} (obtained from $P_5$ by adding an edge joining the two non-adjacent vertices of degree two), first proved with $\chi$-binding function $x\mapsto \binom{x+1}{2}$ by Chudnovsky and Sivaraman~\cite{MR3879962} who showed that every $\{P_5,\text{bull}\}$-free graph has a vertex-partition into two induced subgraphs where one has smaller clique number and the other is perfect.
	\end{itemize}
	
	However, despite much effort, it remained open whether the class of $\{P_5,C_5\}$-free graphs is polynomially $\chi$-bounded, as was emphasised in~\cite{MR3898374,MR4840998,geisser}.
	
	Confirming \cref{conj:pgs} for $T=P_5$, this paper provides the first polynomial $\chi$-binding function for $P_5$-free graphs in four decades, as follows.
	
	\begin{theorem}
		\label{thm:p5}
		There exists $d\ge2$ such that every $P_5$-free graph $G$ satisfies $\chi(G)\le\omega(G)^{d}$.
	\end{theorem}
	
	Our proof of this result introduces a novel `chromatic density' framework on $P_5$-free graphs and examines how the $P_5$-free constraint influences the chromatic number of their induced subgraphs. To the best of our knowledge, the key ideas coming from this framework have not been considered in $\chi$-boundedness prior to this work; we will discuss them in \cref{sec:sketch}. 
	
	\subsection{Connections to hereditary Ramsey properties}
	In what follows, for a graph $G$, we use $\abs G$ and $\alpha(G)$ to denote its number of vertices and its stability number (the maximum number of pairwise nonadjacent vertices in $G$), respectively.
	A graph class $\mac G$ is {\em hereditary} if it is closed under isomorphism and taking induced subgraphs, and is {\em proper} if there exists a graph $H$ not in $\mac G$.
	The notion of polynomial $\chi$-boundedness and the polynomial \gas{} conjecture~\ref{conj:pgs} are of particular interest because of the following celebrated conjecture of Erd\H{o}s and Hajnal from 1977~\cite{MR1031262,MR599767} on diagonal Ramsey numbers under structural constraints (see~\cite{MR1425208,MR3150572} for surveys):
	
	\begin{conjecture}
		[\erh]
		\label{conj:eh}
		For every proper hereditary class $\mac G$, there exists $c>0$ such that $\max(\alpha(G),\omega(G))\ge\abs G^c$ for all $G\in \mac G$.
	\end{conjecture} 
	
	To explain, observe that if a hereditary class $\mac G$ is $\chi$-bounded with $\chi$-binding function $x\mapsto x^d$ for some $d\ge1$ then it satisfies \cref{conj:eh} for $c=\frac1{d+1}$: indeed, for every $G\in\mac G$, the simple inequality $\alpha(G)\chi(G)\ge\abs G$ implies that
	\begin{equation}
		\label{eq:pchieh}
		\omega(G)^d\ge\chi(G)\quad\Rightarrow\quad
		\alpha(G)\omega(G)^d\ge\abs G\quad\Rightarrow\quad
		\max(\alpha(G),\omega(G))\ge\abs G^{\frac1{d+1}}.
	\end{equation}
	
	We remark that the middle inequality above displays a substantially stronger Ramsey property than the prediction of \cref{conj:eh}.
	Indeed, while \cref{conj:eh} could be true for {\em all} proper hereditary classes, if a graph $T$ satisfies $\alpha(G)\omega(G)^d\ge\abs G$ for all $T$-free graphs $G$ then $T$ has to be a forest. This is because the aforementioned probabilistic construction of Erd\H os~\cite{MR102081} actually shows that for every $g\ge 3$ and every sufficiently large $n$, there are $n$-vertex graphs with no cycle of length at most $g$ (hence triangle-free) and with stability number at most $n^{1-1/g+o(1)}$ (and so high-chromatic; see~\cite{MR963118,MR939574} for explicit constructions of this type). Hence, the following approximation of the \gas{} conjecture \ref{conj:gs} proved by Scott, Seymour, and the author~\cite{stable1} can be viewed as `if and only if' for characterising forests: for every $\eps>0$ and every forest $T$, there exists $\delta>0$ such that every $T$-free graph $G$ with $\omega(G)\le\delta\log\log\abs G$ satisfies $\alpha(G)\ge\abs G^{1-\eps}$.
	On a side note, it remains open whether for all forests $T$, every $T$-free graph $G$ with bounded clique number has stability number linear in $\abs G$, which is a consequence of \cref{conj:gs}.
	
	At the other extreme, there are known examples of hereditary classes $\mac G$ that satisfy the middle inequality of \eqref{eq:pchieh} but not the left-hand side one; perhaps the most well-known such example is the hereditary closure of the triangle-free Kneser graphs $\{\operatorname{KG}(3n-1,n):n\in\mab N\}$ where each $\operatorname{KG}(3n-1,n)$ has chromatic number $n+1$ (due to Lov\'asz~\cite{MR514625}) and fractional chromatic number less than $3$.
	On the other hand, if \cref{conj:pgs} is true then all inequalities in \eqref{eq:pchieh} would hold for classes $\mac G$ defined by forbidding a single forest.
	In what follows, 
	let us say that a forest $T$ is:
	\begin{itemize}
		\item {\em poly-$\chi$-bounding} if there exists $d\ge1$ such that $\chi(G)\le\omega(G)^d$ for all $T$-free graphs $G$ (so $T$ satisfies \cref{conj:pgs}); and
		
		\item {\em \erh{}} if there exists $c>0$ such that $\max(\alpha(G),\omega(G))\ge\abs G^c$ for all $T$-free graphs $G$ (so the class of $T$-free graphs satisfies \cref{conj:eh}).
	\end{itemize}
	
	To show that a forest $T$ is poly-$\chi$-bounding, a possible pathway would be first to prove that $T$-free graphs satisfy \cref{conj:gs,conj:eh}.
	One could then build on the \erh{} property of $T$ and the $\chi$-boundedness proof for $T$-free graphs to eventually prove its poly-$\chi$-bounding property.
	The case $T=P_5$ is an example of this approach.
	Indeed, the \erh{} property of $P_5$ was recently verified by Scott, Seymour, and the author~\cite{density7}, and will be employed to prove \cref{thm:p5} in this paper.
	The proof of \cref{thm:p5} also adapts the `Gy\'arf\'as path' argument~\cite{MR951359} that gave the first $\chi$-binding function $x\mapsto 4^{x-1}$ for $P_5$-free graphs.
	It would be interesting to extend this approach to settle \cref{conj:pgs} for $T$ being an arbitrary path. In this regard, it remains open whether $P_t$ has the \erh{} property for all $t\ge 6$; see~\cite{density5} for the best known partial result.

	\subsection*{Acknowledgements}
	We would like to thank Andr\'as Gy\'arf\'as for explaining the motivation of \cite[Problem 2.7]{MR951359} to us.
	We are grateful to Maria Chudnovsky, Alex Scott, and Paul Seymour for insightful discussions and support throughout various stages of this project, and to Sang-il Oum for helpful comments on presentation.
	We also acknowledge the anonymous FOCS reviewers for valuable remarks and suggestions on a preliminary version of this paper.
	
	We acknowledge the use of ChatGPT in proofreading this paper. All the mathematical ideas and final proofs presented are entirely due to the author.
	
	Part of this work was conducted while the author was at Princeton University and was
	supported by AFOSR grant FA9550-22-1-0234, NSF grant DMS-2154169, and a Porter Ogden Jacobus Fellowship.
	The author is currently supported by a Titchmarsh Research Fellowship and a Christ Church Research Centre Grant.
	
	For the purpose of open access, the author has applied a CC BY public copyright licence to
	any author accepted manuscript arising from this submission.

	\section{Preliminaries and proof overview}
	\label{sec:sketch}
	For an integer $k\ge1$, let $[k]:=\{1,2,\ldots,k\}$.
	For a graph $G$ with vertex set $V(G)$, a {\em clique} of $G$ is a set of pairwise adjacent vertices in $G$; and a {\em stable set} of $G$ is a set of pairwise nonadjacent vertices in $G$.
	For $S\subset V(G)$, let $G[S]$ be the graph on vertex set $S$ whose edges are all the edges $uv$ of $G$ with $u,v\in S$, and let $G\setminus S:=G[V(G)\setminus S]$; we also write $\chi(S)$ for $\chi(G[S])$ where there is no danger of ambiguity.
	A pair $(A,B)$ of disjoint subsets of $V(G)$ is:
	\begin{itemize}
		\item {\em complete} in $G$ if every vertex in $A$ is adjacent in $G$ to every vertex in $B$,
		
		\item {\em anticomplete} in $G$ if $G$ has no edge between $A,B$, and
		
		\item {\em pure} in $G$ if it is complete or anticomplete in $G$.
	\end{itemize}
	Also, say that $A$ is {\em complete} to $B$ in $G$ in the first case, and {\em anticomplete} to $B$ in $G$ in the second case.
	A {\em blockade} in a graph $G$ is a sequence $(B_1,\ldots,B_k)$ of disjoint subsets of $V(G)$; its {\em blocks} are $B_1,\ldots,B_k$, its {\em length} is $k$, and its {\em mass} is $\min(\chi(B_i):i\in[k])$.
	We emphasise that our argument extensively studies length and mass of blockades.
	We also say that the blockade $(B_1,\ldots,B_k)$ is {\em complete} in $G$ if $(B_i,B_j)$ is complete in $G$ for all distinct $i,j\in[k]$.
	\cref{thm:p5} is proved via the following result.
	
	\begin{theorem}
		\label{thm:main}
		There exists $d\ge 2$ such that every $P_5$-free graph $G$ with $\chi(G)\ge2$ contains either:
		\begin{itemize}
			\item a complete pair $(X,Y)$ with $\chi(X)\ge y^{d}\chi(G)$ and $\chi(Y)\ge (1-y)\chi(G)$ for some $y\in(0,\frac14)$; or
			
			\item a complete blockade of length $k$ and mass at least $k^{-d}\chi(G)$, for some integer $k\ge2$.
		\end{itemize}
	\end{theorem}
	
	In other words, this result says that every $P_5$-free graph with at least one edge contains an `appropriately unbalanced' complete bipartite subgraph or a `polynomially balanced' complete multipartite subgraph.
	On one hand, \cref{thm:p5} evidently implies \cref{thm:main}, because if $\chi(G)\le\omega(G)^d$ then the second outcome of \cref{thm:main} holds with $k=\omega(G)$ by taking a maximum clique of $G$.
	On the other hand, let us now deduce \cref{thm:p5} from \cref{thm:main}.
	
	\begin{proof}
		[Proof of \cref{thm:p5}, assuming \cref{thm:main}]
		We claim that $d\ge2$ given by \cref{thm:main} suffices.
		To see this, we proceed by induction on $\omega(G)$; and we may assume that $\omega(G)\ge2$ and so $\chi(G)\ge 2$.
		By \cref{thm:main}, $G$ contains either:
		\begin{itemize}
			\item a complete pair $(X,Y)$ with $\chi(X)\ge y^{d}\chi(G)$ and $\chi(Y)\ge(1-y)\chi(G)$ for some $y\in(0,\frac14)$; or
			
			\item a complete blockade $(B_1,\ldots,B_k)$ of mass at least $k^{-d}\chi(G)$, for some integer $k\ge2$.
		\end{itemize}
		
		If the first bullet holds then either $\omega(X)\le y\cdot\omega(G)$ or $\omega(Y)\le(1-y)\cdot\omega(G)$. If the former case holds then $1\le \omega(X)\le y\cdot\omega(G)$ (since $\chi(X)>0$) so by induction $\chi(G)\le y^{-d}\chi(X)\le y^{-d}\omega(X)^{d}\le \omega(G)^{d}$;
		and if the latter case holds then $1\le\omega(Y)\le(1-y)\cdot\omega(G)<\omega(G)$ and so by induction
		\[\chi(G)\le (1-y)^{-1}\chi(Y)\le (1-y)^{-1}\omega(Y)^{d}\le (1-y)^{d-1}\omega(G)^{d}\le \omega(G)^{d}.\]
		
		If the second bullet holds then $2\le k\le \omega(G)$ and there exists $i\in[k]$ with $\omega(B_i)\le \omega(G)/k<\omega(G)$ since $(B_1,\ldots,B_k)$ is complete in $G$; and so by induction 
		$\chi(G)\le k^{d}\chi(B_i)\le k^{d}(\omega(G)/k)^{d}= \omega(G)^{d}$.
		This proves \cref{thm:p5}.
	\end{proof}
	The rest of this paper presents a proof of \cref{thm:main}, which develops a number of novel ideas and substantially extends those introduced in the author's PhD thesis~\cite[Chapter 14]{2025thes} and by Scott, Seymour, and the author in~\cite{density7}.
	The heart of the proof is a `chromatic density' framework for $P_5$-free graphs, which examines the chromatic number of their induced subgraphs under the $P_5$-free constraint.
	This is done via three primary steps and will be summarised in the following three subsections. Along the way, we will give more definitions and outline the paper.
	
	\subsection{Chromatic quasirandomness}
	\label{subsec:chirdl}
	Introduced in seminal works of Thomason~\cite{MR930498,MR905280} and of Chung, Graham, and Wilson~\cite{MR1054011}, quasirandomness refers to a set of distant but surprisingly equivalent deterministic graph properties that are satisfied almost surely by random graphs.
	A key intuition behind this notion is that if the edge distribution of a large host graph is not very close to being `non-random', then the host graph necessarily contains an induced copy (and so many induced copies) of all small graphs. This idea is manifested by the following 1986 theorem of R\"odl~\cite{MR837962}; see~\cite{MR1952989} for a survey on this result.
	
	\begin{theorem}
		[R\"odl]
		\label{thm:rodl}
		For every $\eps\in(0,\frac12)$ and every graph $H$, there exists $\delta>0$ such that every $H$-free graph $G$ has an induced subgraph $F$ with $\abs F\ge \delta\abs G$ such that $F$ has maximum degree at most $\eps\abs F$ or minimum degree at least $(1-\eps)\abs F$.
	\end{theorem}
	
	This theorem has become a fundamental tool in recent progress on the \erh{} conjecture~\ref{conj:eh}.
	The original proof of R\"odl uses Szemer\'edi's regularity lemma~\cite{MR540024} and gives tower-type dependence of $\delta$ on $\eps$.
	The best known general bound is currently $\delta=\eps^{d\log\frac1\eps/\log\log\frac1\eps}$ for some $d\ge1$ depending on $H$ only~\cite{MR4761787}; and the `polynomial R\"odl' conjecture posed by Fox and Sudakov~\cite{MR2455625} says that $\delta$ can be taken to be $\eps^d$, which is now known to be equivalent to \cref{conj:eh}~\cite{bfp2024}.
	
	When $H=T$ is a forest, Chudnovsky, Scott, Seymour, and Spirkl~\cite{MR4170220} proved the following result on linear-sized anticomplete pairs in sparse $T$-free graphs, which in turn implies (via \cref{thm:rodl}) that the class of graphs excluding $T$ and the complement of $T$ satisfies \cref{conj:eh}:
	
	\begin{theorem}
		[Chudnovsky--Scott--Seymour--Spirkl]
		\label{thm:pp1}
		For every forest $T$, there exists $\delta>0$ such that every $T$-free graph $G$ with $\abs G\ge2$ and maximum degree at most $\delta\abs G$ contains an anticomplete pair $(A,B)$ with $\abs A,\abs B\ge \delta\abs G$.
	\end{theorem}
	
	(Bousquet, Lagoutte, and Thomass\'e~\cite{MR3343757} previously provided a short proof when $T$ is an arbitrary path, based on the `Gy\'arf\'as path' argument~\cite{MR951359}.) Applying \cref{thm:rodl,thm:pp1} gives:
	
	\begin{theorem}
		\label{thm:antidense}
		For every $\eps\in(0,\frac12)$ and every forest $T$, there exists $\delta>0$ such that every $T$-free graph $G$ with $\abs G\ge\eps^{-1}$ contains either:
		\begin{itemize}
			\item an anticomplete pair $(A,B)$ with $\abs A,\abs B\ge\delta\abs G$; or
			
			\item an induced subgraph $F$ with $\abs F\ge\delta\abs G$ and minimum degree at least $(1-\eps)\abs F$.
		\end{itemize}
	\end{theorem}
	
	It is not hard to deduce \cref{thm:rodl} from this result when $H=T$ is a forest; to see this, one can iterate \cref{thm:antidense} to obtain a sufficiently long anticomplete blockade whose blocks still have linear size. Hence, for graphs with a forbidden induced forest, \cref{thm:rodl,thm:antidense} are essentially equivalent.
	
	Our first ingredient of the proof of \cref{thm:main} is a `chromatic quasirandomness' analogue of \cref{thm:rodl,thm:antidense} for $P_5$-free graphs.
	In order to do so, we require some definitions.
	For a graph $G$ with $v\in V(G)$, let $N_G(v)$ be the neighbourhood of $v$ in $G$, and let $N_G[v]:=N_G(v)\cup\{v\}$.
	For $\eps>0$, we say that $G$ is {\em $(\eps,\chi)$-dense} if $\chi(G\setminus N_G[v])<\eps\cdot\chi(G)$ for all $v\in V(G)$; in other words, $G$ is $(\eps,\chi)$-dense if for every $v\in V(G)$, the set of nonneighbours of $v$ in $G$ has chromatic number less than $\eps\cdot\chi(G)$.
	We emphasise that $(\eps,\chi)$-dense graphs are the central objects in this paper.
	As far as we know, these graphs were first introduced and studied in the author's PhD thesis~\cite[Chapter 14]{2025thes} (under the term {\em $\eps$-colourful graphs}).
	One can view $(\eps,\chi)$-dense graphs as chromatic analogues of graphs $G$ with minimum degree at least $(1-\eps)\abs G$, and can interpret $\eps$ as a `chromatic density' parameter of $(\eps,\chi)$-dense graphs.
	However $\chi$-dense graphs can have low minimum degree, and graphs with large minimum degree can be far from being $\chi$-dense; see \cref{prop:dense}.
	
	Here is a chromatic version of \cref{thm:antidense}, which essentially says that the `chromatic distribution' of every $P_5$-free graph is very far from being `uniform'.
	
	\begin{lemma}
		\label{lem:rdlchi}
		For every $\eps\in(0,\frac12)$, there exists $\delta>0$ such that every $P_5$-free graph $G$ contains~either:
		\begin{itemize}
			\item an anticomplete pair $(A,B)$ with $\chi(A),\chi(B)\ge \delta\cdot\chi(G)$; or
			
			\item an $(\eps,\chi)$-dense induced subgraph $F$ with $\chi(F)\ge\delta\cdot\chi(G)$.
		\end{itemize}
	\end{lemma}
	
	We remark that the linear-$\chi$ anticomplete pair outcome as above is related to a conjecture of El-Zahar and Erd\H os~\cite{MR845138} that says graphs with huge chromatic number and bounded clique number contain high-chromatic anticomplete pairs (but not necessarily linear in the chromatic number of the graphs in consideration); see~\cite{MR4676642} for some partial results on this problem.
	Another result of this type, proved by Liebenau, Pilipczuk, Seymour, and Spirkl~\cite[Theorem 1.9]{MR3926277}, asserts that for every {\em caterpillar subdivision} $T$, there exists $b\ge2$ such that every $T$-free graph $G$ contains an anticomplete pair $(A,B)$ with $\chi(A),\chi(B)\ge b^{-\omega(G)}\chi(G)$.
	In~\cite[Chapter 14]{2025thes} a weakening of \cref{lem:rdlchi} was obtained, by allowing for an additional outcome that $G$ contains a complete pair $(X,Y)$ with $\chi(X)\ge\omega(G)^{-d}\chi(G)$ and $\chi(Y)\ge\eps^{d}\chi(G)$, for some universal $d\ge2$.
	
	We are actually not going to use \cref{lem:rdlchi} in the proof of \cref{thm:main}; instead, we will employ the following result on linear-$\chi$ pure pairs versus linearly $\chi$-dense induced subgraphs, which is the aforementioned first ingredient as well as an analogue of \cref{thm:rodl,thm:antidense}:
	
	\begin{lemma}
		\label{lem:chirdl}
		Let $\eps\in(0,\frac12)$ and $\delta\in(0,2^{-7}\eps^2]$. Then every $P_5$-free graph $G$ contains either:
		\begin{itemize}
			\item a pure pair $(A,B)$ with $\chi(A),\chi(B)\ge \delta\cdot\chi(G)$; or
			
			\item an $(\eps,\chi)$-dense induced subgraph $F$ with $\chi(F)\ge \delta\cdot\chi(G)$.
		\end{itemize}
	\end{lemma}
	
	Qualitatively, this lemma appears weaker than \cref{lem:rdlchi} because the first outcome includes a complete pair possibility. However this is sufficient for the conclusion of \cref{thm:main} when $\eps$ is a constant; and the polynomial (actually, quadratic) dependence of $\delta$ on $\eps$ will numerically simplify our argument.
	Also, it is not hard to derive \cref{lem:rdlchi} from \cref{lem:chirdl}, as follows.
	\begin{proof}
		[Proof of \cref{lem:rdlchi}, assuming \cref{lem:chirdl}]
		Let $\eta:=2^{-7}\eps^2$ and $a:=1+\floor{\log\frac1\eps}$; we claim that $\delta:=\eta^a$ suffices.
		To this end, suppose that both outcomes of the lemma do not hold.
		Let $k\ge0$ be maximal such that $k\le a$ and $G$ contains a complete blockade $(B_1,\ldots,B_{2^k})$ with $\chi(B_i)\ge \eta^{k}\chi(G)$ for all $i\in[2^k]$; for $k=0$ one can take $B_1:=V(G)$.
		
		We claim that $k=a$.
		Suppose not. Then for each $i\in[2^k]$, by \cref{lem:chirdl} applied to $G[B_i]$, this graph contains either:
		\begin{itemize}
			\item a pure pair $(A_{2i-1},A_{2i})$ with $\chi(A_{2i-1}),\chi(A_{2i})\ge\eta\cdot\chi(B_i)$; or
			
			\item an $(\eps,\chi)$-dense induced subgraph with chromatic number at least $\eta\cdot\chi(B_i)$.
		\end{itemize}
		
		Since $\eta\cdot\chi(B_i)\ge \eta^{k+1}\chi(G)\ge \eta^{a}\chi(G)$ by our supposition and since the second outcome of the lemma fails, the second bullet fails.
		Thus the first bullet holds; and so $(A_{2i-1},A_{2i})$ is a complete pair in $G[B_i]$ because the first outcome of the lemma fails.
		Since this holds for all $i\in[2^k]$, the complete blockade $(A_1,\ldots,A_{2^{k+1}})$ then violates the maximality of $k$.
		
		Hence $k=a$.
		By removing vertices if necessary we may assume that $\chi(B_1)=\ldots=\chi(B_{2^k})\ge\eta^k\chi(G)=\delta\cdot\chi(G)$.
		Let $F:=G[B_1\cup\cdots\cup B_{2^k}]$.
		Since $2^k=2^a>\eps^{-1}$ by the definition of $a$, we have $\chi(F)=\sum_{i\in[2^k]}\chi(B_i)=2^k\chi(B_1)>\eps^{-1}\chi(B_1)$;
		and so $F$ is $(\eps,\chi)$-dense.
		This proves \cref{lem:rdlchi}.
	\end{proof}
	
	It would be interesting to obtain polynomial dependence of $\delta$ on $\eps$ in \cref{lem:rdlchi}, which would unify \cref{lem:rdlchi,lem:chirdl}; but we have not been able to decide this. It would also be interesting to extend these results to graphs with forbidden induced forests, since this would potentially be useful in the study of (polynomial) $\chi$-boundedness. For instance, if a forest $T$ satisfies \cref{lem:rdlchi,lem:chirdl} without the anticomplete pair outcome (we have not been able to decide this for $T=P_5$), then the Gy\'arf\'as--Sumner conjecture \ref{conj:gs} would hold for $T$ (by taking $\eps=\omega(G)^{-2}$); and if this also holds with $\delta=\eps^d$ for some $d\ge1$ depending on $T$ only, then $T$ would satisfy \cref{conj:pgs}.
	This is parallel with the polynomial R\"odl conjecture discussed at the beginning of this subsection.
	
	\subsection{Decomposing along high-$\chi$ anticomplete pairs}
	\label{subsec:highchi}
	
	The core idea of structural graph theory is to decompose graphs with forbidden substructures into simpler pieces.
	In the context of polynomial $\chi$-boundedness, the most dramatic application of this idea is the proof of the strong perfect graph theorem~\cite{MR2233847}; and recently a variety of decomposition techniques have been instrumental in establishing polynomial $\chi$-boundedness for  circle graphs~\cite{MR4275079,MR4494591}, grounded L-graphs~\cite{MR4670367}, even-hole-free graphs~\cite{MR4568110}, bounded clique-width graphs~\cite{MR4125349}, and bounded twin-width graphs~\cite{MR4865478}.
	The common theme of all these proofs is that the `simpler pieces' in question actually belong to (considerably) more structured hereditary classes than those of their corresponding host graphs. These more structured classes are usually known to be polynomially $\chi$-bounded (or even perfect in the cases of~\cite{MR2233847,MR4275079}), and analysing how the decompositions were formulated would lead to the desired polynomial $\chi$-boundedness (or perfectness in the case of~\cite{MR2233847}); see~\cite{MR3096332,MR4033105,MR4865478} for a number of decompositions that preserve this property.
	To the best of our knowledge, this structural approach was also the motivation behind the second line of research towards \cref{thm:p5} that was summarised in \cref{subsec:main}. Indeed, establishing polynomial $\chi$-boundedness for various subclasses of $P_5$-free graphs in this way was expected to suggest a structural decomposition of these graphs, which would lead to the poly-$\chi$-bounding property of $P_5$.
	However, such a decomposition is yet to be found despite considerable effort.
	
	Our second ingredient of the proof of \cref{thm:main} will be an argument that decomposes $P_5$-free graphs along high-$\chi$ anticomplete pairs, first developed by the author in~\cite[Section 4]{2025thes} (under the term `terminal partitions'). Here, our approach is different from that of~\cite{MR2233847,MR4275079,MR4865478,MR4125349,MR4568110,MR4670367} because the `simpler pieces' in our decomposition are not structurally simpler than general $P_5$-free graphs.
	However, the interaction between these pieces and the constraints on their chromatic number will be useful in capturing high-$\chi$ complete pairs or polynomially $\chi$-dense induced subgraphs (all with appropriate parameters), thus reducing \cref{thm:main} to essentially only one case of $\chi$-dense subgraphs.
	Roughly speaking, the decomposition consists of a central cutset together with several connected components of the remainder satisfying certain conditions on chromatic number, which will be described in detail in the proof of \cref{lem:anti} in \cref{sec:highchi}.
	While high-$\chi$ anticomplete pairs might appear ineffective in resolving general $\chi$-boundedness problems (since $\chi$ is invariant under disjoint union), they are useful in building such a decomposition in $P_5$-free graphs because of the following simple fact:
	
	\begin{lemma}
		[folklore]
		\label{lem:mixed}
		Let $G$ be a $P_5$-free graph with nonempty and anticomplete $A,B\subset V(G)$, such that $G[A],G[B]$ are connected.
		Then every vertex in $V(G)\setminus(A\cup B)$ is pure to one of $A,B$.
	\end{lemma}
	Here, a vertex $v\in V(G)$ is {\em pure} to $S\subset V(G)\setminus\{v\}$ if the pair $(\{v\},S)$ is pure in $G$.
	We also say that $v$ is {\em mixed} on $S$ if it has a neighbour and a nonneighbour in $S$.
	\begin{proof}
		[Proof of \cref{lem:mixed}]
		Suppose that there exists $v\in V(G)\setminus(A\cup B)$ mixed on both $A$ and $B$ in $G$.
		Since $G[A]$ is connected, it has an edge $uu'$ such that $v$ is adjacent to $u$ and nonadjacent to $u'$; and since $G[B]$ is connected, it has an edge $zz'$ such that $v$ is adjacent to $z$ and nonadjacent to $z'$.
		Then $u'\text-u\text-v\text-z\text-z'$ would be an induced $P_5$ in $G$, a contradiction. This proves \cref{lem:mixed}.
	\end{proof}
	
	\cref{lem:mixed} implies that in every connected $P_5$-free graph $G$ with two anticomplete vertex subsets $A,B$ such that $G[A],G[B]$ are connected, every minimal nonempty cutset $S$ separating $A,B$ has a partition into two subsets where one is complete to $A$ and the other is complete to $B$.
	Thus, if we also assume that $\chi(A),\chi(B)\ge\delta\cdot\chi(G)$ (for some fixed small $\delta>0$) and $G[A],G[B]$ are the components of $G\setminus S$ with largest $\chi$,
	then either $G$ contains a complete pair with each part having chromatic number at least $\delta\cdot\chi(G)$ (which satisfies the second outcome of \cref{thm:main}) or one of $\chi(A),\chi(B)$ is actually at least $(1-2\delta)\chi(G)$.
	Hence, we have just improved the `linear-$\chi$ versus linear-$\chi$' assumption to `linear-$\chi$ versus ``almost one''-$\chi$'.
	This observation plays a crucial role in the proof of the following main result of \cref{sec:highchi}.
	
	\begin{lemma}
		\label{lem:antino}
		Let $a\ge2$ and $c\in(0,2^{-9}]$. Then every $P_5$-free graph $G$ with $\chi(G)\ge c^{-a}$ contains either:
		\begin{itemize}
			\item a complete pair $(P,Q)$ with $\chi(P),\chi(Q)\ge \frac12c^4\chi(G)$;
			
			\item a complete pair $(X,Y)$ with $\chi(X)\ge y^{5a}\chi(G)$ and $\chi(Y)\ge (1-y)\chi(G)$ for some $y\in(0,2c]$; or
			
			\item an $(\eps,\chi)$-dense induced subgraph $F$ with $\chi(F)\ge \eps^3\chi(G)$ for some $\eps\in[\chi(G)^{-1/a},c]$.
		\end{itemize}
	\end{lemma}
	
	As can be seen, this lemma essentially replaces the linear-$\chi$ anticomplete outcome of \cref{lem:chirdl,lem:rdlchi} by two others on complete pairs, one with linear $\chi$ and the other with unbalanced $\chi$. The latter two are useful for \cref{thm:main}.
	
	Let us now sketch a proof of \cref{lem:antino}.
	We may assume the graph $G$ in question is connected; and by \cref{lem:chirdl} (or \cref{lem:rdlchi}), we may also assume that every induced subgraph of $G$ with linear chromatic number contains a linear-$\chi$ anticomplete pair (let us say this is a `locally sparse' hypothesis). Our strategy will be to use the aforementioned decomposition to iteratively obtain anticomplete pairs $(A,B)$ in $G$ such that the ratio $\min(\chi(A),\chi(B))/\chi(G)$ comes arbitrarily close to $1$.
	At the beginning, this ratio will be some small fixed constant $\delta>0$ resulting from \cref{lem:chirdl} (as above, with $\eps=2^{-9}$ say); and after the first step, the decomposition enables us to increase this ratio to $1-\xi$ for another fixed small $\xi>0$ for {\em all} connected $P_5$-free $G$ (in other words, `almost one' versus `almost one', not just linear versus `almost one' as in the above observation; see \cref{lem:1step}).
	Provided that none of the outcomes of  \cref{lem:antino} happens, we can upgrade the above `locally sparse' hypothesis for every such $G$; the upgraded hypothesis says that every induced subgraph $J$ of $G$ with $\chi(J)\ge(1-\xi^2)\chi(G)$ contains an anticomplete pair of vertex subsets each with chromatic number at least $(1-O(\xi))\chi(G)$.
	Now, starting from the second step, we will iteratively decrease $\xi$ by a power of $2$ (at each step for all $G$, again via the decomposition; see \cref{lem:incchi}) so that the connected $P_5$-free graph $G$ in consideration will be very close to being disconnected, in the sense that the minimal cutset separating the anticomplete pair $(A,B)$ in $G$ has chromatic number less than $\operatorname{poly}(\xi)\cdot\chi(G)$. But such a cutset is always nonempty; and so when $\xi$ drops below some negative power of $\chi(G)$, the second outcome of \cref{lem:antino} automatically holds with $y=\xi$ via \cref{lem:mixed}.
	Throughout the process, the $\chi$-dense induced subgraph $F$ in the third outcome of \cref{lem:antino} will be given by \cref{lem:chirdl} (first step) or will occur somewhere in the `central' cutset of the decomposition (starting from the second step, via a couple of induced $P_5$ chases);
	and in the unbalanced high-$\chi$ complete pair in the second outcome of \cref{lem:antino} (also the first outcome of \cref{thm:main}), the `lighter' part $X$ will mostly come from the set of nonneighbours of some vertex in the central cutset of the decomposition, and the `heavier' part $Y$ will be a connected component of the rest.
	
	\subsection{Chromatic density increment}
	\label{subsec:incre}
	The method of density increment is a fundamental idea in additive combinatorics (see for instance~\cite{MR4720301}), and has been an inspiration behind recent progress on long-standing open problems in graph theory such as the \erh{} conjecture \ref{conj:eh}~\cite{density3,density4,density6,density7} (under the term `iterative sparsification') and Hadwiger's conjecture~\cite{MR4576840,MR4868948}.
	Our third ingredient of the proof of \cref{thm:main} will be `chromatic density increment' (or `$\chi$-density increment' for short), which is an adaptation of this idea in polynomial $\chi$-boundedness through the lens of $\chi$-dense graphs.
	We will use this technique to convert the polynomially $\chi$-dense induced subgraph of the third outcome of \cref{lem:noanti} into a polynomially balanced complete blockade as in the second outcome of \cref{thm:main}.
	The general strategy is that given a $\chi$-dense graph $G$, we will attempt to pass through a sequence of successively $\chi$-denser induced subgraphs of $G$, with a polynomial `trade-off' between the chromatic number's shrinkage factor and the chromatic density parameter's decay.
	If we get stuck along the way then certain highly desirable outcomes occur, and if not then we obtain a clique of size some small power of $\chi(G)$ when the process terminates.
	This strategy is very much inspired by the method of iterative sparsification behind recent work on \erh{}; and indeed our chromatic density increment argument in this paper will be modelled after the two-round iterative sparsification argument on complements of $P_5$-free graphs in~\cite[Sections 5, 6, 7]{density7}.
	(We therefore strongly recommend the reader to revisit~\cite{density7} before going through this part of the~proof.)
	
	There are a number of nontrivial technical issues arising from this approach, however, mainly because chromatic number is generally not amenable to counting vertices.
	To explain, the inequality $\chi(A)+\chi(B)\ge\chi(A\cup B)$ for all disjoint vertex subsets $A,B$ in any given ($P_5$-free) graph $G$ is far from an equality in general;
	and so most of the counting-based steps in~\cite{density7} such as double counting, random sampling, and the `comb' lemma~\cite[Lemma 2.1]{MR4563865} are no longer applicable in our current setting of $\chi$-dense graphs.
	Thus, we have to develop a workaround in the first round (see \cref{sec:1incre}), and have to implement a different kind of chromatic density increment in the second round compared to its counterpart in~\cite[Section 7]{density7} (see \cref{sec:2incre}).
	To describe the difference in the second round, we need a couple more definitions. 
	For a graph $G$ with disjoint $A,B\subset V(G)$, we say that $B$ is {\em $(\eps,\chi)$-dense} to $A$ in $G$ if $\chi(A\setminus N_G(v))< \eps\cdot\chi(A)$ for all $v\in B$; and a blockade $(B_1,\ldots,B_k)$ in $G$ is {\em $(\eps,\chi)$-dense} if $B_j$ is $(\eps,\chi)$-dense to $B_i$ for all $i,j\in[k]$ with $i<j$ (these definitions first appeared in~\cite[Chapter 14]{2025thes} under the term {\em $\eps$-vivid}).
	The following lemma is the key component of the second round, and will be proved in \cref{sec:conv} by combining \cref{lem:chirdl}, the first round, and the \erh{} property of $P_5$~\cite{density7}, the last of which is critical in producing two separate~outcomes.
	
	\begin{lemma}
		\label{lem:waymid}
		There exists $b\ge2$ such that for every $\eps\in(0,\frac12]$, every $P_5$-free graph $G$ with $\chi(G)\ge\eps^{-b}$ contains an anticomplete or $(\eps,\chi)$-dense blockade of length at least $\eps^{-1}$ and mass at least $\eps^b\chi(G)$.
	\end{lemma}
	
	Since it is not clear to us whether $(\eps,\chi)$-dense blockades can always be converted into $(\operatorname{poly}(\eps),\chi)$-dense induced subgraphs, the second round will use \cref{lem:waymid} to run $\chi$-density increment on these blockades instead of $\chi$-dense induced subgraphs.
	It would also be interesting to upgrade the $(\eps,\chi)$-dense blockade outcome of this lemma into an $(\eps,\chi)$-dense induced subgraph outcome.
	If true, this would substantially strengthen \cref{lem:rdlchi,lem:chirdl}, and would be a full chromatic analogue of the polynomial R\"odl property of $P_5$.
	On a side note, it is not hard to deduce the \erh{} property of $P_5$ from \cref{lem:waymid} in return, by taking $\eps=\omega(G)^{-2}$.
	
	The final product of the two-round $\chi$-density increment argument is the following result on polynomially balanced complete blockades in $\chi$-dense $P_5$-free graphs.
	This turns the third outcome of \cref{lem:antino} into the second outcome of \cref{thm:main}, and will be proved in \cref{sec:2incre}.
	
	\begin{lemma}
		\label{lem:incre}
		There exists $a\ge2$ such that for every $\eps\in(0,2^{-32}]$ and every $(\eps,\chi)$-dense $P_5$-free graph $G$ with $\chi(G)\ge\eps^{-a}$, there exist an integer $k\ge\eps^{-1/16}$ and a complete blockade in $G$ of length $k$ and mass at least $k^{-a}\chi(G)$.
	\end{lemma}
	
	Provided \cref{lem:antino,lem:incre}, we can now complete the proof of \cref{thm:main} as follows.
	\begin{proof}
		[Proof of \cref{thm:main}, assuming \cref{lem:antino,lem:incre}]
		Let $a\ge 2$ be given by \cref{lem:incre}; we claim that $d:=32a+96$ suffices.
		To see this, let $c:=2^{-32}$. If $\chi(G)\le2^d$ then the second outcome holds with $k=2$ since $G$ has at least one edge.
		Thus we may assume $\chi(G)\ge 2^d\ge c^{-a-3}$.
		By \cref{lem:antino} with $a+3$ replacing $a$, the graph $G$ contains either:
		\begin{itemize}
			\item a complete pair $(P,Q)$ with $\chi(P),\chi(Q)\ge \frac12c^4\chi(G)=2^{-129}\chi(G)\ge 2^{-d}\chi(G)$;
			
			\item a complete pair $(X,Y)$ with $\chi(X)\ge y^{5(a+3)}\chi(G)\ge y^d\chi(G)$ and $\chi(Y)\ge (1-y)\chi(G)$ for some $y\in(0,2c]=(0,2^{-31}]$; or
			
			\item an $(\eps,\chi)$-dense induced subgraph $F$ with $\chi(F)\ge \eps^3\chi(G)$ for some $\eps\in[\chi(G)^{-1/(a+3)},c]$.
		\end{itemize}
		
		If the first or second bullet holds then the theorem holds.
		Thus we may assume that the third bullet holds.
		Since $\chi(F)\ge \eps^3\chi(G)\ge \eps^{-a}$,
		the choice of $a$ gives an integer $k \ge\eps^{-1/16}\ge c^{-1/16}\ge 2$ and a complete blockade in $F$ of length $k$ and mass at least 
		$$k^{-a}\chi(F)\ge k^{-a}\eps^3\chi(G)\ge  k^{-a-48}\chi(G)\ge k^{-d}\chi(G).$$
		This verifies the second outcome and in turn proves \cref{thm:main}.
	\end{proof}
	
	To provide an exposition for $\chi$-density increment, in \cref{sec:c5} we will use this method to give a relatively short proof that a variant of \cref{lem:incre} holds for $\{P_5,C_5\}$-free graphs with $a=40$ (see \cref{lem:ext}).
	Via \cref{lem:antino}, this yields a variant of \cref{thm:main} for these graphs with $d=40$ (see \cref{thm:mainc5}), which in turn gives the $\chi$-binding function $x\mapsto x^{40}$ for $\{P_5,C_5\}$-free graphs.
	
	\subsection{Remarks on measures and algorithmic aspects}
	\label{subsec:meas}
	For a graph $G$, a {\em measure} on $G$ is a function $\mu\colon 2^{V(G)}\to\mab R_{\ge0}$ satisfying the following axioms:
	
	\begin{itemize}
		\item $\mu(\emptyset)=0$ and $\mu(\{v\})=1$ for all $v\in V(G)$;
		
		\item $\max(\mu(A),\mu(B))\le \mu(A\cup B)$ for all $A,B\subset V(G)$, where equality holds if $A,B$ are disjoint and $A$ is anticomplete to $B$ in $G$; and
		
		\item $\mu(A\cup B)\le \mu(A)+\mu(B)$ for all $A,B\subset V(G)$, where equality holds if $A,B$ are disjoint and $A$ is complete to $B$ in $G$.
	\end{itemize}
	(In particular, the second axiom implies that $\mu$ is monotone increasing with respect to inclusion.)
	For each such measure $\mu$, let $\mu(F):=\mu(V(F))$ for every induced subgraph $F$ of $G$.
	Then it is not hard to see that clique number, fractional chromatic number, and chromatic number are all measures on any given graph.
	Moreover, clique number and chromatic number, respectively, are the `weakest' and `strongest' measures in the sense that $\omega(G)\le\mu(G)\le\chi(G)$ for all graphs $G$ and all measures $\mu$ on $G$.
	Indeed, the first inequality is evident from the third axiom (including its equality case) and the monotone property of measures.
	To see the second inequality, partition $V(G)$ into stable sets $V_1,\ldots,V_{\chi(G)}$ in $G$, and observe that $\mu(G)\le \sum_{i=1}^{\chi(G)}\mu(V_i)=\chi(G)$; here the inequality follows from the third axiom, and the equality follows from the first and the equality case of the second.
	
	Our proof of \cref{thm:main} only relies on the above three `measure' properties of $\chi$ to explore the chromatic density aspect of $P_5$-free graphs.
	In fact, the proof has only one place where the equality case of the third axiom is used (in the proof of \cref{claim:1dense11}); and this can be avoided since we can strengthen the notion of `maximal components' (defined prior to \cref{thm:gas}) by requiring the component $C$ in the definition to have maximal cardinality in addition. Nevertheless, to keep our proof as natural as possible and avoid unnecessary technical details, we decided to employ the third axiom's equality case for $\chi$.
	Our proof of \cref{thm:main} also does not involve clique number, which is perhaps another notable difference compared to previous papers in (polynomial) $\chi$-boundedness.
	Thus, because one can approximate (in fact, even compute precisely) the fractional chromatic number of $P_5$-free graphs in polynomial time (since {\sc Maximum Weight Independent Set} is poly-time solvable on these graphs~\cite{MR3376403}), the algorithmic nature of our argument yields a constructive version of \cref{thm:main} for fractional chromatic number (see \cref{thm:mainaglo}), which gives a poly-time algorithm for polynomially approximating {\sc Maximum Clique} on $P_5$-free graphs (see~\cref{thm:algop5}).
	
	\subsection{Organisation}
	The rest of this paper is organised as follows.
	In \cref{sec:chirdl} we prove \cref{lem:chirdl}. \cref{sec:highchi} presents a proof of \cref{lem:antino} by decomposing $P_5$-free graphs along high-$\chi$ anticomplete pairs.
	In \cref{sec:c5}, we present a short proof using chromatic density increment that $\{P_5,C_5\}$-free graphs are polynomially $\chi$-bounded.
	\cref{sec:1incre} runs the first round of $\chi$-density increment on $P_5$-free graphs.
	In \cref{sec:conv}, we deduce \cref{lem:waymid} as a product of the first round, \cref{lem:chirdl}, and the \erh{} property of $P_5$.
	\cref{sec:2incre} then employs \cref{lem:waymid} to execute the second round of $\chi$-density increment and in turn prove \cref{lem:incre}. We discuss the algorithmic aspects of chromatic density in \cref{sec:algo}.
	We then conclude the paper with several remarks in \cref{sec:concl}.
	
	\section{Chromatic quasirandomness}
	\label{sec:chirdl}
	This section provides a proof of \cref{lem:chirdl}, a `chromatic quasirandomness' analogue of R\"odl's theorem \ref{thm:rodl} concerning linear-$\chi$ pure pairs versus linearly $\chi$-dense induced subgraphs in $P_5$-free graphs.
	The proof employs the classical `Gy\'arf\'as path' argument~\cite{MR951359} that implies every $P_t$-free graph contains a vertex whose neighbourhood has linear chromatic number (and so such graphs are not `sparse' in the chromatic sense).
	Since the linear lower bound resulting from this theorem will be a crucial input of our argument, we provide a short proof for completeness.
	In what follows, for a graph $G$, a connected component $C$ of $G$ is {\em maximal} if $\chi(C)=\chi(G)$.
	
	\begin{theorem}
		[Gy\'arf\'as]
		\label{thm:gas}
		For each integer $t\ge4$,
		every $P_t$-free graph $G$ with $\chi(G)\ge2$ has a vertex $v$ with $\chi(N_G(v))\ge\frac1{t-2}\chi(G)$.
		Hence, if $G$ is $P_5$-free then it has a vertex $v$ with $\chi(N_G(v))\ge\frac13\chi(G)$.
	\end{theorem}
	\begin{proof}
		We may assume that $G$ is connected. Let $v_0\in V(G)$. If $N_G[v_0]=V(G)$ then we are done because $\chi(N_G(v_0))\ge\chi(G)-1\ge\frac12\chi(G)\ge\frac1{t-2}\chi(G)$ since $\chi(G)\ge2$ and $t\ge4$. Thus we may assume $N_G[v_0]\subsetneq V(G)$. Then there exists $s\ge2$ maximal such that there is a partition $(L_0,L_1,\ldots,L_s)$ of $V(G)$ into nonempty subsets satisfying:
		\begin{itemize}
			\item $L_0=\{v_0\}$ and $L_1=N_G(v_0)$;
			
			\item for every $i\in[s-1]$, the vertex set of each component of $G[L_i]$ is a subset of the neighbourhood of some vertex in $L_{i-1}$; and
			
			\item for every $i\in[s]$ and every component $C$ of $G[L_i]$, there is an induced path $v_0^C\text-v_1^C\text-\cdots\text-v_{i-1}^C$ in $G$ such that $v_{j-1}^C\in L_{j-1}$ for all $j\in[i]$ and $v_{i-1}^C$ is the only vertex in this path having a neighbour in $V(C)$. 
		\end{itemize}
		(These conditions clearly hold for $s=2$ by taking $L_2= V(G)\setminus N_G[v_0]$.)
		
		We claim that the second bullet above also holds for $i=s$. Suppose not. Let $\mac C$ be the set of components of $G[L_s]$ each with vertex set not contained in the neighbourhood of any vertex in $L_{s-1}$. For each $C\in\mac C$ let $v_0^C\text-v_1^C\text-\cdots\text-v_{s-1}^C$ be an induced path given by the third bullet above; then $V(C)\setminus N_G(v_{s-1}^C)$ is nonempty. Let $L_{s+1}:=\bigcup_{C\in\mac C}(V(C)\setminus N_G(v_{s-1}^C))$; then it is not hard to check that $(L_0,\ldots,L_{s-1},L_s\setminus L_{s+1},L_{s+1})$ violates the maximality of $s$. Thus every component of $L_s$ also has vertex set contained in the neighbourhood of some vertex in $L_{s-1}$.
		
		Now, the third bullet above with $i=s$ implies that $G$ has an induced $P_{s+1}$; and so $s\le t-2$ since $G$ is $P_t$-free.
		Since $N_G[v_0]\subsetneq V(G)$, we have $\chi(G\setminus L_1)=\chi(L_0\cup(L_2\cup\cdots\cup L_s))=\chi(L_2\cup\cdots\cup L_s)$; and so $\chi(G)\le\chi(L_1)+\chi(G\setminus L_1)=\chi(L_1)+\chi(L_2\cup\cdots\cup L_s)$, which gives some $i\in[s]$ with $\chi(L_i)\ge \frac1s\chi(G)\ge\frac1{t-2}\chi(G)$.
		The conclusion then follows since each component of $L_i$ has vertex set contained in the neighbourhood of some vertex in $L_{i-1}$.
		This proves \cref{thm:gas}.
	\end{proof}
	
	We now prepare for the proof of \cref{lem:chirdl}.
	First, we show how to create a $\chi$-dense induced subgraph in a $P_5$-free graph under several favourable conditions.
	\begin{lemma}
		\label{lem:bip}
		Let $\eps\in(0,\frac14)$, and let $G$ be a $P_5$-free graph. Let $v\in V(G)$, let $A\subset N_G(v)$, and let $B\subset V(G)\setminus N_G[v]$, such that $B$ is $(\eps,\chi)$-dense to $A$. Then either:
		\begin{itemize}
			\item there exists a pure pair $(X,Y)$ with $X\subset A$, $Y\subset B$, $\chi(X)\ge\eps\cdot\chi(A)$, and $\chi(Y)\ge\eps\cdot\chi(B)$;
			
			\item there exists $C\subset A$ with $\chi(C)\ge\frac12\chi(A)$ such that $G[C]$ is $(4\eps,\chi)$-dense; or
			
			\item there exists $D\subset B$ with $\chi(D)\ge\frac12\chi(B)$ such that $G[D]$ is $(4\eps,\chi)$-dense. 
		\end{itemize}
	\end{lemma}
	\begin{proof}
		Assume that the first outcome fails.
		Let $C:=\{u\in A:\chi(A\setminus N_G[u])\le 3\eps\cdot\chi(A)\}$ and $D:=\{z\in B:\chi(B\setminus N_G[z])\le 2\eps\cdot\chi(B)\}$. We claim that:
		
		\begin{figure}[ht]
			\centering
			
			\begin{tikzpicture}[scale=0.65,auto=left]
				
				\draw [fill=white] (1.5,0) circle (1.25cm);
				
				\draw [fill=white] (1.5,-4.25) circle (1.25cm);
				
				\draw [rounded corners] (-3,-2) rectangle (4,2);
				\draw [rounded corners] (-3,-6.25) rectangle (4,-2.25);
				\node[] at (-3.75,0) {$A$};
				\node[] at (-3.75,-4.25) {$B$};
				
				\node[inner sep=1.5pt, fill=black,circle,draw] (z) at ({-1.5},{-4.25}) {};
				\node[xshift=-0.3cm] at (z) {$z$};
				\node[] at (3.25,-4.25) {$E$};
				\node[] at (3.25,0) {$X$};
				
				\node[inner sep=1.5pt, fill=black,circle,draw] (v) at ({0.5},{3.5}) {};
				\node[xshift=0.3cm] at (v) {$v$};
				
				\node[inner sep=1.5pt, fill=black,circle,draw] (u) at ({-1.5},{0}) {};
				\node[xshift=-0.3cm] at (u) {$u$};
				\node[inner sep=1.5pt, fill=black,circle,draw] (v') at ({1},{-4.5}) {};
				\node[yshift=-0.25cm] at (v') {$v'$};
				\node[inner sep=1.5pt, fill=black,circle,draw] (v'') at ({2},{-4}) {};
				\node[xshift=0.05cm,yshift=-0.25cm] at (v'') {$v''$};
				\node[inner sep=1.5pt, fill=black,circle,draw] (u') at ({1.5},{0}) {};
				\node[xshift=0.35cm,yshift=0.05cm] at (u') {$u'$};
				\draw[-] (u) -- (v) -- (u') -- (v') -- (v'');
				\draw[dashed] (u) -- (u') -- (v'') -- (u) --  (v');

				\draw [fill=white] (12.5,0) circle (1.25cm);
				\node[inner sep=1.5pt, fill=black,circle,draw] (z0) at ({9.5},{-4.25}) {};
				
				\draw [fill=white] (12.5,-4.25) circle (1.25cm);
				
				\draw [rounded corners] (8,-2) rectangle (15,2);
				\draw [rounded corners] (8,-6.25) rectangle (15,-2.25);
				\node[] at (7.25,0) {$A$};
				\node[] at (7.25,-4.25) {$B$};
				\node[color=white] at (15.75,0) {$A$};
				\node[color=white] at (15.75,-4.25) {$B$};
				
				\node[xshift=-0.3cm] at (z0) {$z$};
				\node[] at (14.25,-4.25) {$Y$};
				\node[] at (14.25,0) {$X$};
				
				\node[inner sep=1.5pt, fill=black,circle,draw] (v0) at ({11.5},{3.5}) {};
				\node[xshift=0.3cm] at (v0) {$v$};
				
				\node[inner sep=1.5pt, fill=black,circle,draw] (u0) at ({9.5},{0}) {};
				\node[xshift=-0.3cm] at (u0) {$u$};
				\node[inner sep=1.5pt, fill=black,circle,draw] (z') at ({12.5},{-4.25}) {};
				\node[xshift=0.35cm,yshift=0.05cm] at (z') {$z'$};
				\node[inner sep=1.5pt, fill=black,circle,draw] (u'0) at ({12.5},{0}) {};
				\node[xshift=0.35cm,yshift=0.05cm] at (u'0) {$u'$};
				\draw[-] (z0) -- (u'0) -- (v0) -- (u0) -- (z');
				\draw[dashed] (u0) -- (u'0) -- (z') -- (z0) -- (u0);
			\end{tikzpicture}
			
			\caption{Proof of \cref{claim:bip}.}
			\label{fig:bip}
		\end{figure}
		
		\begin{claim}
			\label{claim:bip}
			$A\setminus C$ is complete to $B\setminus D$.
		\end{claim}

		\begin{subproof}
			Suppose not; then there are $u\in A\setminus C$ and $z\in B\setminus D$ such that $u,z$ are nonadjacent.
			Let $X:=(A\setminus N_G[u])\cap N_G(z)$
			and $Y:=(B\setminus N_G[z])\cap N_G(u)$. 
			The hypothesis implies that
			\[
			\chi(X)\ge\chi(A\setminus N_G[u])-\chi(A\setminus N_G(z))
			\ge 3\eps\cdot\chi(A)-\eps\cdot\chi(A)=2\eps\cdot\chi(A).\]
			
			We next claim that $\chi(Y)\ge \eps\cdot\chi(B)$.
			To see this, we may assume $Y\subsetneq B\setminus N_G[z]$; and so there exists $E\subset (B\setminus N_G[z])\setminus Y=B\setminus(N_G[z]\cup N_G(u))$ such that $G[E]$ is a maximal component of $G[(B\setminus N_G[z])\setminus Y]$.
			If there exists $u'\in X$ mixed on $E$ then $E$ has an edge $v'v''$ for which $u'$ is adjacent to $v'$ and nonadjacent to $v''$; but then $u\text-v\text-u'\text-v'\text-v''$ would be an induced $P_5$ in $G$, a contradiction.
			(See the left-hand side of \cref{fig:bip}.)
			Hence there is a partition $X=X_1\cup X_2$ such that $X_1$ is complete to $E$ and $X_2$ is anticomplete to $E$.
			Thus $\max(\chi(X_1),\chi(X_2))\ge\frac12\chi(X)\ge \eps\cdot\chi(A)$;
			and so $\chi(E)\le \eps\cdot\chi(B)$ since the first outcome of the lemma fails.
			Therefore
			\[\chi(Y)\ge\chi(B\setminus N_G[z])-\chi(E)
			\ge 2\eps\cdot\chi(B)-\eps\cdot\chi(B)=\eps\cdot\chi(B).\] 
			Thus, since the first outcome of the lemma fails, there are nonadjacent vertices $u'\in X$ and $z'\in Y$;
			but then $z\text-u'\text-v\text-u\text-z'$ would be an induced $P_5$ in $G$, a contradiction.
			(See the right-hand side of \cref{fig:bip}.)
			This proves \cref{claim:bip}.
		\end{subproof}
		Now, since the first outcome of the lemma fails, either $\chi(A\setminus C)\le\eps\cdot\chi(A)$ or $\chi(B\setminus D)\le\eps\cdot\chi(B)$; and so either $\chi(C)\ge(1-\eps)\chi(A)\ge\frac34\chi(A)$ or $\chi(D)\ge(1-\eps)\chi(B)\ge\frac34\chi(B)$.
		Thus one of the last two outcomes holds. This proves \cref{lem:bip}.
	\end{proof}
	We are now ready to prove \cref{lem:chirdl}, which we restate here for the reader's convenience.
	\begin{lemma}
		\label{lem:chirdl1}
		Let $\eps\in(0,\frac12)$ and $\delta\in(0,2^{-7}\eps^2]$. Then every $P_5$-free graph $G$ contains either:
		\begin{itemize}
			\item a pure pair $(A,B)$ with $\chi(A),\chi(B)\ge \delta\cdot\chi(G)$; or
			
			\item an $(\eps,\chi)$-dense induced subgraph $F$ with $\chi(F)\ge \delta\cdot\chi(G)$.
		\end{itemize}
	\end{lemma}
	\begin{proof}
		Suppose that none of the outcomes holds.
		Then $\chi(G)\ge\delta^{-1}\ge 2^7$ since the second outcome fails.
		Our plan is to repeatedly apply \cref{thm:gas} to obtain a vertex $v$ and disjoint $A,B\subset V(G)\setminus\{v\}$ with linear chromatic number, such that $v$ is complete to $A$ and anticomplete to $B$, and $B$ is $\chi$-dense to $A$. Then the conclusion follows by \cref{lem:bip}. We first extract $v$ as follows.
		
		\begin{claim}
			\label{claim:nicev}
			There exists $v\in V(G)$ with $\chi(N_G(v))\ge\frac16\chi(G)$ and $\chi(G\setminus N_G[v])\ge\frac12\eps\cdot\chi(G)$.
		\end{claim}
		\begin{subproof}
			Let $Z:=\{u\in V(G):\chi(G\setminus N_G[u])\ge\frac12\eps\cdot\chi(G)\}$.
			If $\chi(G\setminus Z)\ge\frac12\chi(G)$ then $G\setminus Z$ is $(\eps,\chi)$-dense; and so the second outcome of the lemma holds, contrary to our supposition.
			Thus $\chi(G\setminus Z)\le\frac12\chi(G)$; and so $\chi(Z)\ge\frac12\chi(G)\ge2$.
			By \cref{thm:gas} applied to $G[Z]$, there exists $v\in Z$ with $\chi(Z\cap N_G(v))\ge\frac13\chi(Z)\ge \frac16\chi(G)$.
			This proves \cref{claim:nicev}.
		\end{subproof}
		
		For $v$ given by \cref{claim:nicev}, let
		\[A:=\{u\in N_G(v):\chi(G\setminus(N_G(u)\cup N_G(v)))< \delta\cdot\chi(G)\}.\]
		We continue using the $P_5$-freeness of $G$ to deduce that $A$ is high-chromatic, as follows.
		\begin{claim}
			\label{claim:chirdl2}
			$\chi(A)\ge\frac1{8}\chi(G)$.
		\end{claim}
		\begin{subproof}
			Let $P:=N_G(v)\setminus A$. Since $\chi(N_G(v))\ge\frac16\chi(G)$, it suffices to prove $\chi(P)\le\frac1{24}\chi(G)$.
			Suppose not; then since the second outcome of the lemma fails, there exists $u\in P$ with $\chi(P\setminus N_G[u])\ge\frac1{24}\eps\cdot\chi(G)$.
			Let $C$ be a maximal component of $G\setminus(N_G(u)\cup N_G(v))$.
			If there exists $z\in P\setminus N_G[u]$ mixed on $V(C)$, then $C$ has an edge $u'v'$ with $z$ adjacent to $u'$ and nonadjacent to $v'$ in $G$;
			but then $u\text-v\text-z\text-u'\text-v'$ would be an induced $P_5$ in $G$, a contradiction.
			(See the left-hand side of \cref{fig:bip} for a reference, with $z,u',v'$ replacing $u',v',v''$ respectively.)
			Hence, $P\setminus N_G[u]$ admits a partition $(X,Y)$ such that $C$ is complete to $X$ and anticomplete to $Y$.
			Thus $\chi(X),\chi(Y)\le \delta\cdot\chi(G)$ since the first outcome of the lemma fails;
			and so $\chi(P\setminus N_G[u])\le 2\delta\cdot\chi(G)$.
			But then the choice of $u$ yields $2\delta\ge\frac1{24}\eps$, contrary to $\delta\le2^{-7}\eps^2$. This proves \cref{claim:chirdl2}.
		\end{subproof}
		
		Now, let
		\[B:=\{z\in V(G)\setminus N_G[v]: \chi(A\setminus N_G(z))< (\eps/4)\cdot\chi(A)\}.\]
		We apply \cref{lem:bip} to bound $\chi(B)$ via the following claim.
		\begin{claim}
			\label{claim:chirdl3}
			$\chi(B)\le \frac14\eps\cdot\chi(G)$.
		\end{claim}
		\begin{subproof}
			Suppose not. By \cref{lem:bip} with $(\eps,v,A,B)=(\frac14\eps,v,A,B)$, either:
			\begin{itemize}
				\item there exists a pure pair $(X,Y)$ with $X\subset A$, $Y\subset B$, $\chi(X)\ge\frac14\eps\cdot\chi(A)$, and  $\chi(Y)\ge\frac14\eps\cdot\chi(B)$;
				
				\item there exists $C\subset A$ with $\chi(C)\ge\frac12\chi(A)$ such that $G[C]$ is $(\eps,\chi)$-dense; or
				
				\item there exists $D\subset B$ with $\chi(D)\ge\frac12\chi(B)$ such that $G[D]$ is $(\eps,\chi)$-dense.
			\end{itemize}
			
			If the first bullet holds then the first outcome of the lemma holds since $\min(\frac14\eps\cdot\chi(A),\frac14\eps\cdot\chi(B))\ge\min(\frac1{32}\eps,\frac1{16}\eps^2)\cdot\chi(G)\ge\delta\cdot\chi(G)$.
			If the second bullet or the third bullet holds then the second outcome of the lemma holds since $\min(\frac12\chi(A),\frac12\chi(B))\ge \min(\frac1{24},\frac18\eps)\cdot\chi(G)\ge\delta\cdot\chi(G)$.
			These contradictions together verify \cref{claim:chirdl3}.
		\end{subproof}
		Now, let $Q:=V(G)\setminus(B\cup N_G[v])$. By the choice of $v$, \cref{claim:chirdl3} implies
		$$\chi(Q)\ge \chi(G\setminus N_G[v])-\chi(B)\ge (\eps/4)\cdot \chi(G)\ge (\eps/4)\delta^{-1}\ge\eps^{-1}\ge 2.$$
		Hence \cref{thm:gas} gives some $v'\in Q$ with $\chi(N_{G[Q]}(v'))\ge \frac13\chi(Q)\ge\frac1{12}\eps\cdot\chi(G)$.
		Let $B':=A\setminus N_G(v')$ and $A':=N_{G[Q]}(v')$; then $\chi(A')\ge\frac1{12}\eps\cdot\chi(G)$. Since $\delta\le 2^{-7}\eps^2$, we have
		\begin{gather*}
			\chi(B')\ge (\eps/4)\cdot\chi(A)\ge2^{-5}\eps\cdot\chi(G)\ge \delta\cdot\chi(G),\quad\text{and}\\
			\chi(A'\setminus N_G(u))\le \delta\cdot\chi(G)\le 12\eps^{-1}\delta\cdot\chi(A')<(\eps/4)\cdot\chi(A')\quad\text{for all $u\in B'$.}
		\end{gather*} 
		Thus $B'$ is $(\frac14\eps,\chi)$-dense to $A'$; and so by \cref{lem:bip} with $(\eps,v,A,B)=(\frac14\eps,v',A',B')$, either:
		\begin{itemize}
			\item there exists a pure pair $(X,Y)$ with $X\subset A'$, $Y\subset B'$, $\chi(X)\ge\frac14\eps\cdot\chi(A')$, and  $\chi(Y)\ge\frac14\eps\cdot\chi(B')$;
			
			\item there exists $C\subset A'$ with $\chi(C)\ge\frac12\chi(A')$ such that $G[C]$ is $(\eps,\chi)$-dense; or
			
			\item there exists $D\subset B'$ with $\chi(D)\ge\frac12\chi(B')$ such that $G[D]$ is $(\eps,\chi)$-dense.
		\end{itemize}
		
		If the first bullet holds then the first outcome of the lemma holds since $\min(\frac14\eps\cdot\chi(A'),\frac14\eps\cdot\chi(B'))\ge\min(\frac1{48}\eps^2,2^{-7}\eps^2)\cdot\chi(G)\ge\delta\cdot\chi(G)$; note that $\delta\le 2^{-7}\eps^2$.
		If the second bullet or the third bullet holds then the second outcome of the lemma holds since $\min(\frac12\chi(A'),\frac12\chi(B'))\ge \min(\frac1{24}\eps,2^{-6}\eps)\cdot\chi(G)\ge\delta\cdot\chi(G)$.
		These contradictions together verify \cref{lem:chirdl1}.
	\end{proof}

	\section{Capturing $\chi$-dense subgraphs by growing high-$\chi$ anticomplete pairs}
	\label{sec:highchi}
	Our goal in this section is to prove \cref{lem:antino} by decomposing $P_5$-free graphs along high-$\chi$ anticomplete pairs.
	As sketched in \cref{subsec:highchi}, we will carry this out under the `locally sparse' hypothesis that every linear-$\chi$ induced subgraph contains a linear-$\chi$ anticomplete pair.
	To formalise this hypothesis, for $0\le p\le q$, we say that a graph $G$ is {\em $(p,q)$-sparse} if every induced subgraph $F$ of $G$ with $\chi(F)\ge q$ contains an anticomplete pair $(P,Q)$ with $\chi(P),\chi(Q)\ge p$.
	We also say that $S\subset V(G)$ {\em separates} disjoint $B_1,\ldots,B_k\subset V(G)\setminus S$ in $G$ if $G\setminus S$ has no path between $B_i$  and $B_j$ for all distinct $i,j\in[k]$; and $S$ is a {\em cutset} of $G$ if there are at least two nonempty sets among $B_1,\ldots,B_k$.
	The following lemma is the main technical step in the proof of \cref{lem:antino}, which introduces and applies the decomposition mentioned in \cref{subsec:highchi}.
	
	\begin{lemma}
		\label{lem:anti}
		Let $\eps\in(0,\frac14]$, let $G$ be a $P_5$-free graph, and let $0<p\le q\le(1-\eps^2)\chi(G)$. 
		If $G$ is $(p,q)$-sparse, then it contains either:
		\begin{itemize}
			\item an anticomplete pair $(A,B)$ with $\chi(A)\ge q-2\eps^4\chi(G)$ and $\chi(B)\ge(1-\eps^2)\chi(G)$;
			
			\item a complete pair $(X,Y)$ with $\chi(X)\ge \eps^4\chi(G)$ and $\chi(Y)\ge p$; or
			
			\item an $(\eps,\chi)$-dense induced subgraph with chromatic number at least $2\eps^3\chi(G)$.
		\end{itemize}
	\end{lemma}
	(In application we will let $p$ vary and always fix $q=(1-\eps^2)\chi(G)$, but for better clarity we decided to make $q$ an independent variable in this lemma.)
	\begin{proof}
		We may assume that $G$ is connected. Also, assume that the last two outcomes do not hold.
		To facilitate our decomposition, we begin with the following application of \cref{lem:mixed}:
		\begin{claim}
			\label{claim:sparse}
			Every connected induced subgraph $F$ of $G$ with $\chi(F)\ge q$ contains a minimal nonempty cutset $Z$ separating two vertex subsets $A,B$ with $\chi(A)\ge \max(p,\chi(F)-2\eps^4\chi(G))$ and $\chi(B)\ge p$, such that $F[A],F[B]$ are two of the components of $F\setminus Z$.  
		\end{claim}
		\begin{subproof}
			Since $G$ is $(p,q)$-sparse, $F$ contains an anticomplete pair $(A,B)$ with $\chi(A)\ge\chi(B)\ge p$; and we may assume that $F[A],F[B]$ are connected.
			Among all such anticomplete pairs $(A,B)$, choose one with $\chi(A)+\chi(B)$ maximal; and subject to this, with $\abs A+\abs B$ maximal.
			Since $F$ is connected, it has a minimal nonempty cutset $Z$ separating $A,B$.
			By the maximality of $(A,B)$, $F[A],F[B]$ are two of the components of $F\setminus Z$.
			By the minimality of $Z$, every vertex in $Z$ has a neighbour in each of $A$ and $B$; and so, because $F$ is $P_5$-free, \cref{lem:mixed} implies that every such vertex is complete to $A$ or to $B$ in $F$.
			Thus $\chi(Z)\le 2\eps^4\chi(G)$ since the second outcome of the lemma fails; and so the maximality of $\chi(A)$ yields
			$\chi(A)=\chi(F\setminus Z)\ge \chi(F)-2\eps^4\chi(G)$.
			This proves \cref{claim:sparse}.
		\end{subproof}

		\begin{figure}[ht]
			\centering
			\begin{tikzpicture}[scale=0.7,auto=left]

				\draw [rounded corners] (-2,2) rectangle (2,4.5);
				\node [] at (-2.5,3.25) {$A$};
				
				\draw [rounded corners] (-3.5,-1.5) rectangle (3.5,1.5);
				\node [] at (-4,0) {$D$};
				
				\draw [rounded corners] (7.5,-0.75) rectangle (9,0.75);
				\node [] at (7,0) {$E$};
				\draw [color=white] (-7.5,-0.75) rectangle (-9,0.75);
				
				\draw[] (0,0) ellipse (2cm and 1cm);
				\node [] at (-2.5,0) {$S$};
				
				\node[inner sep=1.5pt, fill=black,circle,draw] (u) at ({-1},{0}) {};
				\node[xshift=-0.3cm] at (u) {$u$};
				\node[inner sep=1.5pt, fill=black,circle,draw] (v) at ({1},{0}) {};
				\node[xshift=0.3cm] at (v) {$v$};
				
				\draw [] (-6,-3) circle (1cm);
				\draw [] (-2,-3) circle (1cm);
				\draw [] (2,-3) circle (1cm);
				\draw [] (6,-3) circle (1cm);
				
				\node [] at (-6,-4.5) {$B_1$};
				\node [] at (-2,-4.5) {$B_i$};
				\node [] at (2,-4.5) {$B_j$};
				\node [] at (6,-4.5) {$B_k$};
				
				\node[] at (-4,-3) {$\cdots$};
				\node[] at (0,-3) {$\cdots$};
				\node[] at (4,-3) {$\cdots$};
				
				\node[inner sep=1.5pt, fill=black,circle,draw] (zi) at (-2,-3) {};
				\node[xshift=-0.3cm] at (zi) {$z_i$};
				\node[inner sep=1.5pt, fill=black,circle,draw] (zj) at (2,-3) {};
				\node[xshift=0.3cm] at (zj) {$z_j$};
				
				\draw[-] (zi) -- (u) -- (-1,2.25);
				\draw[-] (zj) -- (v) -- (1,2.25);
				\draw[-] (1,2.25) arc (0:180:1cm);
				\node [] at (0,3.65) {$P$};
				\draw[dashed] (zi) -- (v) -- (u) -- (zj) to 
				[bend left=40] (zi);
			\end{tikzpicture}
			
			\caption{The decomposition and proof of \cref{claim:2k2}.}
			\label{fig:2k2}
		\end{figure}
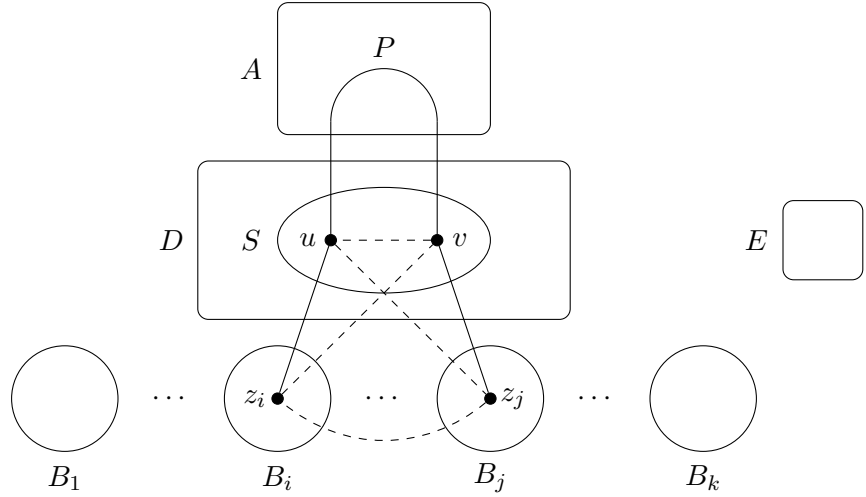

		We now define our decomposition as follows.
		By \cref{claim:sparse} with $F=G$, the graph $G$ has a minimal nonempty cutset $Z$ separating two vertex subsets $A,B$ of $G$ such that $G[A],G[B]$ are two of the components of $G\setminus Z$, $\chi(B)\ge p$, and
		\[\chi(A)\ge(1-2\eps^4)\chi(G)\ge \max(p,q-2\eps^4\chi(G)).\]
		Thus, there exists $k\ge1$ maximal such that $V(G)$ admits a partition $(A,D,B_1,\ldots,B_k,E)$ satisfying:
		\begin{itemize}
			\item $D$ is a nonempty cutset separating $A,B_1,\ldots,B_k,E$ in $G$;
			
			\item $G[A]$ and $G[B_1],\ldots,G[B_k]$ are connected;
			
			\item every vertex in $D$ has a neighbour in $B_1\cup\cdots\cup B_k$; and
			
			\item $\chi(A)\ge \max(p,q-2\eps^4\chi(G))$, $\chi(E)<p$, and $\chi(B_i)\ge p$ for all $i\in[k]$.
		\end{itemize}
		(See \cref{fig:2k2} for an illustration.)
		The maximality of $k$ implies that $\chi(A)$ is not too large, as follows:
		
		\begin{claim}
			\label{claim:smalla}
			$\chi(A)<q$.
		\end{claim}
		
		\begin{subproof}
			Suppose not.
			Then \cref{claim:sparse} gives a minimal nonempty cutset $S'$ separating two vertex subsets $A',B'$ in $G[A]$ such that $\chi(A')\ge\max(p,\chi(A)-2\eps^4\chi(G))\ge \max(p,q-2\eps^4\chi(G))$, $\chi(B')\ge p$, and $G[A'], G[B']$ are two of the components of $G[A\setminus S']$. Let $G[B_{k+1}],\ldots,G[B_{k+r}]$ be the components of $G[A\setminus(A'\cup S')]$ with chromatic number at least $p$, and let $E':=A\setminus(A'\cup S'\cup(B_{k+1}\cup\cdots\cup B_{k+r}))$.
			Then $r\ge1$ and $\chi(E')<p$; and so
			\[(A',D\cup S',B_1,\ldots,B_k,B_{k+1},\ldots,B_{k+r},E\cup E')\]
			would be a partition of $V(G)$ violating the maximality of $k$.
			This proves \cref{claim:smalla}.
		\end{subproof}
		Let $S$ be the set of vertices in $D$ mixed on $A$;
		then every vertex in $S$ is pure to each of $B_1,\ldots,B_k$ by \cref{lem:mixed}.
		We next employ the $P_5$-free hypothesis to analyse the vertices in $S$.
		We begin with:
		
		\begin{claim}
			\label{claim:2k2}
			Assume that there are $u,v\in S$ and $i,j\in[k]$ such that $v$ is complete to $B_i$ and anticomplete to $B_j$, and $u$ is complete to $B_j$ and anticomplete to $B_i$.
			Then $u,v$ are adjacent.
		\end{claim}
		\begin{subproof}
			Let $z_i\in B_i$ and $z_j\in B_j$.
			Since $G[A]$ is connected and each of $u,v$ has a neighbour in $A$, there is an induced path $P$ in $G$ with endpoints $u,v$ and interior in $A$. If $u,v$ are nonadjacent, then $z_i\text-P\text-z_j$ would be an induced path of length at least four in $G$, contrary to the $P_5$-freeness of $G$. (See \cref{fig:2k2}.)
			This proves \cref{claim:2k2}.
		\end{subproof}
		
		We use this to show that every vertex in $S$ has low-$\chi$ nonneighbourhood within $G[S]$, as follows.
		
		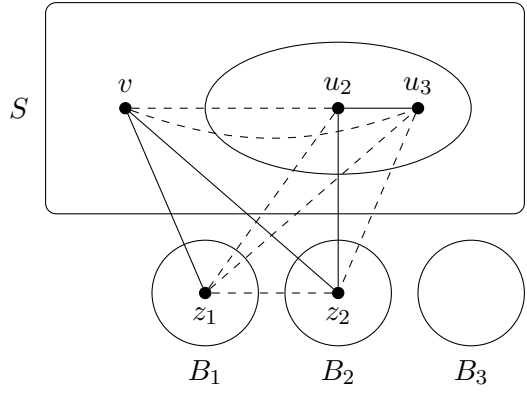
\begin{figure}[ht]
			\centering
			\begin{tikzpicture}[scale=0.7,auto=left]
				\draw [rounded corners] (-3.5,-2) rectangle (5.5,2);
				\node [] at (-4,0) {$S$};
				\node [color=white] at (6,0) {$S$};
				
				\node[inner sep=1.5pt, fill=black,circle,draw] (v) at ({-2},{0}) {};
				\node[yshift=0.3cm] at (v) {$v$};
				\node[inner sep=1.5pt, fill=black,circle,draw] (u2) at ({2},{0}) {};
				\node[yshift=0.3cm] at (u2) {$u_2$};
				\node[inner sep=1.5pt, fill=black,circle,draw] (u3) at ({3.5},{0}) {};
				\node[yshift=0.3cm] at (u3) {$u_3$};
				
				\draw[] (2,0) ellipse (2.5cm and 1.25cm);
				
				\draw [] (-0.5,-3.5) circle (1cm);
				\draw [] (2,-3.5) circle (1cm);
				\draw [] (4.5,-3.5) circle (1cm);
				
				\node[inner sep=1.5pt, fill=black,circle,draw] (z1) at (-0.5,-3.5) {};
				\node[yshift=-0.3cm] at (z1) {$z_1$};
				\node[inner sep=1.5pt, fill=black,circle,draw] (z2) at (2,-3.5) {};
				\node[yshift=-0.3cm] at (z2) {$z_2$};
				
				\node [] at (-0.5,-5) {$B_1$};
				\node [] at (2,-5) {$B_2$};
				\node [] at (4.5,-5) {$B_3$};
				
				\draw[-] (z1) -- (v) -- (z2) -- (u2) -- (u3);
				\draw[dashed] (z1) -- (u3) -- (z2) -- (z1) -- (u2) -- (v) to [bend right=20] (u3);
			\end{tikzpicture}
			
			\caption{Proof of \cref{claim:smalls}; the ellipse represents $S\setminus N_G[v]$.}
			\label{fig:smalls}
		\end{figure}
		
		\begin{claim}
			\label{claim:smalls}
			Every $v\in S$ satisfies $\chi(S\setminus N_G[v])\le 2\eps^4\chi(G)$.
		\end{claim}
		
		\begin{subproof}
			Let $I$ be the set of $i\in[k]$ for which $v$ has a neighbour in $B_i$; then $I\ne\emptyset$ and $v$ is complete to $B_i$ for all $i\in I$.
			For each $i\in I$, let $S_i$ be the set of vertices in $S\setminus N_G[v]$ with a neighbour in $B_i$; then $S_i$ is complete to $B_i$ and so $\chi(S_i)\le\eps^4\chi(G)$ since the second outcome of the lemma fails.
			
			If there exists $u\in S\setminus N_G[v]$ with no neighbour in $\bigcup_{i\in I}B_i$,
			then there would be $j\in [k]\setminus I$ such that $u$ is complete to $B_j$, which contradicts \cref{claim:2k2}.
			Therefore $S\setminus N_G[v]\subset \bigcup_{i\in I}S_i$; and so there exists $J\subset I$ minimal for which $S\setminus N_G[v]\subset \bigcup_{i\in J}S_i$.
			By the minimality of $J$, for each $i\in J$ there exists $u_i\in S_i\setminus(\bigcup_{j\in J\setminus\{i\}}S_j)$.
			
			We claim that $\abs J\le 2$.
			Suppose not; then we may assume $1,2,3\in J$ without loss of generality.
			Let $z_1\in B_1$ and $z_2\in B_2$.
			\cref{claim:2k2} implies that $u_2,u_3$ are adjacent; but then $z_1\text-v\text-z_2\text-u_2\text-u_3$ would be an induced $P_5$ in $G$, a contradiction. (See \cref{fig:smalls}.)
			Hence $\abs J\le 2$; and so $\chi(S\setminus N_G[v])\le \abs J\cdot \eps^4\chi(G)\le 2\eps^4\chi(G)$.
			This proves \cref{claim:smalls}.
		\end{subproof}
		
		Now, since the last outcome of the lemma fails, \cref{claim:smalls} implies that $\chi(S)\le 2\eps^3\chi(G)$. Let $R$ be the set of vertices in $D$ complete to $A$; then $\chi(R)\le \eps^4\chi(G)$ since the second outcome of the lemma fails.
		Let $B:=B_1\cup\cdots\cup B_k$; then since $V(G)=A\cup D\cup B\cup E$, we obtain
		\[\begin{aligned}
			\chi(A\cup(D\setminus (R\cup S))\cup B\cup E)
			&\ge \chi(G)-\chi(R\cup S)\\
			&\ge(1-2\eps^3-\eps^4)\chi(G)\ge (1-\eps^2)\chi(G)\ge q>\chi(A)
		\end{aligned}
		\]
		where the third inequality is due to $\eps\in(0,\frac14]$ and the last inequality is due to \cref{claim:smalla}.
		Thus, since $A$ is anticomplete to $(D\setminus (R\cup S))\cup B\cup E$, \cref{claim:smalla} again yields
		\[\chi((D\setminus (R\cup S))\cup B\cup E)>(1-\eps^2)\chi(G)\]
		and so $(A,(D\setminus (R\cup S))\cup B\cup E)$ satisfies the first outcome of the lemma.
		This completes the proof of \cref{lem:anti}.
	\end{proof}
	
	By taking $p=\eps^4\chi(G)$ and $q=(1-\eps^2)\chi(G)$, \cref{lem:anti} allows us to upgrade \cref{lem:chirdl} into the following, which is the first step sketched in \cref{subsec:highchi}.
	
	\begin{lemma}
		\label{lem:1step}
		For every $\eps\in(0,2^{-9}]$, every $P_5$-free graph $G$ contains either:
		\begin{itemize}
			\item an anticomplete pair $(A,B)$ with $\chi(A),\chi(B)\ge(1-2\eps^2)\chi(G)$;
			
			\item a complete pair $(X,Y)$ with $\chi(X),\chi(Y)\ge\eps^{4}\chi(G)$; or 
			
			\item an $(\eps,\chi)$-dense induced subgraph $F$ with $\chi(F)\ge 2\eps^3\chi(G)$.
		\end{itemize}
	\end{lemma}
	\begin{proof}
		Let $p:=\eps^4\chi(G)$, let $q:=(1-\eps^{2})\chi(G)$,
		and suppose that none of the outcomes holds.
		If $\chi(G)\le \frac12\eps^{-3}$ then the third outcome holds; and if $\frac12\eps^{-3}\le\chi(G)\le \eps^{-4}$ then $2\le \chi(G)\le\eps^{-4}$ and so the second outcome holds.
		Thus $\chi(G)\ge \eps^{-4}$. We claim that:
		\begin{claim}
			\label{claim:sp}
			$G$ is $(p,q)$-sparse.
		\end{claim}
		\begin{subproof}
			Let $F$ be an induced subgraph of $G$ with $\chi(F)\ge q= (1-\eps^{2})\chi(G)\ge\frac12\chi(G)$.
			By \cref{lem:chirdl}, $F$ contains either:
			\begin{itemize}
				\item a pure pair $(A,B)$ with $\chi(A),\chi(B)\ge2^{-7}\eps^2\chi(F)$; or
				
				\item a $(\eps,\chi)$-dense induced subgraph $J$ with $\chi(J)\ge 2^{-7}\eps^2\chi(F)$.
			\end{itemize}
			
			If the second bullet holds, then since $\eps\le 2^{-9}$, we have $2^{-7}\eps^2\chi(F)\ge 2^{-8}\eps^2\chi(G)\ge 2\eps^3\chi(G)$; and so the third outcome of the lemma holds, a contradiction.
			Thus the first bullet holds; and so
			$$\chi(A),\chi(B)\ge 2^{-7}\eps^2\chi(F)\ge2^{-8}\eps^{2}\chi(G)\ge \eps^4\chi(G)=p.$$
			Thus, since the second outcome of the lemma does not  hold, $A$ is anticomplete to $B$.
			This completes the proof of \cref{claim:sp}.
		\end{subproof}
		
		Now, by \cref{claim:sp,lem:anti}, $G$ contains either:
		\begin{itemize}
			\item an anticomplete pair $(A,B)$ with
			$$\chi(A)\ge q-2\eps^{4}\chi(G)=(1-\eps^2-2\eps^{4})\chi(G)\ge(1-2\eps^2)\chi(G)\quad\text{and}\quad\chi(B)\ge(1-\eps^{2})\chi(G);$$
			
			\item a complete pair $(X,Y)$ with $\chi(X)\ge\eps^{4}\chi(G)$ and $\chi(Y)\ge p=\eps^4\chi(G)$; or
			
			\item an $(\eps,\chi)$-dense induced subgraph $F$ with  $\chi(F)\ge2\eps^{3}\chi(G)$.
		\end{itemize}
		
		This proves \cref{lem:1step}.
	\end{proof}
	
	Given \cref{lem:1step}, we now iterate \cref{lem:anti} to grow high-$\chi$ anticomplete pairs indefinitely inside any given $P_5$-free graph, provided that a polynomially $\chi$-dense induced subgraph does not exist.
	In what follows, for $c>0$, let $\phi_0(c):=1$; and
	for every integer $s\ge1$, let $$\phi_s(c):=\prod_{i=1}^{s}\left(1-c^{2^{i+1}}\right),$$
	and for every integer $r$ with $0\le r\le s$, let
	$$\phi_{r,s}(c):=\phi_s(c)/\phi_r(c)=\prod_{r<i\le s}\left(1-c^{2^{i+1}}\right)$$
	in particular $\phi_{0,s}(c)=\phi_s(c)$. 
	
	\begin{lemma}
		\label{lem:incchi}
		For every $c\in(0,2^{-9}]$ and every integer $s\ge0$, every $P_5$-free graph $G$ contains either:
		\begin{itemize}
			\item an anticomplete pair $(A,B)$ with $\chi(A),\chi(B)\ge(1-2c^{2^{s+1}})\chi(G)$;
			
			\item a complete pair $(P,Q)$ with $\chi(P),\chi(Q)\ge\phi_s(c)\cdot c^4\chi(G)$;
			
			\item a complete pair $(X,Y)$ with
			$$\chi(X)\ge \phi_{r,s}(c)\cdot c^{2^{r+2}}\chi(G)
			\qquad\text{and}\qquad
			\chi(Y)\ge \phi_{r,s}(c)\cdot\left(1-3c^{2^{r}}\right)\chi(G)$$
			for some integer $r$ with $1\le r\le s$ (so $s\ge1$); or
			
			\item a $(c^{2^{r}},\chi)$-dense induced subgraph $F$ with $\chi(F)\ge \phi_{r,s}(c)\cdot 2c^{3\cdot2^{r}}\chi(G)$, for some integer $r$ with $0\le r\le s$.
		\end{itemize}
	\end{lemma}
	\begin{proof}
		We proceed by induction on $s\ge0$. The case $s=0$ follows from \cref{lem:1step} with $\eps=c$, whose first, second, and third outcomes verify, respectively, the first, second, and last outcomes of the lemma.
		Now, for $s\ge0$, we will prove the lemma for $s+1$ assuming that it is true for $s$ (and for all $c\in(0,2^{-9}]$ and all $P_5$-free graphs $G$).
		To this end, let $c\in(0,2^{-9}]$, let $G$ be a $P_5$-free graph and assume that the last three outcomes of the lemma (for $s+1$ in place of $s$) do not hold. Let $\eps:=c^{2^{s+1}}$, $p:=(1-3\eps)\chi(G)$, and $q:=(1-\eps^{2})\chi(G)$. 
		We claim that:
		\begin{claim}
			\label{claim:incrho}
			$G$ is $(p,q)$-sparse.
		\end{claim}
		\begin{subproof}
			Let $F$ be an induced subgraph of $G$ with $\chi(F)\ge q=(1-\eps^2)\chi(G)=(1-c^{2^{s+2}})\chi(G)$;
			then $\chi(F)\ge\frac12\chi(G)$.
			By induction for $s$ applied to $F$, $F$ contains either:
			\begin{itemize}
				\item an anticomplete pair $(A,B)$ with $\chi(A),\chi(B)\ge(1-2c^{2^{s+1}})\chi(F)=(1-2\eps)\chi(F)$;
				
				\item a complete pair $(P,Q)$ with $\chi(P),\chi(Q)\ge\phi_s(c)\cdot c^4\chi(F)$;
				
				\item a complete pair $(X,Y)$ with $$\chi(X)\ge\phi_{r,s}(c)\cdot c^{2^{r+2}}\chi(F)
				\qquad\text{and}\qquad
				\chi(Y)\ge\phi_{r,s}(c)\cdot\left(1-3c^{2^{r}}\right)\chi(F)$$
				for some integer $r$ with $1\le r\le s$; or
				
				\item a $(c^{2^{r}},\chi)$-dense induced subgraph $J$ with $\chi(J)\ge \phi_{r,s}(c)\cdot 2c^{3\cdot2^{r}}\chi(F)$, for some integer $r$ with $0\le r\le s$.
			\end{itemize}
			
			Because
			$$\phi_{r,s}(c)\cdot\chi(F)\ge \phi_{r,s}(c)\cdot\left(1-c^{2^{s+2}}\right)\chi(G)=\phi_{r,s+1}(c)\cdot\chi(G)$$
			for all integers $r$ with $0\le r\le s$,
			if one of the last three bullets holds then one of the last three outcomes of the lemma holds (respectively) for $s+1$, a contradiction.
			Thus the first bullet holds; and this proves \cref{claim:incrho}
			because
			\[(1-2\eps)\chi(F)\ge(1-2\eps)(1-\eps^2)\chi(G)\ge(1-3\eps)\chi(G)=p.
			\qedhere\]
		\end{subproof}
		
		Now, by \cref{lem:anti}, $G$ contains either:
		\begin{itemize}
			\item an anticomplete pair $(A,B)$ with $\chi(A)\ge q-2\eps^4\chi(G)$ and $\chi(B)\ge(1-\eps^2)\chi(G)$;
			
			\item a complete pair $(X,Y)$ with $\chi(X)\ge \eps^4\chi(G)$ and $\chi(Y)\ge p$; or
			
			\item an $(\eps,\chi)$-dense induced subgraph $J$ with $\chi(J)\ge2\eps^3\chi(G)$.
		\end{itemize}
		
		If the second bullet holds then \[\chi(X)\ge\eps^4\chi(G)=c^{2^{s+3}}\chi(G)
		\qquad\text{and}\qquad
		\chi(Y)\ge p=(1-3\eps)\chi(G)=\left(1-3c^{2^{s+1}}\right)\chi(G)\]
		and so the third outcome of the lemma holds for $s+1$ (with $r=s+1$), a contradiction.
		
		If the third bullet holds then the last outcome of the lemma holds with $r=s+1$, a contradiction.
		
		Hence the first bullet holds; then $\chi(B)\ge(1-\eps^2)\chi(G)=(1-c^{2^{s+2}})\chi(G)$ and
		\[\chi(A)\ge q-2\eps^4\chi(G)\ge (1-\eps^2-2\eps^4)\chi(G)\ge(1-2\eps^2)\chi(G)
		=\left(1-2c^{2^{s+2}}\right)\chi(G)\]
		and so the first outcome of the lemma holds for $s+1$. This proves \cref{lem:incchi}.
	\end{proof}
	We are now ready to prove \cref{lem:antino}, which we restate here for the reader's convenience.
	
	\begin{lemma}
		\label{lem:noanti}
		Let $a\ge2$ and $c\in(0,2^{-9}]$. Then every $P_5$-free graph $G$ with $\chi(G)\ge c^{-a}$ contains either:
		\begin{itemize}
			\item a complete pair $(P,Q)$ with $\chi(P),\chi(Q)\ge \frac12c^4\chi(G)$;
			
			\item a complete pair $(X,Y)$ with $\chi(X)\ge y^{5a}\chi(G)$ and $\chi(Y)\ge (1-y)\chi(G)$ for some $y\in(0,2c]$; or
			
			\item an $(\eps,\chi)$-dense induced subgraph $F$ with $\chi(F)\ge \eps^3\chi(G)$ for some $\eps\in[\chi(G)^{-1/a},c]$.
		\end{itemize}
	\end{lemma}
	\begin{proof}
		We may assume that $G$ is connected. Also, suppose that none of the outcomes holds.
		Let $w:=\chi(G)\ge c^{-a}$; then $w\ge 2^{9a}$. We claim that:
		\begin{claim}
			\label{claim:toobig}
			$G$ contains an anticomplete pair $(A,B)$ with $\chi(A),\chi(B)\ge (1-2w^{-1/a})\chi(G)$.
		\end{claim}
		\begin{subproof}
			Since $w^{1/a}\ge c^{-1}$, there is a maximal integer $s\ge0$ such that $c^{-2^{s}}\le w^{1/a}$; then $c^{-2^{s+1}}\ge w^{1/a}$ and so $2c^{2^{s+1}}\le 2w^{-1/a}$. For all $r\in\{0,1,\ldots,s\}$, since $c\le 2^{-9}$ we have
			\[\phi_{r,s}(c)=\prod_{r<i\le s}\left(1-c^{2^{i+1}}\right)
			\ge 1-\sum_{r<i\le s}c^{2^{i+1}}\ge 1-2c^{2^{r+2}}\ge1/2\]
			and in particular $\phi_s(c)=\phi_{0,s}(c)\ge 1/2$.
			
			Now, by \cref{lem:incchi}, $G$ contains either:
			\begin{itemize}
				\item an anticomplete pair $(A,B)$ with $\chi(A),\chi(B)\ge(1-2c^{2^{s+1}})\chi(G)$;
				
				\item a complete pair $(P,Q)$ with $\chi(P),\chi(Q)\ge\phi_s(c)\cdot c^4\chi(G)$;
				
				\item a complete pair $(X,Y)$ with
				$$\chi(X)\ge \phi_{r,s}(c)\cdot c^{2^{r+2}}\chi(G)
				\qquad\text{and}\qquad
				\chi(Y)\ge \phi_{r,s}(c)\cdot\left(1-3c^{2^{r}}\right)\chi(G)$$
				for some integer $r$ with $1\le r\le s$ (so $s\ge1$); or
				
				\item a $(c^{2^{r}},\chi)$-dense induced subgraph $F$ with $\chi(F)\ge \phi_{r,s}(c)\cdot 2c^{3\cdot2^{r}}\chi(G)$, for some integer $r$ with $0\le r\le s$.
			\end{itemize}
			
			If the second bullet holds then the first outcome of the lemma holds since
			$\phi_s(c)\cdot c^{4}\chi(G)\ge \frac12c^4\chi(G)$, contrary to our assumption.
			
			If the third bullet holds then since $r\ge1$, $a\ge 2$, and $c\le 2^{-9}$, we have
			\[\begin{aligned}
				\chi(X)&\ge \phi_{r,s}(c)\cdot c^{2^{r+2}}\chi(G)\ge (1/2)\cdot c^{2^{r+2}}\chi(G)\ge c^{5\cdot 2^{r}}\chi(G)\ge c^{5a\cdot 2^{r-1}}\chi(G),\quad\text{and}\\
				\chi(Y)&\ge\phi_{r,s}(c)\cdot\left(1-3\cdot c^{2^{r}}\right)\chi(G)\\
				&\ge\left(1-2c^{2^{r+2}}\right)\left(1-3\cdot c^{2^{r}}\right)\chi(G)
				\ge \left(1-4\cdot c^{2^{r}}\right)\chi(G)
				\ge \left(1-c^{2^{r-1}}\right)\chi(G)
			\end{aligned}\]
			and so the second outcome of the lemma holds with $y=c^{2^{r-1}}\le c$, contrary to our assumption.
			
			If the fourth bullet holds then for $\eps=c^{2^{r}}\ge c^{2^{s}}\ge \chi(G)^{-1/a}$ we have
			$$\chi(F)\ge\phi_{r,s}(c)\cdot 2c^{3\cdot2^{r}}\chi(G)\ge c^{3\cdot2^{r}}\chi(G)\ge\eps^{3}\chi(G)$$
			but then the third outcome of the lemma holds, a contradiction. 
			
			Hence, the first bullet above holds, proving \cref{claim:toobig}.
		\end{subproof}
		
		Now, \cref{claim:toobig} gives an anticomplete pair $(A,B)$ with $\chi(A),\chi(B)\ge (1-2w^{-1/a})\chi(G)$.
		Among all such pairs, choose one such that $G[A],G[B]$ are connected and $\abs A+\abs B$ is maximal.
		Since $G$ is connected, it has a minimal nonempty cutset $S$ separating $A,B$.
		By the maximality of $\abs A+\abs B$, $G[A],G[B]$ are two components of $G\setminus S$.
		Let $v\in S$. Since $G[A],G[B]$ are connected, $S$ is minimal, and $G$ is $P_5$-free, \cref{lem:mixed} implies that $v$ is complete to $A$ or $B$. Thus $\chi(N_G(v))\ge\min(\chi(A),\chi(B))\ge(1-2w^{-1/a})\chi(G)$.
		Then the second outcome of the lemma holds with $X=\{v\}$, $Y=N_G(v)$, and $y=2w^{-1/a}\le \min(w^{-1/(5a)},2c)$, a contradiction.
		This proves \cref{lem:noanti}.
	\end{proof}
	
	\section{Excluding an additional five-cycle}
	\label{sec:c5}
	Since the polynomial $\chi$-boundedness problem for $\{P_5,C_5\}$-free graphs remained open for some time, the current section provides a relatively short proof of this (provided \cref{lem:antino}).
	Much of the argument here also serves as a proof of concept for chromatic density increment, which hopefully will be a preparation for the remainder of the paper.
	The proof does not rely on any known \erh{} result and produces the explicit polynomial $\chi$-binding function $x\mapsto x^{40}$ (we make no attempt to optimise this).
	\begin{theorem}
		\label{thm:pc5}
		Every $\{P_5,C_5\}$-free graph $G$ satisfies $\chi(G)\le\omega(G)^{40}$.
	\end{theorem}
	A blockade $(B_1,\ldots,B_k)$ in a graph $G$ is {\em restrained} if $\omega(B_1)+\cdots+\omega(B_k)\le\omega(G)$.
	Thus complete blockades are restrained, and \cref{thm:pc5} is proved via the following analogue of \cref{thm:main}.
	
	\begin{theorem}
		\label{thm:mainc5}
		Every $\{P_5,C_5\}$-free graph $G$ with $\chi(G)\ge2$ contains either:
		\begin{itemize}
			\item a complete pair $(X,Y)$ with $\chi(X)\ge y^{20}\chi(G)$ and $\chi(Y)\ge (1-y)\chi(G)$ for some $y\in(0,\frac14)$; or
			
			\item a restrained blockade of length $k$ and mass at least $k^{-40}\chi(G)$, for some integer $k\ge2$.
		\end{itemize}
	\end{theorem}
	It is not hard to deduce \cref{thm:pc5} from \cref{thm:mainc5} in a manner similar to the deduction of \cref{thm:p5} from \cref{thm:main} given in \cref{sec:sketch}; here restrained blockades are as useful as their complete strengthening. Therefore we omit the detailed proof of the deduction.
	
	\cref{thm:mainc5} is proved via the following variant of \cref{lem:incre} for $\chi$-dense $\{P_5,C_5\}$-free graphs.
	
	\begin{lemma}
		\label{lem:ext}
		For every $\eps\in(0,2^{-9}]$ and every $(\eps,\chi)$-dense $\{P_5,C_5\}$-free graph $G$, there exist an integer  $k\ge\min(\eps^{-1/8},\chi(G)^{1/2})$ and a restrained blockade in $G$ of length $k$ and mass at least $k^{-16}\chi(G)$.
	\end{lemma}
	
	Provided this result, we can prove \cref{thm:mainc5} by means of \cref{lem:antino}, in a manner similar to the proof of \cref{thm:main} based on \cref{lem:antino,lem:incre} given in \cref{subsec:incre}, as follows.
	
	\begin{proof}
		[Proof of \cref{thm:mainc5}, assuming \cref{lem:ext}]
		If $\chi(G)\le 2^{40}$ then the second outcome holds with $k=2$, since $G$ has at least one edge. Thus we may assume $\chi(G)\ge2^{40}$.
		By \cref{lem:antino} with $a=4$ and $c=2^{-9}$ (then $\chi(G)\ge2^{40}\ge c^{-a}$), $G$ contains either:
		\begin{itemize}
			\item a complete pair $(P,Q)$ with $\chi(P),\chi(Q)\ge \frac12c^4\chi(G)=2^{-37}\chi(G)$;
			
			\item a complete pair $(X,Y)$ with $\chi(X)\ge y^{5a}\chi(G)=y^{20}\chi(G)$ and $\chi(Y)\ge (1-y)\chi(G)$ for some $y\in(0,2c]\subset(0,2^{-8}]$; or
			
			\item an $(\eps,\chi)$-dense induced subgraph $F$ with $\chi(F)\ge \eps^3\chi(G)$ for some $\eps\in[\chi(G)^{-1/4},c]$.
		\end{itemize}
		
		If the first or second bullet holds then the second or first outcome of the theorem holds, respectively.
		Thus we may assume that the third bullet holds.
		Since $\chi(F)\ge \eps^3\chi(G)\ge \eps^{-1}$,
		\cref{lem:ext} gives an integer  $k\ge\min(\eps^{-1/8},\chi(F)^{1/2})\ge\eps^{-1/8}\ge c^{-1/8}\ge2$ and a restrained blockade in $F$ of length $k$ and mass at least $k^{-16}\chi(F)\ge k^{-16}\eps^3\chi(G)\ge  k^{-40}\chi(G)$.
		This verifies the second outcome and in turn proves \cref{thm:mainc5}.
	\end{proof}
	
	The rest of this section is thus devoted to proving \cref{lem:ext} by chromatic density increment.
	We begin with the following simple partition lemma.
	
	\begin{lemma}
		\label{lem:blocks}
		Let $y\in(0,\frac15]$, let $G$ be a graph, and let $(A_1,\ldots,A_k)$ be a partition of $V(G)$ into nonempty subsets such that $\chi(A_i)< y\cdot\chi(G)$ for all $i\in[k]$.
		Then there are at least $y^{-1/2}$ disjoint subsets $I_1,\ldots,I_s$ of $[k]$ such that $\chi(\bigcup_{i\in I_j}A_i)\ge y\cdot\chi(G)$ for all $j\in[s]$.
	\end{lemma}
	\begin{proof}
		Let $s\ge1$ be minimal such that $[k]$ admits a partition $I_1\cup\cdots\cup I_s$ for which $\chi(\bigcup_{i\in I_j}A_i)\le 2y\cdot\chi(G)$ for all $j\in[s]$; for $s=k$ one can take $I_j=\{j\}$ for all $j\in[k]$.
		The minimality of $s$ implies that there is at most one $j\in[s]$ for which $\chi(\bigcup_{i\in I_j}A_i)\le y\cdot\chi(G)$; and so there are at least $\floor{(2y)^{-1}(1-y)}+1$ indices $j\in[s]$ with $\chi(\bigcup_{i\in I_j}A_i)\ge y\cdot \chi(G)$.
		Since $\floor{(2y)^{-1}(1-y)}+1\ge y^{-1/2}$ for $y\in(0,\frac15]$, this proves \cref{lem:blocks}.
	\end{proof}
	We now use the $\{P_5,C_5\}$-free hypothesis to obtain the following fact.
	Its proof is inspired by the resolution of the $C_5$-free case of the \erh{} conjecture \ref{conj:eh}~\cite{MR4563865}, and partly explains why a quick proof of \cref{lem:ext} in this section requires blockades that are only restrained (maximum cliques involved) but not fully complete.
	
	\begin{lemma}
		\label{lem:keyc5}
		Let $y\in(0,2^{-5}]$, and let $G$ be a $\{P_5,C_5\}$-free graph.
		Let $v\in V(G)$, $A\subset N_G(v)$, and $B\subset V(G)\setminus N_G[v]$, such that every $u\in B$ satisfies $\chi(A\setminus N_G(u))< y\cdot\chi(A)$. Then either:
		\begin{itemize}
			\item there exists $D\subset A$ such that $\chi(D)\ge(1-y^{1/2})\chi(A)$ and $(B,D)$ is restrained in $G$; or
			
			\item $G[A]$ contains a complete blockade of length $k$ and mass at least $k^{-4}\chi(A)$, for some integer $k\ge y^{-1/4}$. 
		\end{itemize}
	\end{lemma}
	\begin{proof}
		Let $K$ be a maximum clique in $G[B]$, and let $S$ be the set of vertices in $A$ with a nonneighbour in $K$.
		We may assume that $\chi(S)\ge y^{1/2}\chi(A)$, for otherwise the first outcome of the lemma holds with $D=A\setminus S$ because $K$ is complete to $A\setminus S$.
		Let $\ell\ge1$ be minimal such that there is a subset $\{v_1,\ldots,v_\ell\}$ of $K$ for which every vertex in $S$ has a nonneighbour in this set.
		For every $i\in[\ell]$, let $A_i$ be the set of vertices in $S$ nonadjacent to $v_i$ and adjacent to every vertex in $v_1,\ldots,v_{i-1}$.
		Then $(A_1,\ldots,A_\ell)$ is a partition of $S$; and for all $i\in[\ell]$, $v_i$ is anticomplete to $A_i$ and complete to $A_{i+1}\cup\cdots\cup A_\ell$. The minimality of $\ell$ implies that $A_1,\ldots,A_\ell$ are nonempty.
		Thus, since $\chi(A_i)< y\cdot\chi(A)\le y^{1/2}\chi(S)$ for all $i\in[\ell]$,
		\cref{lem:blocks} (with $y^{1/2}\in(0,\frac15]$ replacing $y$) gives at least $y^{-1/4}$ disjoint subsets $I_1,\ldots,I_s$ of $[\ell]$ such that, for $B_j=\bigcup_{i\in I_j}A_i$ for each $j\in[s]$, we have
		$$\chi(B_j)\ge y^{1/2}\chi(S)\ge y\cdot \chi(A)\ge s^{-4}\chi(A)\quad\text{for all $j\in [s]$.}$$
		
		It suffices to show that $(B_{j_1},B_{j_2})$ is complete in $G$ for all distinct $j_1,j_2\in [s]$.
		Suppose not; then $G[S]$ has a nonedge $u_1u_2$ with $u_1\in B_{j_1}$ and $u_2\in B_{j_2}$.
		Let $i_1\in I_{j_1}$ and $i_2\in I_{j_2}$ satisfy $u_1\in A_{i_1}$ and $u_2\in A_{i_2}$; and we may assume $i_1<i_2$.
		Then, since $v_{i_1}$ is adjacent to $u_{2}$ in $G$, $G[\{u_{1},v,u_{2},v_{i_1},v_{i_2}\}]$ would be isomorphic to $P_5$ (if $u_{1},v_{i_2}$ are nonadjacent) or $C_5$ (if $u_{1},v_{i_2}$ are adjacent), a contradiction.
		This proves \cref{lem:keyc5}.
	\end{proof}
	
	Now we iterate \cref{lem:keyc5} to deduce that in every $\chi$-dense $\{P_5,C_5\}$-free graph, either there is a much $\chi$-denser induced subgraph with linear chromatic number, or there exists a restrained blockade that is polynomially balanced.
	
	\begin{lemma}
		\label{lem:increc5}
		For every $y\in(0,2^{-9}]$, every $(y,\chi)$-dense $\{P_5,C_5\}$-free graph $G$ contains either:
		\begin{itemize}
			\item a $(y^2,\chi)$-dense induced subgraph with chromatic number at least $\frac18\chi(G)$; or
			
			\item a restrained blockade of length $k$ and mass at least $k^{-8}\chi(G)$, for some integer $k\ge y^{-1/8}$.
		\end{itemize}
	\end{lemma}
	\begin{proof}
		Let $k\ge0$ be maximal such that $G$ contains a restrained blockade $(B_0,B_1,\ldots,B_k)$ with
		$$\chi(B_k)\ge(1-y^{1/4})^k\chi(G),\quad\text{and}\quad\chi(B_{i-1})\ge 2^{-4}y^2\chi(G)\quad\text{ for all $i\in[k]$.}$$
		
		If $k\ge 2^{1/2}y^{-1/4}$ then the second outcome of the lemma holds since $2^{-4}y^2= (2^{1/2}y^{-1/4})^{-8}\ge k^{-8}$, and we are done.
		Thus we may assume $k<2^{1/2}y^{-1/4}$, which yields
		\[\chi(B_k)\ge(1-y^{1/4})^k\chi(G)\ge 4^{-ky^{1/4}}\chi(G)\ge 4^{-2^{1/2}}\chi(G)\ge\chi(G)/8.\]
		If $G[B_k]$ is $(y^2,\chi)$-dense then the first outcome of the lemma holds; and so we may assume there exists $v\in B_k$ with $\chi(B_k\setminus N_G[v])\ge y^2\chi(B_k)\ge 2^{-4}y^2\chi(G)$.
		Let $A:=B_k\cap N_G(v)$ and $B:=B_k\setminus N_G[v]$.
		Then $\chi(B)\ge 2^{-4}y^2\chi(G)$; and so $B$ is nonempty, which yields $\chi(B)=\chi(B_k\setminus N_G(v))$.
		Also, since $G$ is $(y,\chi)$-dense, we have $\chi(B)\le y\cdot\chi(G)\le 8y\cdot\chi(B_k)$, which implies
		\begin{align*}
			\chi(A)&\ge \chi(B_k)-\chi(B_k\setminus N_G(v))
			=\chi(B_k)-\chi(B)
			\ge (1-8y)\chi(B_k)
			\ge \chi(B_k)/2
			\ge \chi(G)/16.
		\end{align*}
		Hence $\chi(A\setminus N_G(u))< y\cdot\chi(G)\le 16y\cdot\chi(A)$ for all $u\in B$.
		Therefore, \cref{lem:keyc5} (with $16y\in(0,2^{-5}]$ in place of $y$) implies that either:
		\begin{itemize}
			\item there exists $D\subset A$ such that $\chi(D)\ge (1-(16y)^{1/2})\chi(A)$ and $(B,D)$ is restrained in $G[B_k]$; or
			
			\item $G[A]$ contains a complete blockade of length $\ell$ and mass at least $\ell^{-4}\chi(A)$,
			for some integer $\ell\ge(16y)^{-1/4}$.
		\end{itemize}
		
		If the first bullet holds, then since $y\in(0,2^{-9}]$, we have
		\begin{align*}
			\chi(D)\ge(1-(16y)^{1/2})\chi(A)
			&\ge (1-(16y)^{1/2})(1-8y)\chi(B_k)\\
			&\ge (1-(16y)^{1/2}-8y)\chi(B_k)
			\ge (1-y^{1/4})\chi(B_k)
			\ge (1-y^{1/4})^{k+1}\chi(G)
		\end{align*}
		but then $(B_0,B_1,\ldots,B_{k-1},B,D)$ would violate the maximality of $k$ because
		$$\sum_{i\in[k]}\omega(B_{i-1})+\omega(B)+\omega(D)\le\sum_{i\in[k]}\omega(B_{i-1})+\omega(B_k)\le\omega(G).$$
		Hence the second bullet holds, which yields the second outcome of the lemma since $\ell\ge (16y)^{-1/4}\ge y^{-1/8}\ge2$ and $\ell^{-4}\chi(A)\ge 2^{-4}\ell^{-4}\chi(G)\ge \ell^{-8}\chi(G)$.
		This proves \cref{lem:increc5}.
	\end{proof}
	
	We also require the following simple property of $(\eps,\chi)$-dense graphs.
	
	\begin{lemma}
		\label{lem:dense}
		For every $\eps>0$, if $G$ is $(\eps,\chi)$-dense, then $\omega(G)\ge\min(\eps^{-1},\chi(G))$.
	\end{lemma}
	\begin{proof}
		Let $K$ be a maximum clique in $G$; then every vertex in $V(G)\setminus K$ has a nonneighbour in $K$.
		Thus, since $\chi(G\setminus N_G(v))=\max(1,\chi(G\setminus N_G[v]))\le \max(1,\eps\cdot\chi(G))$ for all $v\in V(G)$, we have $\chi(G)\le \sum_{v\in K}\chi(G\setminus N_G(v))\le \abs K\cdot \max(1,\eps\cdot\chi(G))$.
		Hence either $\chi(G)\le \abs K$ or $\chi(G)\le \abs K\cdot \eps\cdot\chi(G)$, and the latter case yields $\abs K\ge\eps^{-1}$.
		This proves \cref{lem:dense}.
	\end{proof}
	
	By combining \cref{lem:dense,lem:increc5}, we can run chromatic density increment on $\chi$-dense $\{P_5,C_5\}$-free graphs and prove \cref{lem:ext} in turn, as follows.
	
	\begin{proof}
		[Proof of \cref{lem:ext}]
		If $\chi(G)\le\eps^{-2}$ then $\omega(G)\ge \min(\eps^{-1},\chi(G))\ge\chi(G)^{1/2}$ by \cref{lem:dense}; and so the lemma holds with $k=\omega(G)$ and $\abs{B_i}=\min(1,\chi(G))$ for all $i\in[k]$ (so $B_1\cup\cdots\cup B_k$ is a maximum clique of $G$).
		Thus we may assume $\chi(G)>\eps^{-2}$.
		This allows us to set up $\chi$-density increment, by choosing a minimal $y\in[\chi(G)^{-1/2},\eps]$ of the form $y=\eps^{2^s}$ for some integer $s\ge0$, such that $G$ has a $(y,\chi)$-dense induced subgraph $F$ with $\chi(F)\ge y\cdot\chi(G)$.
		Thus, \cref{lem:dense} implies 
		$$\omega(F)\ge\min(y^{-1},\chi(F))\ge \min(y^{-1},y\cdot \chi(G))\ge \min(\eps^{-1},\chi(G)^{1/2}).$$
		By \cref{lem:increc5}, $F$ contains either:
		\begin{itemize}
			\item a $(y^2,\chi)$-dense induced subgraph $J$ with $\chi(J)\ge\frac1{8}\chi(F)\ge \frac18y\cdot \chi(G)\ge y^2\chi(G)$; or
			
			\item a restrained blockade of length $k$ and mass at least $k^{-8}\chi(F)$, for some integer $k\ge y^{-1/8}$.
		\end{itemize}
		
		If the first bullet holds then $y\le \chi(G)^{-1/4}$ by the minimality of $y$; and so the lemma holds with $k=\omega(F)\ge\min(\eps^{-1},\chi(G)^{1/2})$ and $\abs{B_i}=1\ge y^{4}\chi(G)\ge k^{-4}\chi(G)$ for all $i\in[k]$.
		If the second bullet holds then the lemma also holds because $k\ge y^{-1/8}\ge \eps^{-1/8}$ and $k^{-8}\chi(F)\ge k^{-8}\cdot y\cdot\chi(G)\ge k^{-16}\chi(G)$.
		This proves \cref{lem:ext}.
	\end{proof}
	
	\section{Chromatic density increment, first round}
	\label{sec:1incre}
	The rest of this paper is devoted to the proof of \cref{lem:incre}.
	This section executes the first round of chromatic density increment in a manner similar to~\cite[Section 5]{density7}. The main result of the section is the following chromatic analogue of~\cite[Lemma 5.1]{density7}.
	
	\begin{lemma}
		\label{lem:incre1}
		There exists $a\ge2$ such that for every $x\in(0,2^{-a}]$ and for every $P_5$-free graph $G$ with $\chi(G)\ge x^{-a}$, there exist an integer $k\in[2,x^{-1}]$ and a pure or $(x,\chi)$-dense blockade in $G$ of length $k$ and mass at least $k^{-a}\chi(G)$.
	\end{lemma}
	
	As mentioned in \cref{subsec:incre}, to prove \cref{lem:incre1} we will need to handle our $\chi$-dense $P_5$-free graphs with more care and circumvent technical issues involving counting. The following simple lemma will be useful for this purpose.
	
	\begin{lemma}
		\label{lem:avgalt}
		Let $r>0$, and let $G$ be a graph with nonempty disjoint $A,B\subset V(G)$, such that $\chi(A\setminus N_G(v))\le r$ for all $v\in B$. Then for every $y\in(0,1)$, either:
		\begin{itemize}
			\item there are $X\subset A$ and $Y\subset B$ with $\chi(X)\ge \chi(A)-r$ and $\chi(Y)\ge y^2\chi(B)$, such that $X$ is $(y,\chi)$-dense to $Y$; or
			
			\item there exists nonempty $D\subset A$ such that the set $E$ of vertices in $B$ with no neighbour in $D$ satisfies $y^2\chi(B)\le\chi(E)\le y\cdot\chi(B)$.
		\end{itemize}
	\end{lemma}
	\begin{proof}
		Assume that the first outcome fails.
		Then there exists $u\in A$ with $\chi(B\setminus N_G(u))\ge y\cdot\chi(B)$.
		Hence there exists $D\subset A$ maximal such that $\chi(E)\ge y^2\chi(B)$ where $E=\bigcap_{u\in D}(B\setminus N_G(u))$.
		In particular $E$ is nonempty; and so the hypothesis implies $\chi(D)\le r$, which gives $\chi(A\setminus D)\ge\chi(A)-r$.
		Also, the maximality of $D$ yields $\chi(E\setminus N_G(u))\le y^2\chi(B)$ for all $u\in A\setminus D$; and so $\chi(E)\le y\cdot\chi(B)$ since the first outcome of the lemma fails.
		This proves \cref{lem:avgalt}.
	\end{proof}
	We now carry out the first main step of this section. We are given a $(y,\chi)$-dense $P_5$-free graph $G$ for some sufficiently small `floating' chromatic density parameter $y$. If $G$ is $(y^2,\chi)$-dense then our increment step proceeds; so we may assume there is some $v\in V(G)$ whose nonneighbourhood in $G$ has chromatic number at least $y^2\chi(G)$. We would like to discard a low-$\chi$ portion of $N_G(v)$ and shrink $V(G)\setminus N_G[v]$ by a chromatic factor of $\operatorname{poly}(y)$ so that the new neighbourhood of $v$ is $(\operatorname{poly}(y),\chi)$-dense to its new nonneighbourhood.
	If this never occurs, we iterate \cref{lem:avgalt} and chase induced five-vertex paths to extract appropriate pure blockades within $V(G)\setminus N_G[v]$.
	This is done rigorously by the following.
	
	\begin{lemma}
		\label{lem:avgp5}
		Let $r>0$, let $y\in(0,\frac12]$, and let $q$ be an integer with $1\le q\le \frac12y^{-1}$.
		Let $G$ be a $P_5$-free graph with $v\in V(G)$, and let $A\subset N_G(v)$ and $B\subset V(G)\setminus N_G[v]$ satisfy $\chi(A)\ge qr$ and $\chi(A\setminus N_G(u))< r$ for all $u\in B$.
		Then either:
		\begin{itemize}
			\item there are $X\subset A$ and $Y\subset B$ with $\chi(X)\ge \chi(A)-qr$ and $\chi(Y)\ge \frac12y^2\chi(B)$, such that $X$ is $(y,\chi)$-dense to $Y$; or
			
			\item $G[B]$ contains a pure blockade of length $q$ and mass at least $\frac12y^2\chi(B)$.
		\end{itemize}
	\end{lemma}
	\begin{proof}
		Assume that the first outcome fails, and that $B$ is nonempty (else the second outcome holds).
		Let $k\ge0$ be maximal for which there are nonempty disjoint $D_1,\ldots,D_k\subset A$ and $B_1,\ldots,B_k\subset B$ such that for all $i\in[k]$:
		\begin{itemize}
			\item $\frac12y^2\chi(B)\le\chi(B_i)\le y\cdot\chi(B)$;
			
			\item $D_i$ is the set of vertices in $A\setminus(D_1\cup\cdots\cup D_{i-1})$ with no neighbour in $B_i$; and
			
			\item $G[B_i]$ is a maximal component of $G[E_i]$, where $E_i$ is the set of vertices in $B\setminus(B_1\cup\cdots\cup B_{i-1})$ with no neighbour in $D_i$.
		\end{itemize}
		(These conditions are clearly satisfied for $k=0$.) We claim that:
		
		\begin{claim}
			\label{claim:avgp51}
			$k\ge q$.
		\end{claim}
		\begin{subproof}
			Suppose not.
			Let $A':=A\setminus(D_1\cup\cdots\cup D_k)$ and $B':=B\setminus(B_1\cup\cdots\cup B_k)$.
			Because $k<q\le\frac12y^{-1}$ and $\chi(D_i)\le r$ for all $i\in[k]$, we have $\chi(A')\ge \chi(A)-kr>\chi(A)-qr\ge0$ (so $A'$ is nonempty); and
			$$\chi(B')\ge \chi(B)-k\cdot y\cdot\chi(B)\ge (1-qy)\chi(B)\ge\chi(B)/2.$$
			Thus, by \cref{lem:avgalt} applied to $(A',B')$, either:
			\begin{itemize}
				\item there are $X\subset A'$ and $Y\subset B'$ with $\chi(X)\ge\chi(A')-r$ and $\chi(Y)\ge y^2\chi(B')$, such that $X$ is $(y,\chi)$-dense to $Y$; or
				
				\item there exists nonempty $D\subset A'$ such that the set $E$ of vertices in $B'$ with no neighbour in $D$ satisfies $y^2\chi(B')\le\chi(E)\le y\cdot\chi(B')$.
			\end{itemize}
			
			The first bullet cannot hold because $\chi(A')-r\ge\chi(A)-(k+1)r\ge \chi(A)-qr$ (as $k+1\le q$ by our supposition), $y^2\chi(B')\ge\frac12y^2\chi(B)$, and the first outcome of the lemma fails.
			Thus the second bullet holds.
			Let $E'\subset E$ satisfy $G[E']$ is a maximal component of $G[E]$, and let $D'$ be the set of vertices in $A'=A\setminus(D_1\cup\cdots\cup D_k)$ with no neighbour in $E'$.
			Then $D\subset D'$ and $G[E']$ is a maximal component of $G[E'']$, where $E''$ is the set of vertices in $B'$ with no neighbour in $D'$ (note that $E'\subset E''\subset E$).
			But then, since $\chi(E')=\chi(E)\ge y^2\chi(B')\ge\frac12y^2\chi(B)$, we see that $(D_1,\ldots,D_k,D')$ and $(B_1,\ldots,B_k,E')$ would violate the maximality of $k$.
			This proves \cref{claim:avgp51}.
		\end{subproof}
		
		Now, we use the $P_5$-freeness of $G$ to deduce that $(B_1,\ldots,B_k)$ is pure. In what follows, for disjoint $X,Y\subset V(G)$, we say that $X$ is {\em pure} to $Y$ if every vertex in $X$ is pure to $Y$ in $G$; thus $(X,Y)$ is a pure pair in $G$ if and only if $X$ is pure to $Y$ and $Y$ is pure to $X$.
		
		\begin{figure}[ht]
			\centering
			
			\begin{tikzpicture}[scale=0.7,auto=left]
				\draw [] (-1.5,0) circle (1cm);
				
				\node [circle,draw] () at (-1.5,-4.25) [minimum size=1.65cm] {};
				\node [circle,draw] () at (2,-4.25) [minimum size=1.65cm] {};
				
				\draw [rounded corners] (-4,-2) rectangle (4.5,2);
				\draw [rounded corners] (-4,-6.25) rectangle (4.5,-2.25);
				\node[] at (-4.75,0) {$A$};
				\node[] at (-4.75,-4.25) {$B$};
				
				\node[] at (-3.35,-4.3) {$B_i$};
				\node[] at (-3,-0.05) {$D_i$};
				\node[] at (3.85,-4.3) {$B_j$};
				
				\node[inner sep=1.5pt, fill=black,circle,draw] (v) at ({0.25},{3.5}) {};
				\node[xshift=0.3cm] at (v) {$v$};
				\node[inner sep=1.5pt, fill=black,circle,draw] (u) at ({2},{-4.25}) {};
				\node[xshift=0.3cm] at (u) {$u$};
				\node[inner sep=1.5pt, fill=black,circle,draw] (z) at ({-1.5},{0}) {};
				\node[xshift=-0.3cm] at (z) {$z$};
				\node[inner sep=1.5pt, fill=black,circle,draw] (v') at ({-2},{-4.75}) {};
				\node[xshift=-0.2cm,yshift=0.25cm] at (v') {$v'$};
				\node[inner sep=1.5pt, fill=black,circle,draw] (u') at ({-1},{-3.75}) {};
				\node[xshift=-0.3cm,yshift=0.15cm] at (u') {$u'$};
				\draw[-] (v) -- (z) -- (u) -- (u') -- (v');
				\draw[dashed] (u') -- (z) -- (v') -- (u);

				\draw [] (14,0) circle (1cm);
				
				\node [circle,draw] (bi) at (10.5,-4.25) [minimum size=1.65cm] {};
				\node [circle,draw] () at (14,-4.25) [minimum size=1.65cm] {};
				
				\draw [rounded corners] (8,-2) rectangle (16.5,2);
				\draw [rounded corners] (8,-6.25) rectangle (16.5,-2.25);
				\node[] at (7.25,0) {$A$};
				\node[] at (7.25,-4.25) {$B$};
				\node[color=white] at (17.25,0) {$A$};
				\node[color=white] at (17.25,-4.25) {$B$};
				
				\node[] at (8.65,-4.3) {$B_i$};
				\node[] at (15.85,-4.3) {$B_j$};
				\node[] at (15.5,-0.05) {$D_j$};
				
				\node[inner sep=1.5pt, fill=black,circle,draw] (v0) at ({12.25},{3.5}) {};
				\node[xshift=-0.3cm] at (v0) {$v$};
				\node[inner sep=1.5pt, fill=black,circle,draw] (u0) at ({10.5},{-4.25}) {};
				\node[xshift=-0.3cm,yshift=0.075cm] at (u0) {$u''$};
				\node[inner sep=1.5pt, fill=black,circle,draw] (z0) at ({14},{0}) {};
				\node[xshift=0.3cm] at (z0) {$z$};
				\node[inner sep=1.5pt, fill=black,circle,draw] (v'0) at ({14.5},{-4.75}) {};
				\node[xshift=0.25cm,yshift=0.25cm] at (v'0) {$v'$};
				\node[inner sep=1.5pt, fill=black,circle,draw] (u'0) at ({13.5},{-3.75}) {};
				\node[xshift=0.325cm,yshift=0.15cm] at (u'0) {$u'$};
				\draw[-] (v0) -- (z0) -- (u0) -- (u'0) -- (v'0);
				\draw[dashed] (u'0) -- (z0) -- (v'0) -- (u0);
				
				\draw[-] (tangent cs:node=bi,point={(u'0)},solution=1) -- (u'0) -- (tangent cs:node=bi,point={(u'0)},solution=2);
				\draw[dashed] (tangent cs:node=bi,point={(u'0)},solution=1) -- (v'0) -- (tangent cs:node=bi,point={(v'0)},solution=2);
			\end{tikzpicture}
			
			\caption{Proof of \cref{claim:avgp52}.}
			\label{fig:avgp52}
		\end{figure}
		
		\begin{claim}
			\label{claim:avgp52}
			$(B_1,\ldots,B_k)$ is a pure blockade in $G[B]$.
		\end{claim}

		\begin{subproof}
			It suffices to show that $(B_i,B_j)$ is pure for all $1\le i<j\le k$.
			First, suppose that there exists $u\in B_j$ mixed on $B_i$.
			Since $G[B_i]$ is connected, it has an edge $u'v'$ such that $u$ is adjacent to $u'$ and nonadjacent to $v'$.
			If $u\in B_j$ has a neighbour $z\in D_i$
			then $v\text-z\text-u\text-u'\text-v'$ would be an induced $P_5$ in $G$, a contradiction. (See the left-hand side of \cref{fig:avgp52}.)
			Thus $u$ has no neighbour in $D_i$ and so $u\in E_i$; but then $B_i\cup u$ contradicts that $G[B_i]$ is a component of $G[E_i]$.
			Hence $B_j$ is pure to $B_i$.
			
			Now, suppose that there exists $u\in B_i$ mixed on $B_j$; then $G[B_j]$ has an edge $u'v'$ such that $u$ is adjacent to $u'$ and nonadjacent to $v'$.
			Then, since $B_j$ is pure to $B_i$, $u'$ is complete to $B_i$ and $v'$ is anticomplete to $B_i$.
			Because $D_j$ is nonempty, there exists $z\in D_j$.
			Since $D_i$ is the set of vertices in $A\setminus(D_1\cup\cdots\cup D_{i-1})$ with no neighbour in $B_i$, we see that $z\in D_j$ has a neighbour $u''\in B_i$.
			But then $v\text-z\text-u''\text-u'\text- v'$ would be an induced $P_5$ in $G$, a contradiction. (See the right-hand side of \cref{fig:avgp52}.)
			Hence $B_i$ is pure to $B_j$.
			This proves \cref{claim:avgp52}.
		\end{subproof}
		\cref{claim:avgp51,claim:avgp52} together complete the proof of \cref{lem:avgp5}.
	\end{proof}
	
	Now we are given a `chromatic density threshold' $x$ which acts as a fixed lower bound for our original floating chromatic density parameter $y$.
	For $v\in V(G)$ with $y^2\chi(G)\le\chi(G\setminus N_G[v])<y\cdot\chi(G)$, assume that we have obtained suitable $A\subset N_G(v)$ and $B\subset V(G)\setminus N_G(v)$ with $\chi(A)\ge(1-\operatorname{poly}(y))\chi(G)$ and $\chi(B)\ge \operatorname{poly}(y)\cdot \chi(G)$ such that $A$ is $(\operatorname{poly}(y),\chi)$-dense to $B$.
	We will combine a greedy `anticovering' argument (similar to but more painstaking than the proof of \cref{lem:keyc5}) with more induced $P_5$ chases, to either produce another appropriate pure blockade within $G[B]$, or remove a few vertices from $A$ (resulting in $X$ say) and further shrink $B$ by a chromatic factor of half (resulting in $Y$ say) so that $X$ is $(x,\chi)$-dense to $Y$ (a highly desirable outcome).
	This is done by the following lemma, which is the second step of the current section.
	
	\begin{lemma}
		\label{lem:1dense1}
		Let $x,y\in(0,2^{-8}]$ with $x\le y$, and let $G$ be a $P_5$-free graph. Let $v\in V(G)$, and let $A\subset N_G(v)$ and $B\subset V(G)\setminus N_G[v]$.
		If $\chi(A)\ge 2x^{-1}$ and $A$ is $(y,\chi)$-dense to $B$, then either:
		\begin{itemize}
			\item there are $X\subset A$ and $Y\subset B$ with $\chi(X)\ge\chi(A)-2x^{-1}$ and $\chi(Y)\ge\frac12\chi(B)$, such that $X$ is $(x,\chi)$-dense to $Y$; or
			
			\item $G[B]$ contains a pure blockade of length $k$ and mass at least $k^{-2}\chi(B)$, for some integer $k\in[y^{-1/2},x^{-2}]$.
		\end{itemize}
	\end{lemma}
	\begin{proof}
		Assume that the second outcome fails.
		Let $v_1\in A$ satisfy $\chi(B\setminus N_G(u))\le\chi(B\setminus N_G(v_1))$ for all $u\in A$;
		and inductively for $j\ge1$, assuming that $v_1,\ldots,v_j$ have been defined, we choose $v_{j+1}\in A\setminus\{v_1,\ldots,v_j\}$ such that the set of nonneighbours of $v_{j+1}$ in $\bigcap_{i\in[j]}N_G(v_i)\cap B$ has maximum chromatic number among all choices of $u\in A\setminus\{v_1,\ldots,v_j\}$; that is, for all such~$u$,
		$$\chi\,\Bigg(\Bigg(\bigcap_{i\in[j]}N_G(v_i)\cap B\Bigg)\setminus N_G(u)\Bigg)\le \chi\,\Bigg(\Bigg(\bigcap_{i\in[j]}N_G(v_i)\cap B\Bigg)\setminus N_G(v_{j+1})\Bigg).$$
		This gives an ordering $(v_1,\ldots,v_s)$ of $A$, where $s=\abs A$.
		For each $j\in[s]$, let $E_j$ be the set of nonneighbours of $v_j$ in $\bigcap_{i\in[j-1]}N_G(v_i)\cap B$,
		where we adopt the convention that $\bigcap_{i\in[j-1]}N_G(v_i)\cap B=B$ when $j=1$.
		Since $A$ is $(y,\chi)$-dense to $B$, it follows that
		$$y\cdot \chi(B)\ge\chi(E_1)\ge\chi(E_2)\ge\ldots\ge \chi(E_s).$$
		
		If $\chi(E_1)<x^2\chi(B)$ then the first outcome of the lemma holds with $X=A$ and $Y=B$;
		and so we may assume that $\chi(E_1)\ge x^2\chi(B)$.
		Hence there exists $\ell\ge1$ maximal such that $ \chi(E_\ell)\ge x^2\chi(B)$.
		For each $i\in[\ell]$, let $B_i\subset E_i$ satisfy $G[B_i]$ is a maximal component of $G[E_i]$.
		We now use the $P_5$-free hypothesis to show that $(B_1,\ldots,B_\ell)$ is pure in $G$, in a manner similar to \cref{claim:avgp52}.
		This is also where our greedy selection of $(v_1,\ldots,v_s)$ comes into play.
		\begin{claim}
			\label{claim:1dense11}
			$(B_1,\ldots,B_\ell)$ is a pure blockade in $G$.
		\end{claim} 
		\begin{subproof}
			It suffices to show that $(B_i,B_j)$ is pure for all $1\le i<j\le\ell$.
			First, suppose that there exists $u\in B_j$ mixed on $B_i$. Then $G[B_i]$ has an edge $u'v'$ such that $u$ is adjacent to $u'$ and nonadjacent to $v'$; and so $v\text-v_i\text-u\text-u'\text-v'$ would be an induced $P_5$ in $G$, a contradiction (See the left-hand side of \cref{fig:avgp52} for a reference, with $v_i$ replacing $z$.). Hence $B_j$ is pure to $B_i$.
			
			Now, suppose that there exists $u\in B_i$ mixed on $B_j$; then $G[B_j]$ has an edge $u'v'$ such that $u$ is adjacent to $u'$ and nonadjacent to $v'$. Then, since $B_j$ is pure to $B_i$, $u'$ is complete to $B_i$ and $v'$ is anticomplete to $B_i$.
			If $v_j$ has a neighbour $u''\in B_i$, then $v\text-v_j\text-u''\text-u'\text-v'$ would be an induced $P_5$ in $G$, a contradiction. (See the right-hand side of \cref{fig:avgp52} for a reference, with $v_j$ replacing $z$.)
			Hence $v_j$ has no neighbour in $B_i$.
			It follows that the set of nonneighbours of $v_j$ in $\bigcap_{p\in[i-1]}N_G(v_p)\cap B$ contains $B_i\cup\{u'\}$, and so has chromatic number at least
			\[\chi(B_i\cup\{u'\})>\chi(B_i)=\chi(E_i).\]
			Since $E_i$ is the set of nonneighbours of $v_i$ in $\bigcap_{p\in[i-1]}N_G(v_p)\cap B$,
			this violates the choice of $(v_1,\ldots,v_s)$.
			Thus $B_i$ is pure to $B_j$. This proves \cref{claim:1dense11}.
		\end{subproof}
		
		Given \cref{claim:1dense11}, a simple dyadic partitioning argument now shows that both $\ell$ and $\chi(E_1\cup\cdots\cup E_\ell)$ cannot be large, as follows.
		\begin{claim}
			\label{claim:1dense12}
			$\ell\le 2x^{-1}$ and $\chi(E_1\cup\cdots\cup E_\ell)\le \frac12\chi(B)$.
		\end{claim}
		\begin{subproof}
			Let $q:=\ceil{\log\frac1x-\frac12\log \frac1y}\ge1$. For every $p\in[q]$, let $$I_p:=\{j\in[\ell]:2^{2p-2}x^2\chi(B)\le\chi(E_j)< 2^{2p}x^2\chi(B)\}$$
			then $[\ell]=I_1\cup\cdots\cup I_q$.
			For each $p\in[q]$, note that
			$$2^{1-p}x^{-1}\ge 2^{1-q}x^{-1}\ge 2^{\frac12\log\frac1y-\log\frac1x}x^{-1}=y^{-1/2}.$$
			Thus, by \cref{claim:1dense11} and since $\chi(B_j)=\chi(E_j)$ for all $j\in[\ell]$, if there exists $p\in[q]$ with $\abs{I_p}\ge 2^{1-p}x^{-1}$ then the second of the lemma holds with $k=\ceil{2^{1-p}x^{-1}}\le\ceil{x^{-1}}\le 2x^{-1}\le x^{-2}$, contrary to our initial assumption.
			Therefore $\abs{I_p}\le 2^{1-p}x^{-1}$ for all $p\in[q]$. It follows that
			\[
			\ell=\sum_{p\in[q]}\abs{I_p}\le \sum_{p\in[q]}2^{1-p}x^{-1}\le 2x^{-1},\]
			and since $2^q\le 2^{1+\log\frac1x-\frac12\log\frac1y}=2x^{-1}y^{1/2}\le \frac18x^{-1}$ because $y\le 2^{-8}$, we also have
			\begin{align*}
				\chi(E_1\cup\cdots\cup E_\ell)
				\le \sum_{p\in[q]}\sum_{j\in I_p}\chi(E_j)
				&\le\sum_{p\in[q]}2^{1-p}x^{-1}\cdot 2^{2p}x^2\chi(B)\\
				&=x\cdot\chi(B)\sum_{p\in[q]}2^{1+p}
				\le 2^{2+q}x\cdot\chi(B)
				\le\chi(B)/2.
			\end{align*}
			
			This proves \cref{claim:1dense12}.
		\end{subproof}
		Now, let $X:=A\setminus\{v_1,\ldots,v_\ell\}$ and $Y:=B\setminus(E_1\cup\cdots\cup E_\ell)$.
		\cref{claim:1dense12} yields $\chi(X)\ge\chi(A)-2x^{-1}$ and $\chi(Y)\ge\frac12\chi(B)$; and
		the maximality of $\ell$ implies that every $u\in X$ satisfies
		$$\chi(Y\setminus N_G(u))< x^2\chi(B)\le 2x^2\chi(Y)\le x\cdot\chi(Y).$$
		Thus $X$ is $(x,\chi)$-dense to $Y$, verifying the first outcome of the lemma.
		This proves \cref{lem:1dense1}.
	\end{proof}
	
	We next choose an appropriate $q$ in \cref{lem:avgp5} and combine it with \cref{lem:1dense1} to deduce:
	
	\begin{lemma}
		\label{lem:1dense2}
		Let $r>0$, let $x,y\in(0,2^{-8}]$ with $x\le y$, and let $G$ be a $P_5$-free graph with $v\in V(G)$.
		Let $A\subset N_G(v)$ and $B\subset V(G)\setminus N_G[v]$ satisfy $\chi(A)\ge 2y^{-1/2}r+2x^{-1}$ and $\chi(A\setminus N_G(u))<r$ for all $u\in B$.
		Then either:
		\begin{itemize}
			\item there are $X\subset A$ and $Y\subset B$ with $\chi(X)\ge\chi(A)-2y^{-1/2}r-2x^{-1}$ and $\chi(Y)\ge\frac14y^2\chi(B)$, such that $X$ is $(x,\chi)$-dense to $Y$; or
			
			\item $G[B]$ contains a pure blockade of length $k$ and mass at least $k^{-7}\chi(B)$, for some integer $k\in[y^{-1/2},x^{-2}]$.
		\end{itemize}
	\end{lemma}
	\begin{proof}
		Assume that the second outcome fails.
		Let $q:=\ceil{y^{-1/2}}$; then $q\le 2y^{-1/2}\le\frac12y^{-1}$ since $y\le 2^{-8}$.
		By \cref{lem:avgp5}, either:
		\begin{itemize}
			\item there are $P\subset A$ and $Q\subset B$ with $\chi(P)\ge\chi(A)-qr\ge\chi(A)-2y^{-1/2}r$ and $\chi(Q)\ge\frac12y^2\chi(B)$, such that $P$ is $(y,\chi)$-dense to $Q$; or
			
			\item $G[B]$ contains a pure blockade of length $q$ and mass at least $\frac12y^2\chi(B)\ge q^{-5}\chi(B)$.
		\end{itemize}
		
		Since the second outcome of the lemma fails, the first bullet holds.
		Now, by \cref{lem:1dense1} applied to $(P,Q)$, either:
		\begin{itemize}
			\item there are $X\subset P$ and $Y\subset Q$ with $\chi(X)\ge\chi(P)-2x^{-1}\ge\chi(A)-2y^{-1/2}r-2x^{-1}$ and $\chi(Y)\ge \frac12\chi(Q)\ge\frac14y^2\chi(B)$, such that $X$ is $(x,\chi)$-dense to $Y$; or
			
			\item $G[Q]$ contains a pure blockade of length $k$ and mass at least $k^{-2}\chi(Q)$, for some integer $k\in[y^{-1/2},x^{-2}]$.
		\end{itemize}
		
		Since $k^{-2}\chi(Q)\ge\frac12k^{-2}y^2\chi(B)\ge k^{-7}\chi(B)$ and the second outcome of the lemma fails, the second bullet fails; and so the first bullet holds.
		Hence the first outcome of the lemma holds.
		This proves \cref{lem:1dense2}.
	\end{proof}
	
	In what follows, for $\eps>0$ and a graph $G$ with disjoint $X,Y\subset V(G)$, we say that the ordered pair $(X,Y)$ is {\em $(\eps,\chi)$-dense} in $G$ if $X$ is $(\eps,\chi)$-dense to $Y$ in $G$.
	We next choose a suitable value of $r$ (in terms of $x,y$) in \cref{lem:1dense2} to prepare for our $\chi$-density increment step:
	
	\begin{lemma}
		\label{lem:1incre1}
		Let $x,y\in(0,2^{-8}]$ with $x\le y$, and let $G$ be a $(y,\chi)$-dense $P_5$-free graph with $\chi(G)\ge x^{-2}$.
		If $G$ is not $(y^2,\chi)$-dense, then it contains either:
		\begin{itemize}
			\item an $(x,\chi)$-dense pair $(X,Y)$ with $\chi(X)\ge (1-3y^{1/2})\chi(G)$ and $\chi(Y)\ge\frac14y^4\chi(G)$; or
			
			\item a pure blockade of length $k$ and mass at least $k^{-11}\chi(G)$, for some integer $k\in[y^{-1/2},x^{-2}]$.
		\end{itemize}
	\end{lemma}
	\begin{proof}
		Assume that the second outcome fails.
		Since $G$ is not $(y^2,\chi)$-dense, it has a vertex $v$ with $\chi(G\setminus N_G[v])\ge y^2\chi(G)\ge1=\chi(\{v\})$.
		Let $r:=y\cdot\chi(G)$, $A:=N_G(v)$, and $B:=V(G)\setminus N_G[v]$; then $\chi(B)=\chi(G\setminus A)$.
		Thus, since $G$ is $(y,\chi)$-dense, $y\le 2^{-8}$, and $\chi(G)\ge x^{-2}\ge 4x^{-1}$, we have
		\begin{align*}
			\chi(A)\ge\chi(G)-\chi(B)
			&\ge \chi(G)-y\cdot\chi(G)\\
			&=(1-y-2y^{1/2})\chi(G)+2y^{-1/2}r
			\ge \chi(G)/2+2y^{-1/2}r
			\ge 2y^{-1/2}r+2x^{-1}.
		\end{align*}
		
		Thus, \cref{lem:1dense2} implies that either:
		\begin{itemize}
			\item there are $X\subset A$ and $Y\subset B$ with $\chi(X)\ge\chi(A)-2y^{-1/2}r-2x^{-1}$ and $\chi(Y)\ge\frac14y^2\chi(B)$, such that $X$ is $(x,\chi)$-dense to $Y$; or
			
			\item $G[B]$ contains a pure blockade of length $k$ and mass at least $k^{-7}\chi(B)$, for some integer $k\in[y^{-1/2},x^{-2}]$.
		\end{itemize}
		
		Since $k^{-7}\chi(B)\ge k^{-7}y^2\chi(G)\ge k^{-11}\chi(G)$ and the second outcome of the lemma fails, the second bullet fails; and so the first bullet holds.
		Then the first outcome of the lemma holds because (note that $y\le 2^{-8}$ and $\chi(G)
		\ge x^{-2}\ge 8y^{-1/2}x^{-1}$)
		$$\chi(X)\ge\chi(A)-2y^{-1/2}r-2x^{-1}\ge (1-y-2y^{1/2})\chi(G)-2x^{-1}\ge (1-3y^{1/2})\chi(G)$$ and $\chi(Y)\ge \frac14y^2\chi(B)\ge \frac14y^4\chi(G)$.
		This proves \cref{lem:1incre1}.
	\end{proof}
	
	We next iterate \cref{lem:1incre1} to turn its first outcome into an $(x,\chi)$-dense blockade that is polynomially balanced, as follows.
	
	\begin{lemma}
		\label{lem:1incre2}
		Let $x,y\in(0,2^{-12}]$ with $x\le y$, and let $G$ be a $(y,\chi)$-dense $P_5$-free graph with $\chi(G)\ge x^{-3}$.
		Then $G$ contains either:
		\begin{itemize}
			\item a $(y^{3/2},\chi)$-dense induced subgraph $F$ with $\chi(F)\ge\frac14\chi(G)$;
			
			\item an $(x,\chi)$-dense blockade of length at least $y^{-1/4}$ and mass at least $y^4\chi(G)$; or
			
			\item a pure blockade of length $k$ and mass at least $k^{-12}\chi(G)$, for some integer $k\in[y^{-1/4},x^{-2}]$.
		\end{itemize}
	\end{lemma}
	\begin{proof}
		Assume that the first two outcomes do not hold.
		Let $\ell\ge0$ be maximal such that $G$ contains an $(x,\chi)$-dense blockade $(A_0,A_1,\ldots,A_\ell)$ with
		\[\chi(A_\ell)\ge (1-6y^{1/2})^\ell\chi(G),\quad\text{and}\quad
		\chi(A_{i-1})\ge y^4\chi(G)\quad\text{for all $i\in[\ell]$}.
		\]
		Since the second outcome fails, we have $\ell<y^{-1/4}$.
		Because $y\le 2^{-12}$, it follows that
		\[\chi(A_\ell)\ge(1-6y^{1/2})^\ell\chi(G)\ge 4^{-6\ell y^{1/2}}\chi(G)\ge 4^{-6y^{1/4}}\chi(G)\ge \chi(G)/4\ge x^{-3}/4\ge x^{-2}\]
		and so $G[A_\ell]$ is $(4y,\chi)$-dense.
		Since the first outcome of the lemma fails, $G[A_\ell]$ is not $(y^{3/2},\chi)$-dense, and so not $(16y^2,\chi)$-dense since $16y^2\le y^{3/2}$.
		Thus, by \cref{lem:1incre1} applied to $G[A_\ell]$ with $4y\le 2^{-8}$ in place of $y$, this graph contains either:
		\begin{itemize}
			\item an $(x,\chi)$-dense pair $(X,Y)$ with
			$\chi(X)\ge (1-6y^{1/2})\chi(A_\ell)\ge(1-6y^{1/2})^{\ell+1}\chi(G)$ and 
			$\chi(Y)\ge 2^6y^4\chi(A_\ell)\ge y^4\chi(G)$; or
			
			\item a pure blockade of length $k$ and mass at least $k^{-11}\chi(A_\ell)$, for some integer $k\in[\frac12y^{-1/2},x^{-2}]$.
		\end{itemize}
		
		The first bullet cannot hold by the maximality of $\ell$.
		Thus the second bullet holds; and so the third outcome of the lemma holds since $k\ge\frac12y^{-1/2}\ge y^{-1/4}\ge4$ and $k^{-11}\chi(A_\ell)\ge\frac14k^{-11}\chi(G)\ge k^{-12}\chi(G)$.
		This proves \cref{lem:1incre2}.
	\end{proof}
	
	The following simple lemma allows us to turn $\chi$-dense graphs into $\chi$-dense blockades.
	\begin{lemma}
		\label{lem:partition}
		For every integer $\ell\ge1$ and every graph $G$ with $\chi(G)\ge2\ell$, there is a partition $(A_1,\ldots,A_\ell)$ of $V(G)$ such that $\chi(A_i)\ge \chi(G)/(2\ell)$ for all $i\in[\ell]$.
	\end{lemma}
	\begin{proof}
		Let $k\in\{0,1,\ldots,\ell-1\}$ be maximal such that there are disjoint $A_1,\ldots,A_k\subset V(G)$ with $\chi(G)/\ell-1\le \chi(A_i)\le\chi(G)/\ell$ for all $i\in[k]$ (these clearly exist for $k=0$). Let $A_\ell:=V(G)\setminus(A_1\cup\cdots\cup A_k)$;
		then $\chi(A_\ell)\ge \chi(G)-k\cdot\chi(G)/\ell\ge\chi(G)/\ell$.
		Let $B\subset A_\ell$ be minimal with $\chi(B)\ge\chi(G)/\ell$, and let $v\in B$;
		then $\chi(G)/\ell-1\le\chi(B\setminus\{v\})\le\chi(G)/\ell$.
		Hence, if $k<\ell-1$ then making $A_{k+1}:=B\setminus\{v\}$ would violate the maximality of $k$.
		Thus $k=\ell-1$; and so $(A_1,\ldots,A_\ell)$ satisfies the lemma since $\chi(G)/\ell-1\ge\chi(G)/(2\ell)$.
		This proves \cref{lem:partition}.
	\end{proof}
	Provided \cref{lem:1incre2,lem:partition}, we are ready to perform $\chi$-density increment to prove \cref{lem:incre1}, which we restate here for the reader's convenience.
	
	\begin{lemma}
		\label{lem:1incre}
		There exists $a\ge2$ such that for every $x\in(0,2^{-a}]$ and for every $P_5$-free graph $G$ with $\chi(G)\ge x^{-a}$, there exist an integer $k\in[2,x^{-1}]$ and a pure or $(x,\chi)$-dense blockade in $G$ of length $k$ and mass at least $k^{-a}\chi(G)$.
	\end{lemma}
	\begin{proof}
		We claim that $a:=200$ suffices.
		To this end, let $\eps:=2^{-32}$, and let $\delta:=2^{-7}\eps^2=2^{-71}$; then $\eps,\delta$ satisfy \cref{lem:chirdl}.
		Hence, $G$ contains either:
		\begin{itemize}
			\item a pure pair $(A,B)$ with $\chi(A),\chi(B)\ge\delta\cdot\chi(G)\ge 2^{-a}\chi(G)$; or
			
			\item an $(\eps,\chi)$-dense induced subgraph with chromatic number at least $\delta\cdot\chi(G)\ge\eps^3\chi(G)$.
		\end{itemize}
		
		If the first bullet holds then the lemma holds with $k=2$; note that if $(A,B)$ is complete then it is clearly $(x,\chi)$-dense.
		Thus, we may assume that the second bullet holds.
		Hence there exists $y\in[x^{3},\eps]$ minimal of the form $y=\eps^{(3/2)^s}$ for some integer $s\ge0$, such that $G$ has a $(y,\chi)$-dense induced subgraph $F$ with $\chi(F)\ge y^3\chi(G)$. We run chromatic density increment as follows.
		\begin{claim}
			\label{claim:1incre}
			We may assume that $y\le x^{2}$.
		\end{claim}
		\begin{subproof}
			Suppose that $y\ge x^{2}$; then the choice of $a$ yields
			$$\chi(F)\ge y^3\chi(G)\ge y^3x^{-a}\ge x^{6-a}\ge x^{-6}.$$
			Hence, \cref{lem:1incre2} with $x^2$ in place of $x$ implies that $F$ contains either:
			\begin{itemize}
				\item a $(y^{3/2},\chi)$-dense induced subgraph $J$ with $\chi(J)\ge\frac14\chi(F)$;
				
				\item an $(x,\chi)$-dense blockade of length $\ceil{y^{-1/4}}$ and mass at least $y^4\chi(F)$; or
				
				\item a pure blockade of length $q$ and mass at least $q^{-12}\chi(F)$, for some integer $q\in[y^{-1/4},x^{-4}]$.
			\end{itemize}
			
			If the first bullet holds then $\chi(J)\ge\frac14\chi(F)\ge\frac14y^3\chi(G)\ge y^{9/2}\chi(G)$, which violates the minimality of $y$ since $x^3\le y^{3/2}<y$.
			
			If the second bullet holds then the lemma holds with $k=\ceil{y^{-1/4}}\ge2$, because $k\le y^{-1/4}+1\le x^{-1/2}+1\le x^{-1}$ and $y^4\chi(F)\ge y^7\chi(G)\ge k^{-28}\chi(G)\ge k^{-a}\chi(G)$.
			
			If the third bullet holds then the lemma holds with $k=\floor{q^{1/4}}\le x^{-1}$. Indeed, since $k\ge \floor{y^{-1/16}} \ge\floor{\eps^{-1/16}}\ge2$, we have $k\ge q^{1/4}-1\ge q^{1/4}/2\ge q^{1/8}\ge2$, which implies
			\begin{align*}
				q^{-12}\chi(F)\ge q^{-12}y^3\chi(G)\ge q^{-24}\chi(G)\ge k^{-192}\chi(G)\ge k^{-a}\chi(G).
			\end{align*}
			This proves \cref{claim:1incre}.
		\end{subproof}
		
		Now, let $\ell:=\ceil{x^{-1/2}}\le \frac12x^{-1}$.
		The choice of $a$ yields
		$$\chi(F)\ge y^3\chi(G)\ge y^3x^{-a}\ge x^{9-a}\ge 4x^{-1/2}\ge 2\ell$$
		and so \cref{lem:partition} gives a partition $(A_1,\ldots,A_\ell)$ of $V(F)$ such that
		$\chi(A_i)\ge\chi(F)/(2\ell)$ for all $i\in[\ell]$. 
		Since $F$ is $(y,\chi)$-dense and $(2\ell)^{-1}>x$, \cref{claim:1incre} implies that, 
		for all $1\le i<j\le\ell$ and all $v\in A_j$, we have
		$\chi(A_i\setminus N_G(v))\le y\cdot \chi(F)\le x^2\chi(F)<x\cdot\chi(A_i)$.
		Hence, the lemma holds because 
		$$\chi(A_i)\ge \chi(F)/(2\ell)\ge x\cdot\chi(F)\ge x\cdot y^3\chi(G)\ge x^{10}\chi(G)\ge \ell^{-20}\chi(G)\ge \ell^{-a}\chi(G)$$
		for all $i\in[\ell]$. This proves \cref{lem:1incre}. 
	\end{proof}
	
	\section{Anticomplete or $\chi$-dense blockades via \erh{}}
	\label{sec:conv}
	This section provides a proof of \cref{lem:waymid}. The following general theorem allows us to convert the blockades given by \cref{lem:incre1} into more coherent structures; its proof is essentially the proof of~\cite[Theorem 6.1]{density7} but for chromatic number.
	\begin{theorem}
		\label{thm:conv}
		Let $\eps\in(0,\frac12]$ and $a\ge1$, and let $G$ be a graph with $\chi(G)\ge\eps^{-3a^2}$. Let $x:=\eps^{3a}$. Assume that for every induced subgraph $F$ of $G$ with $\chi(F)\ge \eps^a\chi(G)$, there exist an integer $k\in[2,x^{-1}]$ and a pure or $(x,\chi)$-dense blockade in $F$ of length $k$ and mass at least $k^{-a}\chi(F)$.
		Then $G$ contains a blockade $(B_1,\ldots,B_\ell)$ of length at least $\eps^{-1}$ and mass at least $x^{2a}\chi(G)$, such that for all $i,j\in[\ell]$ with $i<j$, $B_j$ is either anticomplete or $(\eps^a,\chi)$-dense to $B_i$.
	\end{theorem}
	\begin{proof}
		Let $J$ be a graph; in what follows we assume $V(J)=\{1,\ldots,\abs{V(J)}\}$.
		For each $j\in V(J)$ let $A_j$ be a nonempty subset of $V(G)$, such that for all distinct $i,j\in V(J)$, $A_i$ and $A_j$ are disjoint, and they are anticomplete in $G$ whenever $i,j$ are nonadjacent in $J$. We call $\mathcal{L}=(J,(A_j:j\in V(J)))$
		a {\em layout}. A pair $\{u,v\}$ of distinct vertices of $G$ is {\em undecided} for the layout $\mac L$
		if there exists $j\in V(J)$ with $u,v\in A_j$; and {\em decided} otherwise.
		A decided pair $\{u,v\}$ is {\em wrong} for $\mac L$ if there are distinct
		$i,j\in V(J)$ such that $u\in A_i$, $v\in A_j$, and $u,v$ are nonadjacent in $G$ while $i,j$ are adjacent in $J$.
		
		Choose a layout $\mathcal{L}=(J,(A_j:j\in V(J)))$ satisfying the following:
		\begin{itemize}
			\item $\chi(A_j)\ge \eps^{2a}\chi(G)$ for each $j\in V(J)$;
			\item $\sum_{j\in V(J)}\chi(A_j)^{1/a}\ge \chi(G)^{1/a}$;
			\item for all $i,j\in V(J)$ adjacent in $J$ with $i<j$ and for all $v\in A_j$, the chromatic number of the set of vertices $u\in A_i$ for which $\{u,v\}$ is a wrong pair is less than $x\cdot \chi(G)$; and
			\item subject to these three conditions, $|J|$ is maximum.
		\end{itemize}
		(This is possible since we may take $V(J)=\{1\}$ and $A_1=V(G)$ to satisfy the first three conditions.)
		
		\begin{claim}
			\label{claim:short}
			We may assume that $|J|\le \eps^{-1}$.
		\end{claim}
		
		\begin{subproof}
			Assume that $|J|\ge \eps^{-1}$. 
			For every $i,j\in V(J)$ adjacent in $J$ with $i<j$ and for all $v\in A_j$,
			the chromatic number of the set of vertices $u\in A_i$ for which $\{u,v\}$ is a wrong pair is less than $x\cdot\chi(G)\le x\cdot \eps^{-2a}\chi(A_i)= \eps^a\cdot\chi(A_i)$.
			Hence $(A_j:j\in V(J))$ is a blockade satisfying the theorem.
			This proves \cref{claim:short}. 
		\end{subproof}
		
		Now, let $A\in\{A_j:j\in V(J)\}$ satisfy $\chi(A)=\max(\chi(A_j):j\in V(J))$. Since $\sum_{j\in V(J)}\chi(A_j)^{1/a}\ge \chi(G)^{1/a}$,
		and $|J|\le \eps^{-1}$ by \cref{claim:short}, it follows that
		$\chi(A)^{1/a}\ge \eps\cdot\chi(G)^{1/a}$, that is, 
		$\chi(A)\ge \eps^a\chi(G)$.
		By applying the hypothesis to $G[A]$,
		we obtain an integer $k\in[2,x^{-1}]$ and a pure or $(x,\chi)$-dense blockade $(B_1,\ldots, B_{k})$ in $G[A]$ of mass at least $ k^{-a}\chi(A)$.
		Let $K$ be the graph with vertex set $[k]$, such that for all distinct $p,q\in[k]$, $p$ is nonadjacent to $q$ in $K$ if and only if $B_p$ is anticomplete to $B_q$ in $G[A]$; in particular $K$ is complete if $(B_1,\ldots,B_{k})$ is $(x,\chi)$-dense in $G[A]$.
		\begin{claim}
			\label{claim:long}
			$k\ge \eps^{-1}$.
		\end{claim}
		\begin{subproof}
			Suppose that $k\le \eps^{-1}$. Then $\chi(B_p)\ge k^{-a}\chi(A)\ge \eps^{a}\chi(A)$ for all $p\in[k]$.
			By substituting $K$ for
			the vertex of $J$ corresponding to $A$, and replacing $A$ by $B_1,\ldots, B_k$ in the natural order, we obtain a new layout $\mathcal{L}'=(J', (A_j':j\in V(J')))$ say,
			where $|J'|>|J|$.
			We claim that this violates the choice of $\mathcal{L}$; and so we must verify that $\mathcal{L}'$ satisfies the
			first three bullets in the definition of $\mathcal{L}$. 
			To see this, observe that each $B_p$ satisfies
			$\chi(B_p)\ge \eps^{a}\chi(A)\ge \eps^{2a} \chi(G)$,
			and so the first bullet is satisfied.
			For the second bullet, since $\chi(B_p)\ge k^{-a}\chi(A)$ for all $p\in[k]$, it follows that
			\[\chi(B_1)^{1/a}+\cdots+\chi(B_k)^{1/a}\ge\chi(A)^{1/a}\]
			and so $\sum_{j\in V(J')}\chi(A_j')^{1/a}\ge \chi(G)^{1/a}$.
			
			It remains to verify the third bullet for $\mac L'$. Let $i\in V(J)$ be such that $A=A_i$.
			For each neighbour $j$ of $i$ in $J$ with $i<j$, for each $p\in[k]$, and for each $v\in A_j$, the set of vertices $u\in B_p$ for which $\{u,v\}$ is a wrong pair is a subset of the set of vertices $u'\in A_i$ for which $\{u',v\}$ is a wrong pair, and hence has chromatic number less than $x\cdot \chi(G)$.
			For each neighbour $j$ of $i$ in $J$ with $j<i$ and for each $v\in B_1\cup\cdots\cup B_k$, we have $v\in A_i$; and so the set of vertices $u\in A_j$ for which $\{u,v\}$ is a wrong pair has chromatic number less than $x\cdot\chi(G)$ since $\mac L$ is a layout. 
			Next, if $(B_1,\ldots,B_k)$ is $(x,\chi)$-dense, then for each $p,p'\in[k]$ adjacent in $K$ with $p<p'$ and for each $v\in B_{p'}$, the chromatic number of the set of vertices $u\in B_p$ for which $\{u,v\}$ is a wrong pair is less than $x\cdot \chi(B_p)\le x\cdot \chi(G)$; and if $(B_1,\ldots,B_k)$ is pure then it contains no wrong pair.
			Hence $\mac L'$ satisfies the third bullet. 
			This contradicts the choice of $\mathcal{L}$, and so proves \cref{claim:long}.
		\end{subproof}
		
		Now, since $k\le x^{-1}$ and $\chi(A)\ge\eps^a\chi(G)\ge x^a\chi(G)$, we have $\chi(B_p)\ge k^{-a}\chi(A)\ge x^{a}\chi(A)\ge x^{2a}\chi(G)$ for all $p\in[k]$;
		and for all distinct $p,p'\in[k]$ with $p<p'$,
		$B_{p'}$ is either anticomplete or $(\eps^a,\chi)$-dense to $B_p$ since $x=\eps^{3a}\le \eps^a$.
		Hence $(B_1,\ldots,B_k)$ satisfies the theorem by \cref{claim:long}.
		This proves \cref{thm:conv}.
	\end{proof}
	We are now ready to prove \cref{lem:waymid} by using \cref{lem:incre1}, \cref{thm:conv}, and the \erh{} property of $P_5$~\cite{density7}.
	We restate \cref{lem:waymid} here for the reader's convenience:
	
	\begin{lemma}
		\label{lem:midway}
		There exists $b\ge2$ such that for every $\eps\in(0,\frac12]$, every $P_5$-free graph $G$ with $\chi(G)\ge\eps^{-b}$ contains an anticomplete or $(\eps,\chi)$-dense blockade of length at least $\eps^{-1}$ and mass at least $\eps^b\chi(G)$.
	\end{lemma}
	\begin{proof}
		Let $a\ge2$ satisfy \cref{lem:incre1}; by increasing $a$ if necessary we may assume that every $k$-vertex $P_5$-free graph has a clique or stable set with at least $k^{1/a}$ vertices, by the \erh{} property of $P_5$.
		We claim that $b:=6a^3$ suffices.
		To this end, let $x:=\eps^{3a^2}$; we first prove that:
		\begin{claim}
			\label{claim:midway}
			For every induced subgraph $F$ of $G$ with $\chi(F)\ge \eps^{a^2}\chi(G)$, there exist an integer $k\in[2,x^{-1}]$ and a pure or $(x,\chi)$-dense blockade in $F$ of length $k$ and mass at least $k^{-a}\chi(F)$.
		\end{claim}
		\begin{subproof}
			Since $\chi(F)\ge \eps^{a^2}\chi(G)\ge \eps^{a^2-b}\ge \eps^{-3a^3}= x^{-a}$,
			the choice of $a$ and \cref{lem:1incre} together give an integer $k\in[2,x^{-1}]$ and a pure or $(x,\chi)$-dense blockade in $F$ of length $k$ and mass at least $k^{-a}\chi(F)$.
			This proves \cref{claim:midway}.
		\end{subproof}
		Now, \cref{claim:midway} and \cref{thm:conv} (with $\eps^a$ in place of $\eps$, note that $\chi(G)\ge\eps^{-b}\ge \eps^{-3a^3}$) together give a blockade $(B_1,\ldots,B_\ell)$ in $G$ of length at least $\eps^{-a}$ and mass at least $x^{2a}\chi(G)=\eps^{6a^3}\chi(G)=\eps^b\chi(G)$, such that for all $i,j\in[\ell]$ with $i<j$, $B_j$ is either anticomplete or $(\eps^a,\chi)$-dense to $B_i$.
		Let $J$ be the graph on vertex set $[\ell]$ such that for every $i,j\in[\ell]$ with $i<j$, $i$ is adjacent to $j$ in $J$ if and only if $B_j$ is $(\eps^a,\chi)$-dense to $B_i$.
		For completeness we give a proof that $J$ is $P_5$-free, as follows.
		\begin{claim}
			\label{claim:pattern}
			$J$ is $P_5$-free.
		\end{claim}
		\begin{subproof}
			Suppose not; then there are $j_1,\ldots,j_5\in J$ with $j_1>\ldots>j_5$ such that $J[\{j_1,\ldots,j_5\}]$ is isomorphic to $P_5$.
			Let $p\in\{1,\ldots,5\}$ be maximal such that there are $v_1\in B_{j_1},\ldots,v_p\in B_{j_p}$ for which $G[\{v_1,\ldots,v_p\}]$ is isomorphic to $J[\{j_1,\ldots,j_p\}]$; for $p=1$ one can choose an arbitrary $v_1\in B_{j_1}$.
			Since $G$ is $P_5$-free, $p\le 4$.
			For each $i\in[p]$, let $E_i:=\emptyset$ if $j_i$ is nonadjacent to $j_{p+1}$ in $J$ (so $B_{j_i}$ is anticomplete to $B_{j_{p+1}}$), and let $E_i$ be $B_{j_{p+1}}\setminus N_G(v_i)$ if $j_i$ is adjacent to $j_{p+1}$ in $J$ (so $B_{j_i}$ is $(\eps^a,\chi)$-dense to $B_{j_{p+1}}$).
			Then $\chi(E_i)<\eps^a\chi(B_{j_{p+1}})\le\frac14\chi(B_{j_{p+1}})$ for all $i\in[p]$;
			and since $p\le 4$, there exists $v_{p+1}\in B_{j_{p+1}}\setminus(E_1\cup\cdots\cup E_p)$.
			But then $v_1,\ldots,v_p,v_{p+1}$ violate the maximality of $p$.
			This proves \cref{claim:pattern}.
		\end{subproof}
		
		\cref{claim:pattern} and the choice of $a$ together give a clique or stable set $I$ of $J$ with $\abs I\ge \abs J^{1/a}=\ell^{1/a}\ge\eps^{-1}$.
		We may assume $I=[q]$ for some $q\ge\eps^{-1}$; then $(B_1,\ldots,B_q)$ is an anticomplete blockade in $G$ if $I$ is a stable set in $J$, and is an $(\eps,\chi)$-dense blockade in $G$ if $I$ is a clique in $J$.
		This proves \cref{lem:midway}.
	\end{proof}
	
	\section{Chromatic density increment, second round}
	\label{sec:2incre}
	In this section we complete the proof of \cref{lem:incre}. As sketched in \cref{subsec:incre}, our plan here is to use \cref{lem:waymid} to implement the second round of chromatic density increment.
	The main issue now is that it is not clear to us how to convert the $(\eps,\chi)$-dense blockades given by \cref{lem:waymid} into $(\operatorname{poly}(\eps),\chi)$-dense induced subgraphs, even in $P_5$-free graphs.
	But all is not lost; we will run chromatic density increment on the blockades themselves to obtain polynomially balanced complete blockades in several fashions.
	We aim to show that for some suitable $b>1$, given a $(y,\chi)$-dense blockade $(A_1,\ldots,A_\ell)$ in a graph $G$ for some chromatic density parameter $y$ and for some integer $\ell\ge y^{-1}$, either:
	\begin{itemize}
		\item some $G[A_j]$ contains a $(y^b,\chi)$-dense blockade of length at least $y^{-b}$ and mass at least $y^{b^2}\chi(A_j)$ (i.e. $\chi$-density increment);
		
		\item there exist $I\subset[\ell]$ with $\abs I\ge y^{-1/2}$ and $B_i\subset A_i$ with $\chi(B_i)\ge y^{b^2}\chi(A_i)$ for all $i\in I$ so that $(B_i:i\in I)$ is a complete blockade (i.e. a `partial transversal' complete blockade); or
		
		\item some $G[A_j]$ contains a complete blockade of length at least $y^{-1/2}$ and mass at least $y^{b^2}\chi(A_j)$ (i.e. a polynomially balanced complete blockade inside some block).
	\end{itemize}
	
	(The numerical constants in the actual proof will be slightly different.)
	Our $\chi$-density increment terminates when the chromatic density parameter drops below some negative power of $\chi(G)$, which would immediately give a clique in $G$ of size at least some positive power of $\chi(G)$.
	
	To run chromatic density increment on $\chi$-dense blockades, we will chase induced five-vertex paths via anticonnected induced subgraphs; here, a graph is {\em anticonnected} if its complement is connected.
	The following lemma shows how to extract induced $P_5$'s from pairwise anticomplete and anticonnected vertex subsets.
	It also underlines the importance of the \erh{} property of $P_5$ in producing a full anticomplete blockade as one outcome in \cref{lem:waymid}.
	
	\begin{lemma}
		\label{lem:2dense1}
		Let $k\ge1$ be an integer, and let $r>0$. Let $G$ be a $P_5$-free graph with nonempty disjoint $A,B_1,\ldots,B_k\subset V(G)$, such that:
		\begin{itemize}
			\item $B_1,\ldots,B_k$ are pairwise anticomplete in $G$ and $G[B_i]$ is anticonnected for each $i\in[k]$;
			
			\item no vertex in $A$ is mixed on two of $B_1,\ldots,B_k$; and
			
			\item $\chi(A\setminus N_G(v))\le r$ for all $v\in B_1\cup\cdots\cup B_k$.
		\end{itemize}
		
		Then either:
		\begin{itemize}
			\item there exists $I\subset[k]$ with $\abs I\ge k-k^{1/2}$ such that for each $i\in I$, there exists $A_i'\subset A$ complete to $B_i$ with $\chi(A_i')\ge (1-k^{-1})\chi(A)-kr$; or
			
			\item $G[A]$ contains a complete blockade of length at least $k^{1/2}$ and mass at least $k^{-1}\chi(A)$.
		\end{itemize}
	\end{lemma}
	\begin{proof}
		
		\begin{figure}[ht]
			\centering
			
			\begin{tikzpicture}[scale=0.7,auto=left]

				\draw [rounded corners] (-5,-2) rectangle (5,2);
				\node [] at (-5.75,0) {$A$};
				\draw [-] (-4,-2) -- (-4,2);
				\node [] at (-4.5,0) {$D$};
				\draw [-] (-4,-0.5) -- (5,-0.5);
				\draw [-] (-2,-0.5) -- (-2,-2);
				\draw [-] (0,-0.5) -- (0,-2);
				\draw [-] (3,-0.5) -- (3,-2);
				\node [] at (-3,-1.25) {$D_1$};
				\node [] at (-1,-1.25) {$D_2$};
				\node [] at (1.5,-1.25) {$\cdots$};
				\node [] at (4,-1.25) {$D_k$};
				
				\draw [] (-3,-3.5) circle (0.8cm);
				\draw [] (-1,-3.5) circle (0.8cm);
				\draw [] (4,-3.5) circle (0.8cm);
				\node [] at (-3,-3.5) {$B_1$};
				\node [] at (-1,-3.5) {$B_2$};
				\node [] at (1.5,-3.5) {$\cdots$};
				\node [] at (4,-3.5) {$B_k$};

				\draw [rounded corners] (12,0) rectangle (14,2);
				\draw [rounded corners] (8.5,0) rectangle (10.5,2);
				
				\draw [] (9.5,-3) circle (1.25cm);
				\draw [] (13,-3) circle (1.25cm);
				
				\node[] at (7.75,-3) {$B_i$};
				\node[] at (8,0.95) {$D_i$};
				\node[] at (14.75,-3.05) {$B_j$};
				\node[] at (14.5,0.95) {$D_j$};
				
				\node[inner sep=1.5pt, fill=black,circle,draw] (v) at ({13},{1}) {};
				\node[xshift=0.3cm] at (v) {$v$};
				\node[inner sep=1.5pt, fill=black,circle,draw] (z) at ({9.5},{-3}) {};
				\node[xshift=-0.3cm] at (z) {$z$};
				\node[inner sep=1.5pt, fill=black,circle,draw] (u) at ({9.5},{1}) {};
				\node[xshift=-0.3cm] at (u) {$u$};
				\node[inner sep=1.5pt, fill=black,circle,draw] (v') at ({12.5},{-3.5}) {};
				\node[xshift=0.2cm,yshift=-0.25cm] at (v') {$v'$};
				\node[inner sep=1.5pt, fill=black,circle,draw] (u') at ({13.5},{-2.5}) {};
				\node[xshift=0.25cm,yshift=-0.15cm] at (u') {$u'$};
				\draw[-] (z) -- (v) -- (u') -- (u) -- (v');
				\draw[dashed] (z) -- (u') -- (v') -- (v) -- (u) -- (z) -- (v');
				
			\end{tikzpicture}
			
			\caption{Proof of \cref{lem:2dense1}.}
			\label{fig:2dense1}
		\end{figure}

		Let $D$ be the set of vertices in $A$ with no neighbour in $B_i$ for some $i\in[k]$; then $\chi(D)\le kr$ by the hypothesis.
		For each $i\in[k]$, let $D_i$ be the set of vertices in $A\setminus D$ mixed on $B_i$;
		then the hypothesis implies that $D_i$ is complete to $B_j$ for all distinct $i,j\in[k]$.
		Hence $D_1,\ldots,D_k$ are pairwise disjoint; and also $A\setminus(D\cup D_i)$ is complete to $B_i$ for all $i\in[k]$.
		
		\begin{claim}
			\label{claim:complete}
			$(D_1,\ldots,D_k)$ is a complete blockade in $G[A]$.
		\end{claim}
		\begin{subproof}
			Suppose not; then there are distinct $i,j\in[k]$ with nonadjacent $u\in D_i$ and $v\in D_j$.
			Since $v$ is mixed on $B_j$ and since $G[B_j]$ is anticonnected by the hypothesis, there are $u',v'\in B_j$ nonadjacent in $G[B_j]$ such that $v$ is adjacent to $u'$ and nonadjacent to $v'$.
			Since $u$ is mixed on $B_i$, it has a nonneighbour $z\in B_i$;
			but then $z\text-v\text-u'\text-u\text-v'$ would be an induced $P_5$ in $G$, a contradiction. (See the right-hand side of \cref{fig:2dense1}.)
			This proves \cref{claim:complete}.
		\end{subproof}
		
		Now, let $I:=\{i\in[k]:\chi(D_i)\le k^{-1}\chi(A)\}$;
		then $\chi(A\setminus(D\cup D_i))\ge (1-k^{-1})\chi(A)-kr$ for all $i\in I$.
		Thus, if $\abs I\ge k-k^{1/2}$ then the first outcome of the lemma holds; and if $\abs I\le k-k^{1/2}$ then the second outcome of the lemma holds by \cref{claim:complete}.
		This proves \cref{lem:2dense1}.
	\end{proof}
	
	A simple averaging argument leverages the setting of one vertex subset $A$ as in \cref{lem:2dense1} to that of many disjoint vertex subsets $A_1,\ldots,A_\ell$, as follows.
	\begin{lemma}
		\label{lem:2dense2}
		Let $k,\ell\ge1$ be integers and let $r_1,\ldots,r_\ell>0$.
		Let $G$ be a $P_5$-free graph with nonempty disjoint $A_1,\ldots,A_\ell,B_1,\ldots,B_k\subset V(G)$, such that:
		\begin{itemize}
			\item $B_1,\ldots,B_k$ are pairwise anticomplete in $G$ and $G[B_i]$ is anticonnected for all $i\in[k]$;
			
			\item no vertex in $A_1\cup\cdots\cup A_\ell$ is mixed on two of $B_1,\ldots,B_k$; and
			
			\item $\chi(A_j\setminus N_G(v))\le r_j$ for all $j\in[\ell]$ and all $v\in B_1\cup\cdots\cup B_k$.
		\end{itemize}
		Then either:
		\begin{itemize}
			\item there exist $i\in[k]$ and $I\subset[\ell]$ with $\abs I\ge(1-k^{-1/2})\ell$ such that for all $j\in I$, there exists $A_j'\subset A_j$ complete to $B_i$ with $\chi(A_j')\ge (1-k^{-1})\chi(A_j)-kr_j$; or
			
			\item there exist $j\in[\ell]$ and a complete blockade in $G[A_j]$ of length at least $k^{1/2}$ and mass at least $k^{-1}\chi(A_j)$.
		\end{itemize}
	\end{lemma}
	\begin{proof}
		Assume that the second outcome fails.
		Then for each $j\in[\ell]$, \cref{lem:2dense1} with $A=A_j$ and $r=r_j$ yields $I_j\subset[k]$ with $\abs{I_j}\ge k-k^{1/2}=(1-k^{-1/2})k$ such that for every $i\in I_j$, there exists $A_{ij}\subset A_j$ complete to $B_i$ with $\chi(A_{ij})\ge (1-k^{-1})\chi(A_j)-kr_j$.
		By averaging, there exists $i\in[k]$ such that $i\in I_j$ for at least $(1-k^{-1/2})\ell$ indices $j\in[\ell]$.
		Thus the first outcome of the lemma holds, proving \cref{lem:2dense2}.
	\end{proof}
	
	The following corollary of \cref{lem:blocks} says that every graph contains either a high-$\chi$ anticonnected induced subgraph or a polynomially balanced complete blockade.
	
	\begin{lemma}
		\label{lem:anticonn}
		Let $y\in(0,\frac18]$, and let $G$ be a graph such that there is no $S\subset V(G)$ satisfying $\chi(S)\ge y\cdot\chi(G)$ and $G[S]$ is anticonnected. Then $G$ contains a complete blockade of length at least $y^{-1/2}$ and mass at least $y\cdot\chi(G)$.
	\end{lemma}
	\begin{proof}
		Let $(A_1,\ldots,A_k)$ be a partition of $V(G)$ such that each $A_i$ induces a connected component in the complement of $G$.
		Then $(A_1,\ldots,A_k)$ is a complete blockade in $G$ and $\chi(A_i)< y\cdot\chi(G)$ for all $i\in[k]$ by the hypothesis.
		The conclusion then follows from \cref{lem:blocks}.
	\end{proof}
	We next prepare for chromatic density increment on $\chi$-dense blockades. Combining \cref{lem:waymid,lem:mixed,lem:2dense2,lem:anticonn} gives:
	\begin{lemma}
		\label{lem:2incre1}
		Let $b\ge2$ be given by \cref{lem:waymid}.
		Let $y\in(0,2^{-8}]$, and let $G$ be a $P_5$-free graph with a $(y,\chi)$-dense blockade $(A_1,\ldots,A_\ell)$ where $\ell\ge y^{-1/4}$ and $\chi(A_\ell)\ge y^{-b^2}$. Then either:
		\begin{itemize}
			\item $G[A_\ell]$ contains a $(y^b,\chi)$-dense blockade of length at least $y^{-b}$ and mass at least $y^{b^2}\chi(A_\ell)$;
			
			\item there exist $B_\ell\subset A_\ell$ and $I\subset [\ell-1]$ with $\chi(B_\ell)\ge y^{b^2+1}\chi(A_\ell)$ and $\abs I\ge(1-2y^{1/4})\ell$, such that for all $j\in I$, there exists $A_j'\subset A_j$ complete to $B_\ell$ with $\chi(A_j')\ge (1-3y^{1/2})\chi(A_j)$; or
			
			\item there exist $j\in[\ell]$ and a complete blockade in $G[A_j]$ of length at least $y^{-1/4}$ and mass at least $y^{b^2+1}\chi(A_j)$.
		\end{itemize}
	\end{lemma}
	\begin{proof}
		Assume that the first and third outcomes do not hold.
		Then \cref{lem:waymid} with $\eps=y^b$ gives an anticomplete blockade $(D_1,\ldots,D_q)$ in $G[A_\ell]$ of length at least $y^{-b}$ and mass at least $y^{b^2}\chi(A_\ell)$;
		we may assume $G[D_i]$ is connected for all $i\in[q]$.
		Hence, by \cref{lem:mixed} and since $G$ is $P_5$-free, no vertex in $A_1\cup\cdots\cup A_{\ell-1}$ is mixed on at least two of $D_1,\ldots,D_q$.
		By \cref{lem:anticonn} and since the third outcome fails (with $j=\ell$, note that $y\cdot\chi(D_i)\ge y^{b^2+1}\chi(A_\ell)$),
		for each $i\in[q]$ there exists $B_i\subset D_i$ satisfying $G[B_i]$ is anticonnected and 
		$\chi(B_i)\ge y\cdot\chi(D_i)\ge y^{b^2+1}\chi(A_\ell)$.
		
		Let $k:=\ceil{y^{-1/2}}\le 2y^{-1/2}\le y^{-1}\le q$; then by \cref{lem:2dense2} with $\ell-1$ in place of $\ell$ and $r_j=y\cdot\chi(A_j)$ for all $j\in [\ell-1]$, either:
		\begin{itemize}
			\item there exist $i\in[k]$ and $I\subset[\ell-1]$ with $\abs I\ge(1-k^{-1/2})(\ell-1)$ such that for all $j\in I$, there exists $A_j'\subset A_j$ complete to $B_i$ with $\chi(A_j')\ge (1-k^{-1})\chi(A_j)-kr_j$; or
			
			\item there exist $j\in[\ell-1]$ and a complete blockade in $G[A_j]$ of length at least $y^{-1/4}$ and mass at least $k^{-1}\chi(A_j)\ge y\cdot\chi(A_j)$.
		\end{itemize}
		
		If the second bullet holds then the third outcome of the lemma holds, contrary to our assumption; hence the first bullet holds.
		Then the second outcome of the lemma holds because
		\begin{align*}
			(1-k^{-1/2})(\ell-1)&\ge (1-y^{1/4})\ell-1\ge(1-2y^{1/4})\ell,\quad\text{and}\\
			(1-k^{-1})\chi(A_j)-kr_j&\ge (1-y^{1/2})\chi(A_j)-2y^{1/2}\chi(A_j)=(1-3y^{1/2})\chi(A_j)\quad\text{for all $j\in I$}
		\end{align*} 
		(note that $y^{-1/2}\le k\le 2y^{-1/2}$). This proves \cref{lem:2incre1}. 
	\end{proof}
	
	Iterating \cref{lem:2incre1} gives the following rigorous version of our sketch at the start of this section.
	
	\begin{lemma}
		\label{lem:2incre2}
		Let $b\ge2$ given by \cref{lem:waymid}.
		Let $y\in(0,2^{-16}]$, and let $G$ be a $P_5$-free graph with a $(y,\chi)$-dense blockade $(A_1,\ldots,A_\ell)$ where $\ell \ge 8y^{-1/4}$ and $\chi(A_j)\ge y^{-2b^2}$ for all $j\in[\ell]$.
		Then either:
		\begin{itemize}
			\item there exists $j\in[\ell]$ such that $G[A_j]$ contains a $(y^{2b/3},\chi)$-dense blockade $(D_1,\ldots,D_k)$ with $k\ge y^{-2b/3}$ and $\chi(D_i)\ge y^{b^2}\chi(A_j)$ for all $i\in[k]$;
			
			\item there exists $J\subset[\ell]$ with $\abs J\ge y^{-1/8}$ such that for all $j\in J$, there exists $B_j\subset A_j$ satisfying $\chi(B_j)\ge y^{2b^2}\chi(A_j)$ and $(B_j:j\in J)$ is a complete blockade; or
			
			\item there exist $j\in[\ell]$ and a complete blockade in $G[A_j]$ of length at least $y^{-1/8}$ and mass at least $y^{2b^2}\chi(A_j)$.
		\end{itemize}
	\end{lemma}
	\begin{proof}
		Let $J$ be a maximal subset of $[\ell]$ satisfying:
		\begin{itemize}
			\item for each $j\in J$, there exists $B_j\subset A_j$ with $\chi(B_j)\ge y^{2b^2}\chi(A_j)$ such that $(B_j:j\in J)$ is a complete blockade; and
			
			\item there exists $I\subset[\ell]$ such that $\abs I\ge (1-4y^{1/4})^{\abs J}\ell$, $i<j$ for all $i\in I$ and $j\in J$, and for each $i\in I$ there exists $C_i\subset A_i$ complete to $\bigcup_{j\in J}B_j$ with $\chi(C_i)\ge (1-6y^{1/2})^{\abs J}\chi(A_i)$.
		\end{itemize}
		(These conditions are clearly satisfied for $J$ empty.)
		If $\abs J\ge y^{-1/8}$ then the second outcome of the lemma holds and we are done; so we may assume $\abs J<y^{-1/8}$. Then since $y\le 2^{-16}$, it follows that
		\begin{align*}
			\abs I&\ge (1-4y^{1/4})^{\abs J}\ell\ge 4^{-4y^{1/4}\abs J}\ell\ge 4^{-4y^{1/8}}\ell\ge \ell/4\ge (2y)^{-1/4},\quad\text{and}\\
			\chi(C_i)&\ge(1-6y^{1/2})^{\abs J}\chi(A_i)\\
			&\ge 4^{-6y^{1/2}\abs J}\chi(A_i)\ge 2^{-12y^{1/4}}\chi(A_i)\ge\chi(A_i)/2\ge y^{-2b^2}/2\ge y^{-b^2}
			\quad\text{for all $i\in I$.}
		\end{align*}
		Thus $(C_i:i\in I)$ is $(2y,\chi)$-dense with the order respecting the natural order on $I$.
		Now, let $s$ be the maximum element of $I$;
		then by \cref{lem:2incre1} applied to $(C_i:i\in I)$ with $2y$ replacing $y$, either:
		\begin{itemize}
			\item there exists a $((2y)^b,\chi)$-dense blockade in $G[C_s]$ of length at least $(2y)^{-b}$ and mass at least $(2y)^{b^2}\chi(C_s)\ge y^{b^2}\chi(A_s)$;
			
			\item there exist $B_s\subset C_s$ and $I'\subset I$ with
			$$\chi(B_s)\ge (2y)^{b^2+1}\chi(C_s)\ge y^{2b^2}\chi(A_s)
			\quad\text{and}\quad
			\abs {I'}\ge(1-4y^{1/4})\abs I\ge (1-4y^{1/4})^{\abs J+1}\ell,$$
			such that for all $i\in I'$, there exists $C_i'\subset C_i$ complete to $B_s$ with
			$\chi(C_i')\ge (1-6y^{1/2})\chi(C_i)\ge (1-6y^{1/2})^{\abs J+1}\chi(A_i)$; or
			
			\item there exist $j\in I$ and a complete blockade in $G[C_j]$ of length at least $(2y)^{-1/4}$ and mass at least
			$(2y)^{b^2+1}\chi(C_j)\ge y^{2b^2}\chi(A_j)$.
		\end{itemize}
		
		If the first bullet holds then the first outcome of the lemma holds since $2y\le y^{2/3}$.
		If the second bullet holds then $J\cup\{s\}$ would violate the maximality of $J$.
		If the third bullet holds then the third outcome of the lemma holds because $(2y)^{-1/4}\ge y^{-1/8}$.
		This proves \cref{lem:2incre2}.
	\end{proof}
	
	We are now in a position to run chromatic density increment and prove \cref{lem:incre}, which we restate here for the reader's convenience:
	\begin{lemma}
		\label{lem:2incre}
		There exists $a\ge2$ such that for every $\eps\in(0,2^{-32}]$ and every $(\eps,\chi)$-dense $P_5$-free graph $G$ with $\chi(G)\ge\eps^{-a}$, there exist an integer $k\ge\eps^{-1/16}$ and a complete blockade in $G$ of length $k$ and mass at least $k^{-a}\chi(G)$.
	\end{lemma}
	\begin{proof}
		Let $b\ge2$ be given by \cref{lem:waymid}; we claim that $a:=16b^2+48b$ suffices.
		To this end, let $q:=\floor{\frac12\eps^{-1/2}}\ge2$, and let $(V_1,\ldots,V_q)$ be a partition of $V(G)$ with $\chi(V_i)\ge\chi(G)/(2q)$ for all $i\in[q]$; this is possible by \cref{lem:partition} since $\chi(G)\ge\eps^{-a}\ge q^2\ge 2q$.
		Because $2q\le \eps^{-1/2}$, the blockade $(V_1,\ldots,V_q)$ is $(\eps^{1/2},\chi)$-dense with $q=\floor{\frac12\eps^{-1/2}}\ge 8\eps^{-1/8}$ (since $\eps\le 2^{-32}$) and $\chi(V_i)\ge \eps^{1/2}\chi(G)$ for all $i\in[q]$.
		Hence, since $\chi(G)\ge\eps^{-a}\ge \eps^{-b/2}$, there exists $y\in[\chi(G)^{-1/b},\eps^{1/2}]$ minimal of the form $y=\eps^{(2b/3)^s/2}$ for some integer $s\ge0$, such that $G$ contains a $(y,\chi)$-dense blockade $(A_1,\ldots,A_\ell)$ of length at least $8y^{-1/4}$ and mass at least $y^{6b}\chi(G)$.
		
		\begin{claim}
			\label{claim:2incre}
			We may assume that $y\le \chi(G)^{-1/(2b^2+6b)}$.
		\end{claim}
		
		\begin{subproof}
			Suppose that $y\ge \chi(G)^{-1/(2b^2+6b)}$. Then $\chi(A_j)\ge y^{6b}\chi(G)\ge y^{-2b^2}$ for all $j\in[\ell]$. Thus, since $y\le \eps^{1/2}\le 2^{-16}$, \cref{lem:2incre2} implies that either:
			\begin{itemize}
				\item there exists $j\in[\ell]$ such that $G[A_j]$ contains a $(y^{2b/3},\chi)$-dense blockade of length at least $y^{-2b/3}$ and mass at least $y^{b^2}\chi(A_j)$;
				
				\item there exists $J\subset[\ell]$ with $\abs J\ge y^{-1/8}$ such that for all $j\in J$, there exists $B_j\subset A_j$ satisfying $\chi(B_j)\ge y^{2b^2}\chi(A_j)$ and $(B_j:j\in J)$ is a complete blockade; or
				
				\item there exist $j\in[\ell]$ and a complete blockade in $G[A_j]$ of length at least $y^{-1/8}$ and mass at least $y^{2b^2}\chi(A_j)$.
			\end{itemize}
			
			If the first bullet holds, then since $y^{-2b/3}\ge 8y^{-b/6}=8(y^{2b/3})^{-1/4}$ and
			$$y^{b^2}\chi(A_j)\ge y^{b^2+6b}\chi(G)\ge y^{4b^2}\chi(G)=(y^{2b/3})^{6b}\chi(G),$$
			this would violate the minimality of $y$ because $\chi(G)^{-1/b}\le y^{2b/3}<y$ by our supposition and $b\ge2$.
			
			If the second bullet holds then the lemma holds since $\abs J\ge y^{-1/8}\ge \eps^{-1/16}$ and $$\chi(B_j)\ge y^{2b^2}\chi(A_j)\ge y^{2b^2+6b}\chi(G)\ge \abs J^{-16b^2-48b}\chi(G)\ge \abs J^{-a}\chi(G)\quad\text{for all $j\in J$.}$$
			
			If the third bullet holds then the lemma holds since $y^{-1/8}\ge \eps^{-1/16}$ and 
			$$y^{2b^2}\chi(A_j)\ge y^{2b^2+6b}\chi(G)= y^{a/8}\chi(G).$$
			This proves \cref{claim:2incre}.
		\end{subproof} 
		
		Now, let $p\in[\ell-1]$ be minimal such that $G$ has a clique $\{v_{p+1},\ldots,v_\ell\}$ with $v_i\in A_i$ for all $i\in[\ell]\setminus[p]$; this is clearly possible with $p=\ell-1$.
		If $\ell-p<y^{-1/4}$ then $p>\ell-y^{-1/4}\ge1$; and since $(A_1,\ldots,A_\ell)$ is $(y,\chi)$-dense, there would be $v_p\in A_p$ adjacent to all of $v_{p+1},\ldots,v_\ell$, contrary to the minimality of $p$.
		Hence $\ell-p\ge y^{-1/4}\ge \eps^{-1/16}$. Let $k:=\ell-p$; then \cref{claim:2incre} yields $k^a\ge y^{-a/4}\ge \chi(G)^{a/(8b^2+24b)}=\chi(G)^2\ge \chi(G)$.
		Therefore the lemma holds with $(\{v_{\ell+1-i}\}:i\in[k])$.
		This proves \cref{lem:2incre}.
	\end{proof}
	
	\section{Algorithmic aspects}
	\label{sec:algo}
	In this section we discuss the algorithmic aspects of our proof of \cref{thm:p5}.
	In view of~\cite{MR3548288,MR4174126}, given a class $\mac G$, if $\mac G$ is $\chi$-bounded with a $\chi$-binding function $f\colon \mab N\to \mab R_{\ge0}$ then it would be desirable to construct a deterministic algorithm that, for any input graph $G\in\mac G$, outputs a clique $K$ in $G$ and a proper colouring of $G$ with at most $g(\abs K)$ colours in polynomial time, for some function $g\colon\mab N\to\mab R_{\ge0}$ having similar order of magnitude as $f$ has. We remark that this is unlikely to be achievable already when $\mac G$ is the class of all $3$-colourable graphs, since weaker variants of the Unique Games Conjecture imply the {\sf NP}-hardness of colouring these graphs with constantly many colours~\cite{MR2538841,MR4118243}.
	On the positive side, there are known examples of $\chi$-bounded classes $\mac G$ admitting such an algorithm:
	\begin{itemize}
		\item The `Gy\'arf\'as path' argument as in the proof of \cref{thm:gas} can be turned into a poly-time algorithm that, for any fixed $t\ge4$ and any input graph $G$, outputs either an induced copy of $P_t$, or a clique $K$ and a proper colouring of $G$ with at most $(t-2)^{\abs K-1}$ colours.
		
		\item Scott and Seymour~\cite{MR3548288} derived a poly-time algorithm that, for any input graph $G$, outputs either an induced cycle of odd length at least five, or a clique $K$ and a proper colouring of $G$ with at most $2^{2^{\abs K}}$ colours.
	\end{itemize}
	
	If the aforesaid $\chi$-binding function $f$ is polynomial, the algorithmic problem in question is conveniently divided into the following two subproblems, which also make sense for general classes $\mac G$ that are not necessarily $\chi$-bounded:
	\begin{problem}
		\label{prob:polyomega}
		Do there exist $c=c(\mac G)>0$ and a poly-time algorithm that, for any input graph $G\in \mac G$, outputs a clique of size at least $\omega(G)^c$?
	\end{problem}
	
	\begin{problem}
		\label{prob:polychi}
		Do there exist $d=d(\mac G)\ge1$ and a poly-time algorithm that, for any input graph $G\in\mac G$, outputs a proper colouring with at most $\chi(G)^d$ colours?
	\end{problem}
	
	\cref{prob:polyomega} strengthens the problem of $n^{1-\eps}$-approximating {\sc Maximum Clique} for any $\eps>0$, which is known to be {\sf NP}-hard if $\mac G$ is the class of all graphs~\cite{MR1687331,MR2403018}. If $\mac G$ is defined by forbidden induced subgraphs, this problem first appeared (as polynomially approximating {\sc Maximum Independent Set}) in the works of Bonnet,  Thomass\'e, Tran, and Watrigant~\cite{MR4140931} and of Dvo\v{r}\'ak, Feldmann, Rai, and Rz\c{a}\.{z}ewski~\cite{MR4568165}; for instance, they positively resolved \cref{prob:polyomega} for $\mac G$ defined by excluding a disjoint union of two complete graphs. Several hardness results related to {\sc Maximum Clique} for strict subclasses of $P_5$-free graphs are known:
	\begin{itemize}
		\item By~\cite[Corollary 29]{MR4140931}, if $\mac G$ is the class of graphs $G$ with $\alpha(G)\le 2$, then it remains {\sf NP}-hard to $n^{1/4-\eps}$-approximate {\sc Maximum Clique} on $\mac G$ for any $\eps>0$, assuming ${\sf NP}\ne{\sf BPP}$.
		
		\item By~\cite[Theorem 2]{MR4132895}, if $\mac G$ is the class of $2K_2$-free graphs (here $2K_2$ is the graph obtained from $P_5$ by removing the middle vertex), then {\sc Maximum Clique} remains {\sf W[1]}-hard on $\mac G$, which means the problem is unlikely to be fixed-parameter tractable.
	\end{itemize}
	
	Despite these hardness results, the chromatic density framework developed in this paper can be adapted to resolve \cref{prob:polyomega} in the positive for $P_5$-free graphs. Indeed, we have presented our proof of \cref{thm:p5} step-by-step to show that our argument is fully algorithmic in nature; and as discussed in \cref{subsec:meas}, the argument works for not just $\chi$ but any measure $\mu$. Thus, to ensure that our framework can derive a deterministic poly-time algorithm that polynomially approximates {\sc Maximum Clique} in the sense of \cref{prob:polyomega}, it remains to fulfil the following two objectives.
	
	The first objective is to find an {\em oracle} that approximates $\mu$ from below by a factor of $1-\eps$ in time $\operatorname{poly}(\abs G,\eps^{-1})$, for any input $P_5$-free graph $G$ and any input error $\eps\in\mab Q_{>0}$. Such an approximation factor is needed for the proofs in \cref{sec:highchi,sec:1incre}; to explain, in these steps, for each such $G$ and each induced subgraph $F$ of $G$, one would need to extract a component $C$ of $F$ with $\mu(C)\ge (1-\abs G^{-1})\mu(F)$ in $\operatorname{poly}(\abs G)$ steps. For the other parts of the proof, any constant multiplicative approximation factor sufficiently close to $1$ suffices.
	While these tasks are {\sf NP}-hard for `standard' measures such as $\chi$ and $\omega$~\cite{MR4132895,MR1905637}, choosing $\mu$ to be {\em fractional chromatic number} (which we denote by $\chis$) turns out to be useful. To explain, Lokshtanov, Vatshelle, and Villanger~\cite{MR3376403} proved that {\sc Maximum Weight Independent Set} (MWIS) is poly-time solvable on $P_5$-free graphs, which allows the simple method of Multiplicative Weight Updates (MWU) (see for instance~\cite{MR2948502,MR1342948,MR2318722,young01}) to derive an algorithm that, for any input $P_5$-free graph $G$ and any input error $\eps\in\mab Q_{>0}$, computes a {\em fractional clique} of value at least $(1-\eps)\chis(G)$ in time $\operatorname{poly}(\abs G,\eps^{-1})$; see \cref{thm:mwu} for a proof. We remark that in fact computing $\chis$ {\em exactly} can be reduced to MWIS in polynomially many steps, but the only known reduction so far critically involves the ellipsoid method~\cite{MR625550}, which is not combinatorial in nature.
	
	The second objective is a {\em constructive} version of the \erh{} property of $P_5$, that is, a deterministic poly-time algorithm that outputs a clique or stable set of polynomial size in every input $P_5$-free graph. We remark that developing constructive versions for \erh{} results was previously investigated in~\cite{MR4140931}; however they allowed the corresponding algorithms to be {\em randomised}. In our case of $P_5$-free graphs, it is possible to derandomise the probabilistic steps in~\cite{density7} to turn the proof there into a fully deterministic poly-time algorithm achieving our purpose.
	Here we briefly explain the derandomisation.
	In~\cite{density7}, the proof of the \erh{} property of $P_5$ starts by applying R\"odl's theorem \ref{thm:rodl} for some suitably small $\eps>0$; we remark that the standard proof of \cref{thm:rodl} uses Szemer\'edi's regularity lemma~\cite{MR540024}, which is known to admit a fully constructive version~\cite[Theorem 1.3]{MR1251840} for any fixed regularity parameters.
	The proof in~\cite{density7} then applies the method of iterative sparsification to $\overline{P_5}$-free graphs (recall that $\overline{P_5}$ is the complement of $P_5$), and is deterministic in nature modulo Lemmas 4.2 and 7.1 whose proofs employ the probabilistic method.
	It then suffices to provide constructive proofs of these two (up to numerical adjustments); the rest of~\cite{density7} remains essentially unchanged. We begin with the following constructive form of~\cite[Lemma 4.2]{density7}.
	
	\begin{lemma}
		[{cf.~\cite[Lemma 4.2]{density7}}]
		\label{lem:locate}
		Let $x>0$ be a rational number, and let $G$ be a bipartite graph with bipartition $(A,B)$ where $A,B$ are nonempty and every vertex in $B$ has at least $x\abs A$ neighbours in $A$.
		Then there is a deterministic algorithm that runs in time $\operatorname{poly}(\abs G,x^{-1})$ and outputs $A'\subset A$ such that $\abs {A'}\le2/x$ and $B$ has at least $\frac12\abs B$ vertices with a neighbour in $A'$.
	\end{lemma}
	\begin{proof}
		Let $B_0:=B$; and for $i=1,2,\ldots,\abs A$ in turn, let $u_{i}\in A\setminus\{u_1,\ldots,u_{i-1}\}$ have the largest number of neighbours in $B_{i-1}$, and let $B_i:=B\setminus \left(\bigcup_{j\in[i]}N_G(u_j)\right)$.
		For every such $i$, $G$ has at least $x\abs A\cdot \abs{B_{i-1}}$ edges between $A\setminus\{u_1,\ldots,u_{i-1}\}$ and $B_{i-1}$. So $u_{i}$ has at least $x\abs{B_{i-1}}$ neighbours in~$B_{i-1}$.
		
		Now, let $i\in\{1,2,\ldots,\abs A\}$ be minimal such that $\abs{B_i}\le\frac12\abs B$; then for each $j\in[i]$, we have $\abs{B_{j-1}\cap N_G(u_j)}\ge x\abs{B_{j-1}}\ge \frac12x\abs B\ge \frac12x\abs{B\setminus B_i}$. Since $(B_{j-1}\cap N_G(u_j):j\in[i])$ forms a partition of $B\setminus B_i$, we obtain $i\le 2/x$. It is not hard to see that the whole procedure runs in time $\operatorname{poly}(\abs G,x^{-1})$; and letting $A':=\{u_1,\ldots,u_i\}$ proves \cref{lem:locate}.
	\end{proof}
	
	Given \cref{lem:locate}, one can slightly numerically adjust the argument in~\cite[Sections 5 and 6]{density7} to obtain the following constructive version of~\cite[Lemma 3.4]{density7}; we omit the detailed proof.
	
	\begin{lemma}
		[{cf.~\cite[Lemma 3.4]{density7}}]
		\label{lem:epsone}
		There exists $d\ge40$ for which the following holds.
		Let $\eps\in(0,\frac12)$ be rational, and let $G$ be a $\overline{P_5}$-free graph with $\abs G\ge \eps^{-10d^2}$.
		Then there is a deterministic algorithm that runs in time $\operatorname{poly}(\abs G,\eps^{-1})$ and outputs a blockade $(B_1,\ldots,B_{\ell})$ in $G$ with $\ell\ge \eps^{-1}$ and $\abs{B_i}\ge \eps^{10d^2}\abs G$ for all $i\in[\ell]$, such that for all distinct $i,j\in[\ell]$, either $B_i$ is complete to $B_j$ in $G$, or $G$ has at most $\eps^d\abs{B_i}\abs{B_j}$ edges between $B_i$ and $B_j$.
	\end{lemma}
	
	Now, we state and prove the following constructive form of~\cite[Lemma 7.1]{density7}.
	
	\begin{lemma}
		[{cf.~\cite[Lemma 7.1]{density7}}]
		\label{lem:house6}
		There exists $d\ge40$ such that the following holds.
		Let $y\in(0,\frac14)$ be rational, and let $G$ be a $\overline{P_5}$-free graph with maximum degree at most $y\abs G$.
		Then there is a deterministic algorithm that runs in time $\operatorname{poly}(\abs G,y^{-1})$ and outputs one of the following configurations in $G$:
		\begin{itemize}
			\item an induced subgraph $F$ with $\abs F\ge y^{30d^3}\abs G$ and maximum degree at most $y^{2d}\abs F$;
			
			\item a complete blockade $(B_1,\ldots,B_\ell)$ with $\ell\ge y^{-1}$ and $\abs{B_i}\ge y^{33d^3}\abs G$ for all $i\in[\ell]$; or
			
			\item an anticomplete pair $(X,Y)$ with $\abs X\ge y^{33d^3}\abs G$ and $\abs Y\ge(1-3y)\abs G$.
		\end{itemize}
	\end{lemma}
	
	\begin{proof}
		In what follows, we will give the existential proof using deterministic arguments (in particular, dyadic partitioning); the desired poly-time algorithm then follows straightforwardly.
		
		For $x>0$ and disjoint $A,B\subset V(G)$, we say that $A$ is {\em $x$-sparse} to $B$ in $G$ if every vertex in $A$ has at most $x\abs B$ neighbours in $B$. Also, we say that an induced subgraph of $G$ is {\em anticonnected} if its complement is connected.
		
		We claim that $d\ge40$ given by \cref{lem:epsone} satisfies the lemma.
		To see this, assume that the first two outcomes do not hold.
		In particular $\abs G\ge y^{-30d^3}$ since the first outcome does not hold.
		Let $\eps:=y^{3d}\in (0,2^{-3d})$;
		then $\abs G\ge y^{-30d^3}= \eps^{-10d^2}$.
		Let $\ell:=\ceil{\eps^{-1}}\le 2\eps^{-1}$.
		
		\begin{claim}
			\label{claim:clean}
			There is a blockade $(B_1,\ldots,B_{\ell})$ in $G$ such that:
			\begin{itemize}
				\item for all $i\in[\ell]$, $G[B_i]$ is anticonnected and $\abs{B_i}\ge\frac14y\eps^{10d^2}\abs G$;
				
				\item for all distinct $i,j\in[\ell]$, $(B_i,B_j)$ is either complete or $\eps^{d-6}$-sparse to each other in $G$; and
				
				\item $\abs{B_1\cup\cdots\cup B_\ell}\le y\abs G$.
			\end{itemize}
		\end{claim}
		\begin{subproof}
			By \cref{lem:epsone}, there is a blockade $(A_1,\ldots,A_{\ell})$ in $G$ with $\abs{A_i}\ge \eps^{10d^2}\abs G$ for all  $i\in[\ell]$, such that for all distinct $i,j\in[\ell]$, either $A_i$ is complete to $A_j$ in $G$, or $G$ has at most $\eps^d\abs{A_i}\abs{A_j}$ edges between $A_i$ and $A_j$.
			Let $J$ be the~graph with vertex set $[\ell]$ where distinct $i,j\in V(J)$ are adjacent in $J$ if and only if $A_i$ is complete to $A_j$ in $G$.
			
			For $i=1,2,\ldots,\ell$ in turn, define $B_i\subset A_i$ as follows.
			Assume that $B_1,\ldots,B_{i-1}$ have been defined, such that $\abs{B_p}=\ceil{\frac14y\abs{A_p}}$ for all $p<i$; and for all $1\le p<q\le\ell$ with $p\notin N_J(q)$ and $p<i$, we~have
			\begin{itemize}
				\item if $q<i$, then $B_p$ is $\eps^{d-4}$-sparse to $B_q$ and $B_q$ is $\eps^{d-6}$-sparse to $B_p$; and
				
				\item if $q\ge i$, then
				$B_p$ is $\eps^{d-2}$-sparse to $A_q$.
			\end{itemize}
			For each $p\in [\ell]\setminus N_J[i]$, let $C_p$ be the set of vertices in $A_i$ with at least $\eps^{d-4}\abs{B_p}$ neighbours in $B_p$ if $p<i$,
			and let $C_p$ be the set of vertices in $A_i$ with at least $\eps^{d-2}\abs{A_p}$ neighbours in $A_p$ if $p>i$;
			then $\abs{C_p}\le \eps^2\abs{A_i}$ for all $p\in[\ell]\setminus\{i\}$.
			Let $D_i:=A_i\setminus\left(\bigcup_{p\in[\ell]\setminus N_J[i]}C_p\right)$;
			then $\abs{D_i}\ge (1-\eps^2\ell)\abs{A_i}\ge (1-2\eps)\abs{A_i}\ge\frac12\abs{A_i}$.
			Let $k:=\ceil{y^{-1}}\le 2y^{-1}$.
			By~\cite[Lemma 4.1]{density7} (this has a simple constructive proof), either:
			\begin{itemize}
				\item there exists $B_i\subset D_i$ such that $G[B_i]$ is anticonnected and $\abs{B_i}\ge\abs{D_i}/k\ge \frac12 y\abs{D_i}\ge\frac14y\abs{A_i}$; or
				
				\item there is a complete blockade $(E_1,\ldots,E_k)$ in $G[D_i]$ with $\abs{E_i}\ge\abs{D_i}/k^2$ for all $i\in[k]$.
			\end{itemize}
			
			If the second bullet holds then so does the second outcome of the lemma since $k\ge y^{-1}$ and $\abs{D_i}/k^2\ge \frac18y^2\abs{A_i}\ge \frac18y^2\eps^{10d^2}\abs G\ge \eps^{11d^2}\abs G=y^{33d^3}\abs G$, a contradiction.
			Thus the first bullet holds; and by removing vertices from $B_i$ if necessary, we may assume that $\abs{B_i}=\ceil{\frac14y\abs{A_i}}$.
			For every non-neighbour $p$ of $i$ in $J$ with $p<i$, since $B_p$ is $\eps^{d-2}$-sparse to $A_i$, we have that $B_p$ is $\eps^{d-4}$-sparse to $B_i$; and $B_i$ is $\eps^{d-6}$-sparse to $B_p$ by definition.
			This completes the inductive definition of $B_1,\ldots,B_{\ell}$.
			
			Now, we have $\frac14y\eps^{10d^2}\abs G\le \abs{B_i}\le \frac12y\abs{A_i}+1$ for all $i\in[\ell]$; and
			$\abs{B_1\cup\cdots\cup B_\ell}\le 
			\frac12y(\abs{A_1}+\cdots+\abs{A_\ell})+\ell\le y\abs G$ since $\abs G\ge \eps^{-10d^2}\ge 4y^{-1}\eps^{-1}\ge 2y^{-1}\ell$.
			This proves \cref{claim:clean}.
		\end{subproof}
		
		Let $B:=V(G)\setminus(B_1\cup\cdots\cup B_{\ell})$; then $\abs B\ge (1-y)\abs G$. We claim that:
		
		\begin{claim}
			\label{claim:mixed}
			Every vertex in $B$ is mixed on fewer than $y\ell$ blocks among $B_1,\ldots,B_{\ell}$.
		\end{claim}
		
		\begin{subproof}
			Suppose that there exists $v\in B$ mixed on at least $y\ell$ blocks among $B_1,\ldots,B_\ell$;
			and assume that it is mixed on $B_1,\ldots,B_r$, where $r\ge y\ell\ge y\eps^{-1}$.
			If there are distinct $i,j\in[r]$ for which $B_i$ is complete to $B_j$ in $G$,
			then since $G[B_i]$ and $G[B_j]$ are anticonnected,
			there would be nonadjacent pairs $u_i,w_i\in B_i$ and $u_j,w_j\in B_j$ such that $u_i,u_j\in N_G(v)$ and $w_i,w_j\nin N_G(v)$;
			but then $\{v,u_i,u_j,w_i,w_j\}$ would form a copy of $\overline{P_5}$ in $G$, a contradiction.
			Thus, $B_i$ is $\eps^{d-6}$-sparse to $B_j$ for all distinct $i,j\in[r]$.
			Since $\ceil{\frac14y\eps^{10d^2}\abs G}\le \abs{B_i}\le y\abs G$ for all $i\in[\ell]$, dyadic partitioning gives some  $I\subset[r]$ with
			$$\abs I\ge \frac r{2\log(\eps^{-10d^2})}\ge\frac{y\eps^{-1}}{20d^2\log(\eps^{-1})}=\frac{y^{1-3d}}{60d^3\log (y^{-1})}\ge y^{-2-2d}\ge 4y^{-1} $$
			(note that $d\ge 40$ and $y\in(0,\frac12)$) such that $\abs{B_i}\le 2\abs{B_j}$ for all $i,j\in I$.
			Let $m:=\min_{i\in I}\abs{B_i}$; then $\max_{i\in I}\abs{B_i}\le 2m$.
			Let $S:=\bigcup_{i\in I}B_i$; then $\abs S\ge\abs I m$ and $G[S]$ has maximum degree at most
			\[2m+\abs I\cdot\eps^{d-6}\cdot2m\le (2y^{2d+2}+2\eps^{d-6})\abs I m
			\le 2y^{2d+1}\abs Im\le y^{2d}\abs Im\le y^{2d}\abs S\]
			where the second inequality holds since $\eps^{d-6}=y^{3d(d-6)}\le y^{3d}\le\frac12 y^{2d+1}$ (note that $d\ge40$).
			Then $F:=G[S]$ satisfies the first outcome of the lemma since $\abs S\ge \abs I m\ge 4y^{-1}\cdot \frac14y\eps^{10d^2}\abs G=y^{30d^3}\abs G$, a contradiction.
			This proves \cref{claim:mixed}.
		\end{subproof}
		
		\cref{claim:mixed} gives some $i\in[\ell]$ such that $B$ has fewer than $y\abs B$ vertices mixed on $B_i$.
		Thus, since $G$ has maximum degree at most $y\abs G$, $B$ has at most $y\abs G+y\abs B$ vertices with a neighbour in $B_i$.
		Let $Y$ be the set of vertices in $B$ with no neighbour in $B_i$; then, because $\abs B\ge (1-y)\abs G$, we have
		\[\abs Y\ge (1-y)\abs B-y\abs G\ge (1-y)^2\abs G-y\abs G
		\ge (1-3y)\abs G\]
		and the third outcome of the lemma holds since $\abs{B_i}\ge \frac12y\eps^{10d^2}\abs G\ge \eps^{11d^2}\abs G=y^{33d^3}\abs G$.
		This proves \cref{lem:house6}.
	\end{proof}
	
	Given \cref{lem:house6}, it remains to turn the rest of~\cite[Section 7]{density7} into a constructive version of the \erh{} property of $P_5$. This is not hard and we omit the details.
	
	Provided the above discussion, we deduce the following algorithmic version of \cref{thm:main} for fractional chromatic number.
	\begin{theorem}
		\label{thm:mainaglo}
		There exist $d\ge2$ and a deterministic algorithm that, for any input $P_5$-free graph $G$ with at least one edge, outputs in time $\abs G^d$ one of the following:
		\begin{itemize}
			\item a complete pair $(X,Y)$ with $\chis(X)\ge y^d\chis(G)$ and $\chis(Y)\ge (1-y)\chis(G)$, for some rational number $y\in(0,\frac14)$; and
			
			\item a complete blockade $(B_1,\ldots,B_k)$ with $k\ge 2$ and $\chis(B_i)\ge k^{-d}\chis(G)$ for all $i\in[k]$.
		\end{itemize}
	\end{theorem}
	
	Via this result, it is not hard to turn the deduction of \cref{thm:p5} from \cref{thm:main} into the following.
	
	\begin{theorem}
		\label{thm:algop5}
		There exist $d\ge2$ and a deterministic algorithm that, for any input $P_5$-free graph $G$, outputs a clique of size at least $\chis(G)^{1/d}\ge\omega(G)^{1/d}$ in time $\abs G^d$.
	\end{theorem}
	In other words, we have described how to adapt (fractional) chromatic density to resolve \cref{prob:polyomega} when $\mac G$ is the class of $P_5$-free graphs.
	Thus, given that MWIS is known to be poly-time solvable on $P_6$-free graphs~\cite{MR4374260} and quasi-polynomial-time solvable on $P_t$-free graphs for all $t\ge7$~\cite{MR4232071,MR4537965}, we hope the chromatic density framework will be useful in resolving \cref{prob:polyomega} for $P_6$-free graphs and in obtaining a similar quasi-polynomial time algorithm for general $P_t$-free graphs.
	
	Even though \cref{thm:algop5,thm:p5} yield a poly-time algorithm computing a clique of size polynomial in the chromatic number on $P_5$-free graphs, our method in this paper does not seem to give a positive answer to \cref{prob:polychi} for these graphs. It would be interesting to decide this.
	
	\section{Concluding remarks}
	\label{sec:concl}
	There are some final points we would like to make.
	Firstly, even though we made no attempt to optimise the exponent $d$ of \cref{thm:p5} in this paper, our method appears unlikely to result in a quadratic bound in $\omega(G)$, which was conjectured by Choudum, Karthick, and Shalu~\cite{MR2292667}.
	If true, this would be optimal up to a logarithmic factor, since there are $n$-vertex graphs with no three-vertex stable set and with clique number $O(\sqrt{n\log n})$~\cite{MR1369063}.
	
	Secondly, the proof of \cref{thm:p5} can be extended to a more general type of trees. For integers $k,t\ge1$, the {\em $(k,t)$-broom} is the graph obtained from the $(k+1)$-vertex path by substituting a stable set of size $t$ into one of the two endpoints. Then every $(k+1)$-vertex path is the $(k,1)$-broom. We remark that the proof of \cref{thm:p5} can be adapted to show that the $(4,t)$-broom is poly-$\chi$-bounding for all $t\ge1$. To see this, for a $(4,t)$-broom-free graph $G$, one could shrink down each connected induced subgraph of $G$ with chromatic number at least $\chi(G)/\operatorname{poly}(\omega(G))$ (wherever this occurs in the proof) a little more to a connected subgraph with roughly the same chromatic number and with minimum degree at least $\omega(G)^t$ (to create a $t$-vertex stable set). Also, the completeness of pairs and blockades throughout this paper would be replaced by $(\omega(G)^{-2},\chi)$-denseness. One can now use \cite[Theorem 13.3]{2025thes} to deduce that disjoint unions of copies of the $(4,t)$-broom are also poly-$\chi$-bounding. We omit the details, which are just technical adjustments of the presented~materials.
	
	Lastly, the chromatic density framework introduced in this paper not only yields a resolution to the polynomial $\chi$-boundedness problem for excluding $P_5$, but also raises a number of open questions, many of which were already discussed in \cref{sec:sketch}. Let us conclude the paper with one we have not mentioned.
	As emphasised in \cref{subsec:chirdl}, the main notion in this paper is that of $\chi$-dense graphs. We extend this to hereditary graph classes as follows: we say that such a class $\mac G$ is {\em $\chi$-dense} if for every $\eps>0$, there exists $\delta=\delta(\mac G,\eps)>0$ such that every graph $G\in\mac G$ has an $(\eps,\chi)$-dense induced subgraph with chromatic number at least $\delta\cdot\chi(G)$.
	Observe that for every $(\frac13,\chi)$-dense graph $F$ with $\chi(F)\ge2$, every $v\in V(F)$ satisfies $\chi(N_F(v))\ge \frac13\chi(F)$.
	Thus, for any $\chi$-dense class $\mac G$, there exists $\delta>0$ such that every $G\in\mac G$ with $\chi(G)\ge2$ contains some $v\in V(G)$ with $\chi(N_G(v))\ge\delta\cdot \chi(G)$; and so $\mac G$ admits an exponential $\chi$-binding function $x\mapsto \delta^{1-x}$ via induction on the clique number.
	Hence, by the constructions in~\cite{MR4707561} that there are $\chi$-bounded classes with optimal $\chi$-binding function of arbitrary growth, there are $\chi$-bounded classes that are not $\chi$-dense.
	It would be interesting to know if every $\chi$-bounded class of graphs excluding a given tree is actually $\chi$-dense; as remarked in \cref{subsec:chirdl}, we do not know if this is true for $P_5$-free graphs.
	
	Another research direction related to $\chi$-dense classes is as follows. Because every complete graph is $(\eps,\chi)$-dense for all $\eps>0$, every $\chi$-bounded class admitting a linear $\chi$-binding function (such classes are also known as {\em linearly $\chi$-bounded} classes in the literature) is $\chi$-dense.
	On the other hand, it is not hard to check that in every $(\frac13,\chi)$-dense graph $G$, every two nonadjacent vertices (if they exist) have common neighbourhood with chromatic number at least $\frac13\chi(G)$;
	and so if $G$ also has no induced four-cycle $C_4$ then $\omega(G)\ge\frac13\chi(G)$.
	Hence, all $\chi$-dense classes not containing $C_4$ are linearly $\chi$-bounded.
	One may wonder if this still holds for $\chi$-bounded classes without $C_4$:
	
	\begin{conjecture}
		\label{conj:c4}
		Every $\chi$-bounded hereditary class $\mac G$ with $C_4\not\in\mac G$ is linearly $\chi$-bounded.
	\end{conjecture}
	This would confirm, in a strong form, a conjecture of Chudnovsky, Cook, Davies, and Oum~\cite{MR4951164} that such classes $\mac G$ are polynomially $\chi$-bounded.
	See~\cite{2026ptc4} for results and problems related to \cref{conj:c4}.

	\appendix
	\section{$\chi$-dense graphs versus graphs of large minimum degree}
	Here we provide examples of $\chi$-dense graphs with low minimum degree, and graphs with high minimum degree but very far from being $\chi$-dense, as mentioned in \cref{subsec:chirdl}.
	\begin{proposition}
		\label{prop:dense}
		For every $\eps\in(0,1)$, there exists an $(\eps,\chi)$-dense graph $F$ with $\chi(F)\ge\eps^{-1}$ and minimum degree at most $\eps\abs F$,
		and there exists a graph $G$ with $\chi(G)\ge\eps^{-1}$ and minimum degree at least $(1-\eps)\abs G$ but not $(1-\eps,\chi)$-dense.
	\end{proposition}
	\begin{proof}
		To show the first statement, let $F$ be a graph with a partition $V(F)=X\cup Y$ where $X$ is a clique of size at least $\eps^{-1}$ and $Y$ is a stable set of size at least $\eps^{-1}\abs X$, such that $X$ is complete to $Y$.
		Then $\chi(F)=\abs X+1> \eps^{-1}$ and $F\setminus N_F[v]$ is stable for all $v\in V(F)$; thus $F$ is $(\eps,\chi)$-dense. Moreover $F$ has minimum degree $\abs X\le \eps\abs Y\le\eps\abs F$.
		
		To prove the second statement we need the following claim.
		\begin{claim}
			\label{claim:tri}
			For every $\theta\in(0,\frac13)$, there exists $n_0=n_0(\theta)\ge1$ such that for each $n\ge n_0$, there is an $n$-vertex triangle-free graph $G$ with $\frac{n^{\theta}}{6\ln n}\le \chi(G)\le \frac{n^{2\theta}}{\ln n}$.
		\end{claim}
		\begin{subproof}
			Consider a random graph on $n$ vertices where each edge is included independently with probability $p=n^{\theta-1}$.
			As shown by \L uczak~\cite{MR1112273}, there exists $n_1=n_1(\theta)\ge1$ such that whenever $n\ge n_1$, with probability at least $2/3$ this random graph has chromatic number at most
			\[\frac{pn}{\ln(pn)}=\frac{n^{\theta}}{\theta\ln n}\le \frac{n^{2\theta}}{\ln n}.\]
			
			On the other hand, as shown in~\cite[p. 43--44]{MR3524748} (with $\ell=3$), there exists $n_2=n_2(\theta)\ge1$ such that whenever $n\ge n_2$, with probability at least $2/3$ this random graph also has a triangle-free spanning subgraph with chromatic number at least $\frac{n^{\theta}}{6\ln n}$.
			Taking $n_0:=\max(n_1,n_2)$ proves \cref{claim:tri}.
		\end{subproof}
		
		Now, let $k:=\ceil{\eps^{-1}}$.
		Let $n\ge \max(n_0(1/4),n_0(1/9))$ be an integer satisfying $\frac{n^{1/9}}{6\ln n}\ge 4/\eps$ and $(\eps/6)\cdot n^{1/4}>(k-1)n^{2/9}$.
		By \cref{claim:tri}, there exist triangle-free $n$-vertex graphs $G_1,G_2$ such that $\frac{n^{1/4}}{6\ln n}\le\chi(G_1)$ and $\frac{n^{1/9}}{6\ln n}\le\chi(G_2)\le\frac{n^{2/9}}{\ln n}$.
		Let $G$ be the complete join of one copy of $G_1$ and $k-1$ copies of $G_2$. 
		Then $G$ has minimum degree at least $\abs G-n=(1-1/k)\abs G\ge (1-\eps)\abs G$.
		
		Moreover, the choice of $n$ yields $\chi(G)=\chi(G_1)+(k-1)\chi(G_2)\ge k\cdot\frac{n^{1/9}}{6\ln n}\ge 4k/\eps\ge 4\eps^{-2}$, and
		\begin{align*}
			\chi(G)\le \chi(G_1)+(k-1)n^{2/9}/\ln n
			<\chi(G_1)+(\eps/6)n^{1/4}/\ln n
			\le (1+\eps)\cdot\chi(G_1).
		\end{align*}
		Thus, since $G_1$ is triangle-free, every $v\in V(G_1)$ satisfies
		\[\chi(G\setminus N_G[v])\ge\chi(G_1\setminus N_{G_1}[v])\ge \chi(G_1)-2 \ge \frac1{1+\eps}\chi(G)-2\ge (1-\eps)\chi(G)\]
		where the last inequality holds because $$\frac1{1+\eps}\chi(G)-(1-\eps)\chi(G)=\frac{\eps^2}{1+\eps}\chi(G)\ge (\eps^2/2)\cdot \chi(G)\ge2.$$
		Therefore $G$ is not $(1-\eps,\chi)$-dense. This proves \cref{prop:dense}.
	\end{proof}
	
	\section{Approximating fractional chromatic number via MWIS and MWU}
	\label{sec:mwu}
	
	In this appendix we provide a simple procedure based on Multiplicative Weights Update (MWU) that reduces approximating fractional chromatic number $\chis$ to {\sc Maximum Weight Independent Set} (MWIS) in polynomially many steps. We believe this is well-known but have been unable to locate an appropriate reference in the literature.
	
	Let us recall relevant definitions. For a graph $G$, its {\em fractional chromatic number} $\chis(G)$ is defined as follows: if $\abs G=0$ then $\chis(G)=0$; and if $\abs G\ge1$ then $\chis(G)$ is the minimum $k\ge1$ such that there is a probability distribution $\mathbf{p}$ on the collection $\mac I_G$ of stable sets of $G$ satisfying $\sum_{S\ni v}\mathbf{p}(S)\ge\frac1k$ for all $v\in V(G)$. A {\em fractional clique} of $G$ is a function $f\colon V(G)\to\mab R_{\ge0}$ satisfying $\sum_{v\in S}f(v)\le 1$ for all $S\in\mac I_G$; and the {\em value} of $f$ is $\sum_{v\in V(G)}f(v)$. The {\em fractional clique number} $\omega^*(G)$ of $G$ is the maximum $\ell\ge0$ such that there is a fractional clique of $G$ of value at least $\ell$. By linear programming duality, $\chis(G)=\omega^*(G)$ and this is always a rational number. The MWIS problem asks, for any weighting $w\colon V(G)\to\mab Q_{\ge0}$, to return $S\in\mac I_G$ such that $\sum_{v\in S}w(v)$ is maximal.
	
	Our objective is to use MWIS as an oracle to compute a fractional clique of $G$ of value at least $(1-\eps)\chis(G)$, for any input error $\eps\in\mab Q_{>0}$.
	Our procedure below follows closely the standard MWU setup illustrated by Arora, Hazan, and Kale~\cite{MR2948502}:
	\begin{itemize}
		\item {\bf Initialisation:} Fix a rational $\eta\in(0,\frac12]$. For each $u\in V(G)$, associate the weight $w_1(u):=1$.
		
		\item {\bf For each $i\ge1$, do the following:}
		\begin{enumerate}
			\item call MWIS on $G$ with $w_i$ to extract $S_i\in \mac I_G$ such that $\sum_{v\in S_i}w_i(v)$ is maximum; then
			
			\item update $w_{i+1}(u):=(1-\eta)\cdot w_i(u)$ for all $u\in S_i$ and $w_{i+1}(u):=w_i(u)$ for all $u\in V(G)\setminus S_i$.
		\end{enumerate}
	\end{itemize}
	(Here, each vertex is a `decision', and the `cost vector' at step $i$ is the indicator function on the maximum weight stable set $S_i$ returned by MWIS.)
	
	We analyse the above procedure via the following lemma.
	
	\begin{lemma}
		\label{lem:mwu}
		Let $k\in\mab Q_{>0}$, and let $G$ be a graph with $\abs G\ge2$. If $k\le \chis(G)$, then there is a positive integer $t\le \ceil{k\eta^{-2}\ln \abs G}$ such that
		$$\sum_{v\in S_t}w_t(v)\le\frac{1+2\eta}{k}\sum_{v\in V(G)}w_t(v).$$
		In other words, the function $f\colon V(G)\to\mab R_{\ge0}$ defined by
		$$f(u):=\frac{w_t(u)}{\sum_{v\in S_t}w_t(v)}$$
		is a fractional clique of $G$ of value at least $\frac1{1+2\eta}k\ge (1-2\eta)k$.
	\end{lemma}
	\begin{proof}
		Suppose not. Let $T:=\ceil{k\eta^{-2}\ln\abs G}$. For each $u\in V(G)$ and $i\in[T]$, let $m_i(u)$ be $1$ if $u\in S_i$ and $0$ if $u\nin S_i$, and let $q(u):=\sum_{i=1}^Tm_i(u)$. Also, let $\Phi_i:=\sum_{v\in V(G)}w_i(v)$ for every $i\in[T]$. Then by~\cite[Theorem 2.1]{MR2948502}, every $u\in V(G)$ satisfies
		\[\sum_{i=1}^T\sum_{v\in V(G)}\frac{m_i(v)\cdot w_i(v)}{\Phi_i}
		\le \sum_{i=1}^Tm_i(u)+\eta\sum_{i=1}^T\abs{m_i(u)}+\frac{\ln \abs G}{\eta}
		=(1+\eta)\cdot q(u)+\frac{\ln \abs G}{\eta}.\]
		Moreover, for all $i\in[T]$, our supposition implies that
		$$\sum_{v\in V(G)}m_i(v)\cdot w_i(v)=\sum_{v\in S_i}w_i(v)>\frac{1+2\eta}{k}\Phi_i,$$
		and so
		\begin{equation*}
			\frac{1+2\eta}{k}T<(1+\eta)\cdot q(u)+\frac{\ln \abs G}{\eta}\quad\text{for all $u\in V(G)$.}
		\end{equation*}
		
		Now, since $k\le\chis(G)=\omega^*(G)$, there is a fractional clique $g$ of $G$ of value at least $k$. Therefore, letting $\phi:=\sum_{u\in V(G)}g(u)$, we obtain
		\[\frac{1+2\eta}k T
		<\sum_{u\in V(G)}\frac{g(u)}{\phi}\left((1+\eta)\cdot q(u)+\frac{\ln\abs G}{\eta}\right)
		=\frac{1+\eta}{\phi}\sum_{u\in V(G)}{g(u)\cdot q(u)}+\frac{\ln \abs G}\eta.\]
		Thus, because $\sum_{u\in V(G)}g(u)\cdot q(u)=\sum_{i=1}^T\sum_{u\in S_i}g(u)\le T$ and $\phi\ge k$, we deduce that
		\[\frac{1+2\eta}{k}T<\frac{1+\eta}{k}T+\frac{\ln\abs G}{\eta},\]
		which yields $T<k\eta^{-2}\ln\abs G$, a contradiction. This proves \cref{lem:mwu}.
	\end{proof}
	
	The analysis of \cref{lem:mwu} yields the following approximation algorithm.
	
	\begin{lemma}
		\label{lem:mwu1}
		Let $k,\eta\in\mab Q_{>0}$ with $\eta\le\frac12$, and let $G$ be a graph. Then there is an algorithm that either outputs a fractional clique of $G$ of value at least $(1-2\eta)k$ or correctly concludes that $k>\chis(G)$. The algorithm makes at most $\ceil{k\eta^{-2}\ln\abs G}$ calls to MWIS on $G$, with an additional processing time of $O(\abs G)$ per call.
	\end{lemma}
	\begin{proof}
		We may assume that $\abs G\ge2$. Our algorithm is as follows:
		\begin{itemize}
			\item For each $u\in V(G)$, associate the weight $w_1(u):=1$.
			
			\item {For $i=1,2,\ldots,\ceil{k\eta^{-2}\ln\abs G}$ in turn, do the following:}
			\begin{enumerate}
				\item call MWIS on $G$ with $w_i$ to extract $S_i\in \mac I_G$ such that $s_i:=\sum_{v\in S_i}w_i(v)$ is maximum; then
				
				\item if $s_i\le\frac{1+2\eta}{k}\sum_{v\in V(G)}w_i(v)$ then stop and output the fractional clique $f\colon V(G)\to\mab R_{\ge0}$ defined by $f(u):=w_i(u)/s_i$ for all $u\in V(G)$;
				and if $s_i>\frac{1+2\eta}k\sum_{v\in V(G)}w_i(v)$ then update $w_{i+1}(u):=(1-\eta)\cdot w_i(u)$ for all $u\in S_i$ and $w_{i+1}(u):=w_i(u)$ for all $u\in V(G)\setminus S_i$.
			\end{enumerate}
			
			\item If the algorithm does not stop midway during the above loop then conclude that $k>\chis(G)$.
		\end{itemize}
		
		\cref{lem:mwu} shows that this algorithm satisfies the conclusion of \cref{lem:mwu1}; note that the additional processing time of $O(\abs G)$ per call comes from the weight updates.
	\end{proof}
	
	We now apply a simple linear search to approximate $\chis$ as follows.
	
	\begin{theorem}
		\label{thm:mwu}
		Let $\eps\in(0,1]$ be a rational number, and let $G$ be a graph. Then there is an algorithm that outputs a fractional clique in $G$ of value at least $(1-\eps)\chis(G)$ and makes at most $O(\eps^{-3}\abs G^2\ln\abs G)$ calls to MWIS on $G$, with an additional processing time of $O(\abs G)$ per call.
	\end{theorem}
	\begin{proof}
		We may assume that $\abs G\ge1$ and $G$ is non-complete; and so $\abs G>\chis(G)\ge1$.
		Let $\ell:=\ceil{2\eps^{-1}\abs G}$, and let $\eta:=\eps/4$. Note that $\frac12\eps\ell\ge\abs G$.
		
		Our algorithm is as follows. For $j=1,2,\ldots,\ell$ in turn, we run the algorithm given by \cref{lem:mwu1} with $k=\frac12\eps j$; if this outputs a fractional clique of $G$ of value at least $(1-2\eta)k$ then retain such a fractional clique and proceed to $j+1$, and if it concludes that $k>\chis(G)$ then stop. The total number of calls to MWIS on $G$ is at most 
		$$\sum_{j=1}^{\ell}\ceil{(\eps/2)\cdot  j\cdot\eta^{-2}\ln\abs G}=O(\eps\ell^2\cdot\eps^{-2}\ln\abs G)=O(\eps^{-3}\abs G^2\ln\abs G)$$
		with an additional processing time of $O(\abs G)$ per call.
		
		Now, since $\frac12\eps\ell\ge\abs G>\chis(G)$, either the above algorithm stops at some step $j$ in the loop and outputs a fractional clique of $G$ of value at least $(1-2\eta)\cdot\frac12\eps\cdot(j-1)$, or it outputs a fractional clique of $G$ of value at least $(1-2\eta)\cdot\frac12\eps\cdot \ell$. Hence there exists $j\in[\ell]$ such that $\frac12\eps j>\chis(G)$ and the algorithm outputs a fractional clique of $G$ of value at least
		\[(1-2\eta)\cdot (\eps/2)\cdot (j-1)
		>(1-2\eta)\cdot (\chis(G)-\eps/2)
		=(1-\eps/2)(\chis(G)-\eps/2)
		\ge (1-\eps)\chis(G)\]
		where we used that $\eta=\eps/4$ and $\chis(G)\ge1$. This proves \cref{thm:mwu}.
	\end{proof}
	
	\setstretch{1}
	\bibliographystyle{abbrv}
	\bibliography{pc5}
	
\end{document}